\documentclass[bj,numbers, noshowframe]{imsart}

\RequirePackage[numbers]{natbib}
\usepackage{amsthm, amsmath, amsfonts, amssymb}
\usepackage[T1]{fontenc} 
\usepackage[utf8]{inputenc}
\usepackage{mathrsfs}
\usepackage{wrapfig}
\usepackage{bbold}
\usepackage{bbm}
\usepackage{dsfont}
\usepackage{hyperref}
\usepackage{csquotes}
\usepackage{relsize}
\usepackage[normalem]{ulem}
\usepackage{mathtools}
\usepackage{float}
\usepackage{commath}
\usepackage{empheq}
\usepackage[toc,page]{appendix}
\usepackage{subfig}
\usepackage{titlesec}

\DeclareMathOperator*{\argmax}{arg\,max}
\DeclareMathOperator*{\argmin}{arg\,min}

\startlocaldefs

\newcommand{\eps}{\varepsilon}

\newcommand{\wt}{\widetilde}

\newcommand{\ReLU}{\operatorname{ReLU}}
\newcommand{\Var}{\operatorname{Var}}

\newcommand{\E}{\operatorname{E}}

\newcommand{\Q}{\operatorname{Q}}
\renewcommand{\P}{\operatorname{P}}

\newcommand{\supp}{\operatorname{supp}}

\newcommand{\Lip}{\operatorname{Lip}}

\newcommand{\wh}{\widehat}
\newcommand{\ul}{\underline}
\newcommand{\ol}{\overline}

\newcommand{\KL}{\operatorname{KL}}

\newcommand{\mB}{\mathcal{B}}

\newcommand{\mD}{\mathcal{D}}
\newcommand{\mE}{\mathcal{E}}
\newcommand{\mF}{\mathcal{F}}
\newcommand{\mG}{\mathcal{G}}

\newcommand{\mM}{\mathcal{M}}
\newcommand{\mN}{\mathcal{N}}

\newcommand{\mP}{\mathcal{P}}

\newcommand{\mS}{\mathcal{S}}

\newcommand{\sY}{\mathrm{Y}}
\newcommand{\sX}{\mathrm{X}}

\newcommand{\SP}{\, \P}
\newcommand{\SQ}{\, \Q}
\newcommand{\sn}{\, n}
\newcommand{\sm}{\, m}

\allowdisplaybreaks

\newcommand{\RR}{\mathbb{R}}

\newcommand{\dd }{\mathrm{d}}

\newtheorem{theorem}{Theorem}
\newtheorem{lemma}[theorem]{Lemma}

\newtheorem{corollary}[theorem]{Corollary}
\newtheorem{definition}[theorem]{Definition}

\makeatletter
\@addtoreset{proofpart}{theorem}
\makeatother

\theoremstyle{definition}
\newtheorem{example}{Example}

\newtheorem{remark}{Remark}

\endlocaldefs

\begin{document}

\begin{frontmatter}

\title{Local convergence rates of the nonparametric least squares estimator with applications to transfer learning}
\runtitle{Local Convergence Rates}

\begin{aug}

\author[A]{\fnms{Johannes}~\snm{Schmidt-Hieber}\ead[label=e1]{a.j.schmidt-hieber@utwente.nl}}
\author[A]{\fnms{Petr}~\snm{Zamolodtchikov}\ead[label=e2]{p.zamolodtchikov@utwente.nl}}

\address[A]{Department of Applied Mathematics,
University of Twente\printead[presep={,\ }]{e1,e2}}

\end{aug}

\begin{abstract} Convergence properties of empirical risk minimizers can be conveniently expressed in terms of the associated population risk. To derive bounds for the performance of the estimator under covariate shift, however, pointwise convergence rates are required. Under weak assumptions on the design distribution, it is shown that least squares estimators (LSE) over $1$-Lipschitz functions are also minimax rate optimal with respect to a weighted uniform norm, where the weighting accounts in a natural way for the non-uniformity of the design distribution. This implies that although least squares is a global criterion, the LSE adapts locally to the size of the design density. We develop a new indirect proof technique that establishes the local convergence behavior based on a carefully chosen local perturbation of the LSE. The obtained local rates are then applied to analyze the LSE for transfer learning under covariate shift.
\end{abstract}



\end{frontmatter}

\section{Introduction}
\label{sec.intro}

Consider the nonparametric regression model with random design supported on $[0,1],$ that is, we observe $n$ i.i.d.\ pairs $(X_1,Y_1), \dots, (X_n,Y_n) \in [0,1]\times \mathbb{R},$ with 
\begin{align}
    Y_i = f_0(X_i)+ \eps_i, \quad i=1,\dots,n
    \label{eq.mod}
\end{align}
and independent measurement noise variables $\eps_1,\dots,\eps_n \sim \mN(0,1)$. The design distribution is the marginal distribution of $X_1$ and is denoted by $\P_\sX\!.$ Throughout this paper, we assume that $\P_\sX$ has a Lebesgue density $p.$ The least squares estimator (LSE) for the nonparametric regression function $f$ taken over a function class $\mF$ is given by 
\begin{align*}
    \wh f_n \in \argmin_{f\in \mF}
    \sum_{i=1}^n \big(Y_i-f(X_i)\big)^2.
\end{align*}
In the case of non-uniqueness, the subsequent discussion and analysis applies to any minimizer $\wh f_n.$ If the class $\mF$ is convex, computing the estimator $\wh f_n$ results in a convex optimization problem, which can also be written as a quadratic programming problem, see \cite{MR2360803}. For a fixed function $f,$ the law of large numbers implies that the least squares objective $\sum_i (Y_i-f(X_i))^2$ is close to its expectation $n\E[(Y_1-f(X_1))^2]=n+n\E[\int^1_0 (f_0(x) -f(x))^2 p(x) \, dx].$ It is therefore natural that the standard analysis of LSEs based on empirical process methods and metric entropy bounds for the function class $\mF$ leads to convergence rates with respect to the empirical $L^2$-loss $\|\wh f-f_0\|_n^2:=\tfrac 1n \sum_{i=1}^n (\wh f_n(X_i)-f_0(X_i))^2$ and the associated population version $\E[\int_0^1 (\wh f_n(x) -f_0(x))^2 p(x) \, dx],$ see, for instance, \cite{wainwrighthighdimstat, 2019arXiv190902088K, MR1920390, stone1980optimal}. The latter risk is the expected squared loss if a new $X$ is sampled from the design distribution $\P_\sX$ and $f_0(X)$ is estimated by $\wh f_n(X).$

A widely observed phenomenon is that the distribution of the new $X$ is different from the design distribution of the training data. For example, assume that we want to predict the response $Y$ of a patient to a drug based on a measurement $X$ summarizing the patient's health status. To learn such a relationship, data are collected in one hospital resulting in an estimator $\wh f_n$. Later $\wh f_n$ will be applied to patients from a different hospital. It is conceivable that the distribution of $X$ in the other hospital is different. For instance, there could be a different age distribution, or patients have a different socio-economic status due to variations in the imposed treatment costs. 

Therefore, an important problem is to evaluate the estimator's expected squared risk if a new observation $X$ is sampled from a different design distribution $\Q_\sX$ with density $q$. The associated prediction error under the new distribution is then
\begin{align}
    \int_0^1 \big( \wh f_n(x)-f_0(x)\big)^2 q(x) \, dx.
    \label{eq.prediction_risk_under_q}
\end{align}
If $\P_\sX$ and $\Q_\sX$ are similar enough so that for some finite constant $C$ and any $x\in [0,1]$, $q(x)\leq Cp(x),$ then, the prediction error under the design density $q$ is of the same order as under $p.$ However, in machine learning applications, there are often subsets of the domain with very few data points. This motivates the relevance of the problematic case, where the density $q$ is large in a low-density region of $p.$ Differently speaking, we are more likely to see a covariate $X$ in a region with few training data based on the sample $(X_1,Y_1),\ldots,(X_n,Y_n).$ Since the lack of data in such a region means that the LSE will not fit the true regression function $f_0$ well, this could lead to a large prediction error under the new design distribution. 

An extension of this problem setting is transfer learning under covariate shift. Here we know the least-squares estimator $\wh f_n$ and the sample size $n$ based on the sample $(X_1,Y_1),\ldots,(X_n,Y_n)$ with design density $p.$ On top of that, we have a second, smaller dataset with $m \ll n$ new i.i.d.\ data points $(X_1',Y_1'),\dots,(X_m',Y_m')$ where $Y_i'=f_0(X_i')+\eps_i',$ $i=1,\dots,m$ and $X_1'\sim \Q_\sX.$ In the framework of the hospital data above, this means that we also have data from a small study with $m$ patients from the second hospital. In other words, the regression function $f_0$ remains unchanged, but the design distribution changes. Since the number of extra training data points $m$ is small compared to the original sample size $n,$ we want to quantify how well an estimator combining $\wh f_n$ and the new sample can predict under the new design distribution with associated prediction error \eqref{eq.prediction_risk_under_q}. Establishing convergence rates for the loss in \eqref{eq.prediction_risk_under_q}, given a sample with design density $p$, is, however, a hard problem. To our knowledge, no simple modification of the standard least squares analysis allows to obtain optimal rates for this loss.

To address this problem, we study the case where the LSE is selected within the function class $\mF$ consisting of all $1$-Lipschitz functions. For this setting, we prove under weak assumptions that for a sufficiently large constant $K$ and any $0\leq x\leq 1,$ 
\begin{align}
    \big|\wh f_n(x) - f_0(x)\big| \leq Kt_n(x)
    \label{eq.what_we_prove}
\end{align}
with high probability, where the local convergence rate $t_n$ is the function that is, pointwise, the solution to the equation
\begin{align*}
    t_n(x)^2\, \P_\sX([x\pm t_n(x)])=\frac{\log n}n.
\end{align*}
The article \cite{gaiffas2009uniform} shows under slightly different assumptions on the design density that $t_n$ is locally the optimal estimation rate and constructs a wavelet thresholding estimator that is specifically designed to attain this local convergence rate. We prove that the LSE achieves this optimal local rate without any tuning. This is surprising since the LSE is based on minimization of the (global) empirical $L^2$-distance, and convergence in $L^2$ is weaker than convergence in the weighted sup-norm loss underlying the statement in \eqref{eq.what_we_prove}.

To establish \eqref{eq.what_we_prove}, we only assume a local doubling property of the design distribution. By imposing more regularity on the design density, we can prove that $t_n(x)\asymp (\log n/(np(x))^{1/3}.$ For this result, $p$ is also allowed to depend on the sample size $n$ such that the $p(x)$ in the denominator does not only change the constant but also the local convergence rate. This quantifies how the local convergence rate varies depending on the density $p$ and how small-density regions increase the local convergence rate. In Section \ref{sec.main}, we argue that kernel smoothing with fixed bandwidth has a slower convergence rate than the LSE. Therefore, the least squares fit can better recover the regression function if the values of the density $p$ range over different orders of magnitude. This property is particularly important for machine learning applications.

Based on \eqref{eq.what_we_prove}, we can then obtain a high-probability bound for the prediction error in \eqref{eq.prediction_risk_under_q} by 
\begin{align*}
    \int_0^1 \big( \wh f_n(x)-f_0(x)\big)^2 q(x) \, dx
    \leq K^2 \int_0^1 t_n(x)^2 q(x) \, dx.
\end{align*}
In many cases, simpler expressions for the convergence rate can be derived from the right-hand side. For instance in the case $t_n(x)\asymp (\log n/(np(x))^{1/3},$ the convergence rate is $(\log n/n)^{2/3}$ if $\int_0^1 q(x)/p(x)^{2/3} \, dx$ is bounded by a finite constant.\\

A major contribution of this paper is the proof strategy to establish local convergence rates. For that, we argue by contradiction, first assuming that the LSE has a slower local rate. Afterwards, we construct a local perturbation with smaller least squares loss. This means that the original estimator was not the LSE, leading to the desired contradiction. While a similar strategy has been followed for shape-constrained estimation in \cite{MR2546798, MR1891742}, the construction of the local perturbation and the verification of a smaller least squares loss for Lipschitz functions are both non-standard and involved. We believe that these arguments can be generalized to various extensions beyond Lipschitz function classes.

The paper is structured as follows. In Section \ref{sec.main}, we state the new upper and lower bounds on the local convergence rate. This section precedes a discussion on the imposed doubling condition and examples in Section \ref{sec.examples}. Section \ref{sec.proof_strategy} gives a high-level overview of the new proof strategy to establish local convergence rates. The full proof can be found in Section \ref{sec.proof}. Applications to transfer learning are discussed in Section \ref{sec:TL}. Section \ref{sec:discussion} provides a brief literature review and an outlook. The remaining proofs are deferred to the supplement \cite{supplement}.

\textit{\textbf{Notation:}} For two real numbers $a, b,$ we write $a\vee b = \max(a, b)$ and  $a\wedge b = \min(a, b).$ For any real number $x$, we denote by $\lceil x\rceil$ the smallest integer $m$ such that $m \geq x$ and by $\lfloor x\rfloor$ the greatest integer $m$ such that $m \leq x$. Furthermore, for any set $S$, we denote by $x \mapsto \mathds{1}(x \in S)$ the indicator function of the set $S$. To increase readability of the formulas, we define $[a \pm b] := [a-b, a+b].$ For any two positive sequences $\{a_n\}_{n }, \{b_n\}_{n},$ we say that $a_n \lesssim b_n$ if there exists a constant $0 <c< \infty,$ and a positive integer $N$ such that for all $n \geq N, a_n \leq cb_n$. We write $a_n \asymp b_n$ if $a_n \lesssim b_n$ and $b_n \lesssim a_n$. Finally, if for all $\varepsilon > 0,$ there exists a positive integer $N$ such that for all $n \geq N, a_n \leq \varepsilon b_n,$ then we write $a_n \ll b_n$. For a random variable $X$ and a (measurable) set $A,$ $\P_\sX(A)$ stands for $\P(X\in A).$ For any function $h$ for which the integral is finite, we set $\|h\|_{L^2(\P)}:=\big(\int h^2(x) p(x) \, dx\big)^{1/2}.$ We also write $\|h\|_n := \big(\tfrac 1n \sum_{i=1}^n h^2(X_i)\big)^{1/2}.$

\section{Main results}
\label{sec.main}

In this section, we state the local convergence results for the LSE. Set 
\begin{align*}
    \mM
    := \big\{\text{probability measures that are both supported on $[0, 1]$ and admit a Lebesgue density} \big\}.
\end{align*}
The local convergence rate $t_n$ turns out to be the functional solution to an equation that depends on the design distribution $\P_\sX$.

\begin{lemma}
\label{lem.existence_of_tn}
If $\P_\sX\in \mM$, then, for any $n >1$ and any $x \in [0, 1],$ there exists a unique solution $t_n(x)$ of the equation
\begin{align*}
    t_n(x)^2 \, \P_\sX\big(\big[x\pm t_n(x)\big]\big) = \frac{\log n}n.
\end{align*}
Therefore the function $x\mapsto t_n(x)$ is well defined on $[0,1]$. From now on, we refer to $t_n$ as the spread function (associated to $\P_\sX$).
\end{lemma}

The spread function can be viewed as a measure for the local mass of the distribution $\P_\sX$ around $x$. The more mass $\P_\sX$ has around $x$, the smaller $t_n(x)$ is. Whenever necessary, the spread function associated to a probability distribution $\P$ is denoted by $t_n^{\P}$.

To derive a local convergence rate of the least-squares estimator taken over Lipschitz functions, one has to exclude the possibility that the design distribution $\P_\sX$ is completely erratic. Interestingly, no H\"older smoothness has to be imposed on the design density, and it is sufficient to consider design distributions satisfying the following weak regularity assumption.

\begin{definition}
\label{def:local_doubling}
For $n\geq 3$ and $D\geq 2,$ define $\mathcal{P}_n(D)$ as the class of all design distributions $\P_\sX \in \mM$, such that for any $0 < \eta \leq \sqrt{\log n}\, \sup_{x \in [0,1]}t_n(x),$ 
    \begin{align}
        \label{eq.LDP}
        \sup_{x \in [0,1]}\frac{\P_\sX\big([x - 2\eta, x+ 2\eta]\big)}{\P_\sX\big([x - \eta, x+ \eta]\big)} \leq D.\tag{LDP}
    \end{align}
We call \eqref{eq.LDP} the local $D$-doubling property, or local doubling property when the constant $D$ is irrelevant or unambiguous. A design distribution $\P_\sX$ satisfies the (global) doubling property if \eqref{eq.LDP} holds for all $\eta>0.$ Denote by $\mP_G(D)$ the space of all globally doubling measures in $\mM$.
\end{definition}
The restriction $x\in [0,1]$ allows to include distributions with Lebesgue densities that are discontinuous at $0$ or $1.$ For instance the uniform distribution on $[0,1]$ is $2$-doubling, but since $\P_\sX[-3\eta,\eta]/\P_\sX[-2\eta,0]=\infty,$ \eqref{eq.LDP} does not hold if the supremum includes $x=-\eta.$

Since the uniform distribution on $[0,1]$ is contained in $\mathcal{P}_n(2)\subseteq \mathcal{P}_n(D)$ for $D\geq 2,$ we see that these classes are non-empty. Inequality \eqref{eq.LDP} states that doubling the size of a small interval cannot inflate its probability by more than a factor $D.$ The next result shows that the maximum interval size $\sqrt{\log n}\, \sup_{x \in [0,1]}t_n(x)$ tends to zero as $n$ becomes large. 
\begin{lemma}
\label{lem.big_prob}
Let $\P_\sX \in \mP_n(D)$ with $D\geq 2$. If $\varepsilon>0,$ then there exists an $N = N(\varepsilon, D)$ such that for all $n \geq N, \sqrt{\log n}\, \sup_{x \in [0,1]}t_n(x)<\varepsilon.$
\end{lemma}

The local doubling condition allows us to consider sample size-dependent design distributions. See Section \ref{sec.examples} for a more in-depth discussion and some examples.

We now show that the spread function is indeed the minimax rate. Denote by $\Lip(\kappa)$ the set of functions $f \colon [0, 1] \to \RR$ that are Lipschitz, with Lipschitz constant at most $\kappa,$ that is, $f\in \Lip(\kappa)$ iff $|f(x)-f(y)|\leq \kappa|x-y|$ for all $x,y\in [0,1].$ If $f\in \Lip(\kappa),$ then we also say that $f$ is $\kappa$-Lipschitz. Recall that $\P_{\!f_0}$ is the distribution of the data in the nonparametric regression model \eqref{eq.mod} if the true regression function is $f_0$ and that $\P_\sX$ denotes the distribution of the design $X$.  

\begin{theorem}
\label{th.main_theorem}
Consider the nonparametric regression model \eqref{eq.mod}. Let $0<\delta<1,$ and $D \geq 2$. If $\wh f_n$ denotes the LSE taken over the class of $1$-Lipschitz functions $\Lip(1)$, then, for a sufficiently large constant $K = K(D, \delta),$
\begin{align*}
    \sup_{\P_\sX \in \mP_n(D)} \  \sup_{f_0\in \Lip(1-\delta)} \P_{\!f_0} \left(\sup_{x \in [0, 1]} t_n(x)^{-1}|\wh f_n(x) - f_0(x)| > K\right) \to 0 \quad \text{as} \ n\to \infty.
\end{align*}
\end{theorem}
The proof reveals that if the constant $K$ is chosen as the value $K_*$ defined in \eqref{eq.K*_def}, the right-hand side of Theorem \ref{th.main_theorem} converges to zero with a polynomial rate in the sample size $n$. For $\delta \to 0,$ $K_* \asymp \delta^{-1/2 - 3\log_2(D)/4}.$ Consequently, the constant $K$ will become large for small $\delta$ and large doubling constant $D$. We want to stress that no attempt has been made to optimize the constants and that further refinements of the inequalities in the proof will likely improve the constant $K$ considerably.

Since the previous result is uniform over design distributions $\P_\sX \in \mathcal{P}_n(D),$ we can also consider sequences $\P_\sX^n.$ While, at first sight, it might appear unnatural to consider for every sample size $n$ a different design distribution, this constitutes a useful statistical concept to study the effect of low-density regions on the convergence rate. Indeed, the influence of a small density region disappears in the constant for a fixed density, while the dependence on the sample size makes the effect visible in the convergence rate. Moreover, sample size-dependent quantities are widely studied in mathematical statistics, most prominently in high-dimensional statistics, where the number of parameters typically grows with the sample size.

One key question is to identify conditions for which the local convergence rate $t_n$ has a more explicit expression. One such instance is the case of H\"older-smooth design densities. Let $\lfloor \beta \rfloor$ denote the largest integer that is strictly smaller than $\beta.$ The H\"older-$\beta$ semi-norm of a function $g:\mathbb{R}\to \mathbb{R}$ is defined as
\begin{align}
    |g|_{\beta}:=\sup_{x,y\in \mathbb{R}, \, x\neq y} \, \frac{|g^{(\lfloor \beta\rfloor)}(x)-g^{(\lfloor \beta\rfloor)}(y)|}{|x-y|^{\beta-\lfloor \beta\rfloor}}.
\end{align}
For $\beta=1,$ $|g|_\beta$ is the Lipschitz constant of $g$.

\begin{corollary}
\label{cor.corollary_lip}
Consider the nonparametric regression model \eqref{eq.mod}. Let $0<\delta<1$ and $\wh f_n$ be the LSE taken over the class of $1$-Lipschitz functions $\Lip(1)$. For $\beta \in (0, 2]$, let $\P_\sX^n$ be a sequence of distributions with corresponding Lebesgue densities $p_n.$ If for any $n,$ there exists a non-negative function $h_n$ with $p_n(x)=h_n(x)$ for all $x\in [0,1],$ $\max_n |h_n|_\beta \leq \kappa$ and $\min_{x \in [0, 1]} p_n(x) \geq n^{-\beta/(3+\beta)} \log n,$ then, for all $n \geq \exp(4\kappa)\vee 9,$
\begin{align}
    \Big(\frac{\log n}{3np_n(x)}\Big)^{1/3} \leq t_n(x) \leq \Big(\frac{2\log n}{np_n(x)}\Big)^{1/3},
    \label{eq.tn_equiv}
\end{align}
$\P_\sX^n \in \mP_n(2 + 2^{\beta/3}3^\beta \kappa + 2^{1/3}3\kappa^{1/\beta}),$ and there exists a finite constant $K'$ independent of the sequence $\P_\sX^n$, such that
\begin{align*}
        \sup_{f_0 \in \Lip(1 - \delta)} \ \P_{\!f_0}\bigg(\sup_{x \in [0, 1]}p_n(x)^{1/3}\big|\wh f_n(x) - f_0(x)\big| \geq K'\Big(\frac{\log n}n\Big)^{1/3}\bigg) \to 0 \quad \text{as }n\to \infty.
\end{align*}
\end{corollary}

In the previous result, the regression function is assumed to be Lipschitz, and $\beta$ denotes the smoothness index of the design densities $p_n.$ The convergence rate $(\log n/n)^{1/3}$ is known to be the optimal nonparametric rate for Lipschitz regression functions, sup-norm loss and uniform fixed design, cf. \cite{nonparest}, Corollary 2.5.

The rate $(\log n/(np_n(x))^{1/3}$ is natural, since $np_n(x)$ can be viewed as local effective sample size around $x.$

For $\beta\in (0,1],$ we can always choose $h_n(x)=p_n(0)$ for $x<0,$ $h_n(x)=p_n(x)$ for $x\in [0,1]$, and $h_n(x)=p_n(1)$ for $x>1.$ While the rate is independent of the smoothness index $\beta,$ we can allow faster decaying low density regions if $\beta$ gets larger. The fastest possible decay is $n^{-2/5}\log n$ if $\beta=2.$

To extend the result to $\beta>2$ and to allow for even smaller densities, it is widely believed that imposing H\"older smoothness is insufficient. One way around this is to use H\"older smoothness plus some extra flatness constraint. See \cite{MR3449768, MR3714756} for more on this topic. 

The lower bound on the small density regions in Corollary \ref{cor.corollary_lip} ensures that the local doubling property \eqref{eq.LDP} is satisfied. A lower bound is also necessary, as otherwise $p_n(x)\ll \log n/n$ would imply that the rate $t_n(x) \asymp (\log n/(np_n(x))^{1/3}$ diverges. The next lemma shows how the spread function behaves at a point with vanishing Lebesgue density $p.$

\begin{lemma}
\label{lem.wright_assumption_explored_for_tn}
Let $\P_\sX \in \mP_G(D)$ with $D \geq 2$ and density $p$. Suppose that $p(x_0) = 0$ for some $x_0 \in [0, 1]$. If there exists some $A, \alpha > 0$ and an open neighbourhood $U$ of $x_0$ such that for any $x \in U,$ $1/A\leq |p(x) - p(x_0)|/|x - x_0|^\alpha\leq A,$ then, there exists $N>0$, depending only on $U$ and $D$, such that for any $n > N$,
\begin{align*}
    \Big(\frac{(\alpha + 1)\log n}{An}\Big)^{1/(\alpha + 3)} \leq t_n(x_0) \leq \Big(\frac{(\alpha + 1)A\log n}{n}\Big)^{1/(\alpha + 3)}.
\end{align*}  
\end{lemma}
An immediate consequence is that if $p$ is $k$ times differentiable, $p^{(\ell)}(x_0) = 0$ for all $\ell <k$, and $p^{(k)}(x_0) \neq 0$, then $t_n(x_0) \asymp (\log n/n)^{1/(k+3)}.$

We complement Theorem \ref{th.main_theorem} with a matching minimax lower bound. A closely related result is Theorem 2 in \cite{gaiffas2009uniform}.

\begin{theorem}
\label{th.lower_bound}
If $C_\infty$ is a positive constant, then there exists a positive constant $c,$ such that for any sufficiently large $n,$ and any sequence of design distribution $\P_\sX^n \in \mM$ with corresponding Lebesgue densities $p_n$ all upper bounded by $C_\infty,$ we have
\begin{align*}
    \inf_{\wh f_n} \sup_{f_0 \in \Lip(1)} \P_{\!f_0}\bigg( \sup_{x \in [0, 1]} \,  t_n(x)^{-1}|\wh f_n(x) - f_0(x)| \geq \frac {1}{12} \bigg) \geq c,
\end{align*}
where the infimum is taken over all estimators.
\end{theorem}

The proof is deferred to Appendix D in the Supplement \cite{supplement}.

Corollary \ref{cor.corollary_lip} states that $t_n(x) \asymp (\log n/(np_n(x)))^{1/3}.$ Combined with the lower bound, this shows that the local minimax estimation rate in this framework is $(\log n/(np_n(x)))^{1/3}.$

It is known that for Lipschitz functions and squared $L^2$-loss, the LSE achieves the minimax estimation rate $n^{-2/3}$. Summarizing the statements on the convergence rates above shows that the LSE is also minimax rate optimal with respect to the stronger weighted sup-norm loss.\\

Next, we discuss how the derived local rates imply several advantages of the LSE if compared to kernel smoothing estimators. In the case of uniform design $p(x)=\mathds{1}(x\in [0,1]),$ the LSE achieves the convergence rate $n^{-2/3}$ with respect to squared $L^2$-loss and Corollary \ref{cor.corollary_lip} gives the rate $(\log n/n)^{1/3}$ with respect to sup-norm loss. To our knowledge, it is impossible to obtain these two rates simultaneously for kernel smoothing estimators. The squared $L^2$ rate $n^{-2/3}$ can be achieved for a kernel bandwidth $h\asymp n^{-1/3}$ and the sup-norm rate $(\log n/n)^{1/3}$ requires more smoothing in the sense that the bandwidth should be of the order $(\log n/n)^{1/3},$ see Corollary 1.2 and Theorem 1.8 in \cite{nonparest}. Any bandwidth choice in the range $n^{-1/3}\lesssim h\lesssim (\log n/n)^{1/3}$ will incur an additional $\log n$-factor in at least one of these two convergence rates of the kernel smoothing estimator. Although the suboptimality in the rate is only a $\log n$-factor, it is surprising that the LSE does not suffer from this issue.

Secondly, we argue that kernel smoothing estimators with fixed global bandwidth cannot achieve the local convergence rate $(\log n/(np_n(x))^{1/3}$ in the setting of Corollary \ref{cor.corollary_lip}. Denote the bandwidth by $h$ and the kernel smoothing estimator by $\wh f_{nh}.$ The decomposition in stochastic error and bias leads to an inequality of the form
\begin{align}
    \big|\wh f_{nh}(x) -f_0(x)\big| 
    \lesssim \underbrace{\sqrt{\frac{\log n}{nhp_n(x)}}}_{\text{stochastic error}}
    + \underbrace{h}_{\mathrlap{\text{deterministic error}}}, \quad \text{for all \ } x\in [0,1],
    \label{eq.kse_error_bound}
\end{align} 
with high probability. Since the dependence on the density $p_n(x)$ is typically ignored, we provide a heuristic for this bound in Appendix B in the supplement \cite{supplement}. To balance the two errors, one would have to choose as a bandwidth $h\asymp (\log n/(np_n(x)))^{1/3}.$ In this case, the local convergence rate would also be $(\log n/(np_n(x))^{1/3}$. But this requires choosing the bandwidth locally depending on $x.$ From that, one can deduce that any global choice for $h$ in \eqref{eq.kse_error_bound} leads to suboptimal local rates. It is, therefore, surprising that although the LSE is based on a global criterion, it changes the amount of smoothing locally to adapt to the amount of data points in each regime. This is a clear advantage of the least squares method over smoothing procedures. This benefit seems particularly advantageous for machine learning problems that typically have high- and low-density regions in the design distribution.\\

To draw uniform confidence bands, but also for the application to transfer learning discussed later, it is important to estimate the spread function $t_n$ from data. For $\wh \P_\sX^n(A):=\tfrac 1n \sum_{i=1}^n \mathds{1}(X_i\in A)$ the empirical design distribution, a natural estimator is
\begin{align}
    \wh t_n(x):= \inf \Big\{t: t^2\wh \P_\sX^n([x\pm t])\geq \frac{\log n}{n}\Big\}. 
    \label{eq.tn_def}
\end{align}

\begin{theorem}
\label{thm.estimate_tn}
If $\P_\sX\in \mP_G(D)$ for some $D \geq 2$ and $\|p\|_\infty < \infty,$ then 
\begin{align*}
    \max_{n>1} \, \sup_{x\in[0,1]} \, \sqrt{\log n} \, \Big| \frac{\wh t_n(x)}{t_n(x)}-1\Big|<\infty, \quad \text{almost surely,}
\end{align*}
where $t_n(x)$ is the spread function associated to $\P_\sX.$
\end{theorem}

The result implies that for any $\varepsilon>0$ and all sufficiently large $n,$ we have $(1-\varepsilon)t_n(x)\leq \wh t_n(x)\leq (1+\varepsilon)t_n(x)$ for all $x\in [0,1].$

\section{Local doubling property and examples of local rates}
\label{sec.examples}

By Definition \ref{def:local_doubling}, $\mP_n(D)$ is the class of all locally doubling distributions, and $\mP_G(D)$ is the class of all globally doubling distributions in $\mM$. It follows from the definitions that $\mP_G(D) \subseteq \mP_n(D)$. A converse statement is
\begin{lemma}
\label{lem.conversion_local_global}
If $\P_\sX \in \mP_n(D)$, then $\P_\sX \in \mP_G(D_n(\P_\sX))$ for a finite number $D_n(\P_\sX).$ In particular, if $\P_\sX$ does not depend on the sample size $n$, neither does $D_n(\P_\sX).$
\end{lemma}

This means that the distinction between local and global doubling is only relevant in the case where we study sequences of design distributions, such as in the setup of Corollary \ref{cor.corollary_lip}. In Example \ref{ex.example3}, a sequence $\P_\sX^n$ is constructed such that $\P_\sX^n \in \mP_n(D)$ for all $n$ and $\P_\sX^n \in \mP_G(D_n)$ necessarily requires $D_n\to \infty$ as $n\to \infty.$

Doubling is known to be a weak regularity assumption and does not even imply that $\P_\sX$ has a Lebesgue density \cite{MR245734, MR1800819}. It can, moreover, be easily verified for a wide range of distributions. All distributions with continuous Lebesgue density bounded away from zero and all densities of the form $p(x) \propto x^{\alpha}$ for $\alpha \geq 0$ are doubling.

Examples of non-doubling measures are distributions $\P_\sX$ with $\P_\sX([a, b]) = 0$ for some $0 \leq a < b\leq 1$, see Lemma 29 in Appendix C in the supplement \cite{supplement}. If a density verifies $p(x_0) = 0$ for some $x_0 \in [0, 1]$ and behaves like an inverse exponential around $x_0$, then it is not in $\mP_n$ for any constant. The density $x\mapsto x^{-2}e^{1-1/x}\mathds{1}(x\in [0,1])$ with corresponding cumulative distribution function (c.d.f.)\ $x\mapsto e^{1-1/x}\mathds{1}(x\in [0,1])$ provides an example of such a behaviour. To see this, observe that $P([-2\eta,2\eta])=e^{1-1/(2\eta)}=e^{1/(2\eta)}e^{1-1/\eta}=e^{1/(2\eta)}P([-\eta,\eta]).$ Since $e^{1/(2\eta)}\to \infty$ as $\eta\to 0,$ \eqref{eq.LDP} cannot hold.\\

We now derive explicit expressions for the local convergence rates and verify the \eqref{eq.LDP} for different design distributions by proving that $\P_\sX\in \mP_G(D)$ or $\P_\sX\in \mP_n(D)$.

\begin{example}
\label{ex:simple_distrib} 
Assume that the design density $p$ is bounded from below and above, in the sense that 
\begin{align}
    0<\ul p:=\inf_{x\in [0,1]}\leq \ol p:=\sup_{x\in [0,1]} p(x) < \infty.
    \label{eq.p_lb_ub}
\end{align}
The following result shows that in this case Theorem \ref{th.main_theorem} is applicable and the local convergence rate is $t_n(x) \asymp (\log n/n)^{1/3}.$

\begin{lemma}\label{lem.example_bounded_density}
Assume that the design distribution $\P_{\sX}$ admits a Lebesgue density satisfying \eqref{eq.p_lb_ub}. Then, $\P_{\sX}\in \mathcal{P}_G(4\ol p/\ul p)$ and for any $0\leq x\leq 1,$
\begin{align}
\label{eq.bounded_p_tn}
    \Big(\frac{\log n}{2n\ol p}\Big)^{1/3}
    \leq t_n(x) \leq \Big(\frac{\log n}{n\ul p}\Big)^{1/3}.
\end{align}
\end{lemma}
\end{example}

As a second example, we consider densities that vanish at $x=0.$

\begin{example}[Density vanishing with polynomial speed at zero]
\label{ex.example2}
Assume that, for some $\alpha>0,$ the design distribution $\P_\sX$ has Lebesgue density 
\begin{align*}
    p(x) = (\alpha + 1)x^{\alpha + 1}\mathds{1}(x \in [0,1]).
\end{align*}
This means that there is a low-density regime near zero with rather few observed design points. In this regime, it is more difficult to estimate the regression function, which is reflected in a slower local convergence rate.

\begin{lemma}
\label{lem.example_growing_density}
If $n>9, \alpha>0$ and $a_n := (\log n/n2^{\alpha + 1})^{1/(\alpha + 3)},$ then, the distribution with density $p \colon x \mapsto (\alpha + 1)x^{\alpha}\mathds{1}(x \in [0, 1])$ is in $\mP_G(D)$ for some $D$ depending only on $\alpha$. Thus, Theorem \ref{th.main_theorem} is applicable and
\begin{align}
    \label{eq.stated_bounds_first_regime}
    \Big(\frac{\log n}{2^{\alpha+1}n}\Big)^{1/(\alpha+3)} \leq t_n(x) \leq \Big(\frac{\log n}n\Big)^{1/(\alpha + 3)}, \quad \text{for } 0 \leq x \leq a_n,
\end{align}
and 
\begin{align}
    \label{eq.stated_bounds_second_regime}
    \Big(\frac{\log n}{2^{\alpha+1}(\alpha+1)nx^{\alpha}}\Big)^{1/3} \leq t_n(x) \leq \Big(\frac{\log n}{nx^{\alpha}}\Big)^{1/3}, \quad \text{for } a_n \leq x \leq 1.
\end{align}
\end{lemma}
By rewriting the expression for the spread function, we find that the local convergence rate is $t_n(x) \asymp (\log n/(n(x\vee a_n)^\alpha))^{1/3}.$ The behavior of $t_n(0)$ can also directly be deduced from Lemma \ref{lem.wright_assumption_explored_for_tn}.
\end{example}

As a last example, we consider a sequence of design distributions with decreasing densities on $[1/4,3/4].$
\vspace{\baselineskip}

\noindent
\textbf{Example 3 (Sequence of distributions with low-density region).}\label{ex.example3} For $\phi_n=1\wedge n^{-1/4}\log n,$ consider the sequence of distributions $\P_\sX^n$ with associated Lebesgue densities
\begin{align}
    p_n(x) := \phi_n +16(1-\phi_n)\max\Big(\frac 14-x,0,x-\frac 34\Big).
    \label{eq.Ex2_pn_def}
\end{align}
\begin{wrapfigure}{r}{.45\textwidth}
    \centering
    \includegraphics[width=.4\textwidth]{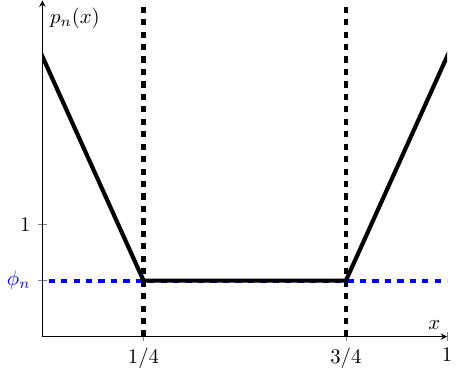}
    \caption{The density $p_n$}
    \label{fig.p_n(x)}
    \vspace{-1.2cm}
\end{wrapfigure}

It is easy to check that this defines Lebesgue densities on $[0,1]$. See Figure \ref{fig.p_n(x)} for a graph of $p_n$. According to Lemma \ref{lem.example_bounded_density}, these distributions are globally doubling. Since $\P_\sX^n([0, 1])/\P_\sX^n([1/4, 3/4]) = 1/\phi_n$, the doubling constants are $\geq 1/\phi_n$ and hence tend to infinity as $n$ grows. Therefore there is no $D>0$ such that $p_n \in \mP_G(D)$ for all $n$. On the contrary, for all $n,\ p_n \in \Lip(16)$ and since $\phi_n \geq n^{-1/4}\log n$, the assumptions of Corollary \ref{cor.corollary_lip} are satisfied with $\beta =1$ and $\kappa = 16$. Therefore, $p_n \in \mP_n(8)$ for all $n$ large enough and the local convergence rate is $(\log n/(np_n(x)))^{1/3}.$ In particular, in the regime $[1/4,3/4]$, the local convergence rate becomes $n^{-1/4}.$

\section{Proof strategy}
\label{sec.proof_strategy}

As the new proof strategy to establish local rates for least squares estimation is the main mathematical contribution of this work, we outline it here. Consider the LSE
\begin{align*}
    \wh f_n \in \argmin_{f \in \Lip(1)} \,  \sum_{i=1}^n\left(Y_i - f(X_i)\right)^2.
\end{align*}
The definition of the estimator ensures that for any $g \in \Lip(1),$ the so called \textit{basic inequality} $\sum_{i=1}^n(Y_i - \wh f_n(X_i))^2 \leq \sum_{i=1}^n (Y_i-g(X_i))^2$ holds. Assume that the function $g$ satisfies
\begin{align}
    \big(\wh f_n(X_i)-g(X_i) \big)\big(g(X_i)-f_0(X_i)\big) \geq 0, \quad \text{for all} \ i=1,\ldots,n,
    \label{eq.balanced_eqs_PS}
\end{align}
which is the same as saying that at all data points $g$ should lie between $\wh f_n$ and the true regression function $f_0.$ Together with the basic inequality and using $Y_i = f_0(X_i) + \varepsilon_i$, we obtain
\begin{align}
\begin{split}
    \sum_{i=1}^n\big(\wh f_n(X_i) - g(X_i)\big)^2 
    &\leq 
    \sum_{i=1}^n\big(\wh f_n(X_i) - f_0(X_i)\big)^2 - \sum_{i=1}^n\big(g(X_i) - f_0(X_i)\big)^2 \\
    &\leq 2\sum_{i=1}^n \varepsilon_i\big(\wh f_n(X_i) - g(X_i)\big).
\end{split}    
    \label{eq.argmin_inequality_PS}
\end{align}
We prove that $t_n(x)$ is a local convergence rate by contradiction. Assume that the LSE $\wh f_n$ is more than $Kt_n(x^*)$ away from the true regression function $f_0$ for some $x^* \in [0, 1]$ and a sufficiently large constant $K$. Then, we choose $g$ as a specific local perturbation of $\wh f_n$ (in the sense that $g$ differs from $\wh f_n$ only on a small interval) such that the previous inequality \eqref{eq.argmin_inequality_PS} is violated, resulting in the desired contradiction.

Denote the space of all possible functions $\wh f_n-g$ by $\mF^*.$ Since $\wh f_n\in \Lip(1)$ and $g \in \Lip(1),$ we have $\wh f_n-g \in \Lip(2)$ and thus, $\mF^* \subseteq \Lip(2).$ In fact, by choosing $g$ as a local perturbation of $\wh f_n,$ the function class $\mF^*$ will be much smaller than $\Lip(2).$ Due to the small support of $\wh f-g$, we have $\wh f(X_i)-g(X_i)=0$ for most $X_i.$ It is conceivable that one can remove these indices from \eqref{eq.argmin_inequality_PS} and that the effective sample size $m=m(X_1,Y_1,\dots,X_n,Y_n)$ is the number of indices for which $\wh f_n(X_i)-g(X_i)\neq 0.$ Denote by $N(r,\mF^*, \|\cdot\|_\infty)$ the covering number of $\mF^*$ with sup-norm balls of radius $r$ and assume moreover that $\mF^*$ is star-shaped, that is, if $h\in \mF^*$ and $\alpha \in [0,1],$ then also $\alpha h\in \mF^*.$ We now argue similarly as in \cite{wainwrighthighdimstat}. Replacing $f^*$ by $g$ in their inequality (13.18, p.\ 452) and then following exactly the same steps as in the proofs for their Theorem 13.1 and Corollary 13.1, one can now show that if there exists a sequence $\eta_n$ with $0\leq \eta_n \leq 1$ satisfying
\begin{align}
    \frac{16}{\sqrt{m}}
    \int_{\eta_n^2/4}^{\eta_n}
    \sqrt{\log N\big(r,\mF^*, \|\cdot\|_\infty \big)} \, dr
    \leq \frac{\eta_n^2}{4},
    \label{eq.heuristic_covering_ineq}
\end{align}
then, 
\begin{align}
    \P\Big(\frac 1m\sum_{i=1}^n \big(\wh f_n(X_i)-g(X_i)\big)^2 \geq 16 \eta_n^2\Big) \leq e^{-m\eta_n^2/2}.
    \label{eq.estim_conc_around_pert}
\end{align}
To derive a contradiction assume that there exists a point $x^*$ such that $|\wh f_n(x^*)-f_0(x^*)| > Kt_n(x^*).$ Suppose moreover, that for all $K$ large enough, we can find a function $g\in \Lip(1)$ satisfying \eqref{eq.balanced_eqs_PS} and $|\wh f_n(x^*)-g(x^*)| > Kt_n(x^*),$ and support of $\wh f_n-g$ with length of the order $Kt_n(x^*).$ The assumed properties of such a function $g$ are plausible due to $\wh f_n-g \in \Lip(2).$ Because we can also choose $K\geq 4,$ another consequence of $\wh f_n-g \in \Lip(2)$ is that $|\wh f_n(x)-g(x)| > K t_n(x^*)/2$ for all $x\in [x^*\pm t_n(x^*)].$ Thus, 
\begin{align*}
    \sum_{i=1}^n \big(\wh f_n(X_i)-g(X_i)\big)^2
    \geq \frac{K^2}4 t_n(x^*)^2
    \sum_{i=1}^n \mathds{1}\big(X_i \in [x^*\pm t_n(x^*)]\big).
\end{align*}
The right-hand side should be close to its expectation $\tfrac 14 K^2t_n(x^*)^2nP_X( [x^*\pm t_n(x^*)])= \tfrac 14 K^2 \log n,$ where we used the definition of $t_n(x^*).$ Thus, up to approximation errors, we obtain the lower bound
\begin{align}
    \sum_{i=1}^n \big(\wh f_n(X_i)-g(X_i)\big)^2
    \geq \frac{K^2}{4}\log n.
    \label{eq.heuristic_lb}
\end{align}

We now explain the choice of $\eta_n$ in \eqref{eq.estim_conc_around_pert} that leads to the upper bound for $\sum_{i=1}^n (\wh f_n(X_i)-g(X_i))^2$ in \eqref{eq.argmin_inequality_PS}. Since the perturbation is supported on an interval with length of the order $Kt_n(x^*),$ one can bound the metric entropy $\log N\big(r,\mF^*, \|\cdot\|_\infty \big)\lesssim Kt_n(x^*)/r,$ with proportionality constant independent of $K.$ Therefore, \eqref{eq.heuristic_covering_ineq} holds for $\eta_n \propto (Kt_n(x^*)\log n/m)^{1/3}.$ The additional $\log n$-factor in $\eta_n$ is necessary to obtain uniform statements in $x.$ For this choice of $\eta_n,$ the probability in \eqref{eq.estim_conc_around_pert} converges to zero. Consequently, on an event with large probability, we have that 
\begin{align}
    \sum_{i=1}^n \big(\wh f_n(X_i)-g(X_i)\big)^2 \leq 16 m \eta_n^2\lesssim \big(mK^2 t_n(x^*)^2\log^2 n\big)^{1/3}.
    \label{eq.9999}
\end{align}
Recall that the support of $\wh f_n-g$ is contained in $[x^*\pm CKt_n(x^*)]$ for some constant $C.$ Moreover, $m$ is the number of observations in the support of $\wh f_n-g.$ Now $m$ should be close to its expectation which can be upper bounded by $n\P_\sX([x^*\pm CKt_n(x^*)]).$ Invoking the local doubling property \eqref{eq.LDP}, $m$ can also be upper bounded by $C_K n\P_\sX([x^*\pm t_n(x^*)]),$ for a constant $C_K$ depending on $K.$ Using the definition of $t_n(x),$ \eqref{eq.9999} can be further bounded by
\begin{align*}
    &\sum_{i=1}^n \big(\wh f_n(X_i)-g(X_i)\big)^2
    \lesssim \Big( C_K K^2 n t_n(x^*)^2 \P_\sX([x^*\pm t_n(x^*)]) \log^2 n\Big)^{1/3}
    \leq (C_KK^2)^{1/3} \log n.
\end{align*}
Comparing this with the lower bound \eqref{eq.heuristic_lb} and dividing both sides by $\log n$, we conclude that on an event with large probability, $\tfrac 14 K^2 \lesssim (C_KK^2)^{1/3},$ where the proportionality constant does not depend on $K.$ A technical argument that links the upper and lower bound more tightly and that we omit here shows that one can even avoid the dependence of $C_K$ on $K,$ such that we finally obtain $K^2 \lesssim K^{2/3}.$ Taking $K$ large and since the proportionality constant is independent of $K,$ we finally obtain a contradiction. This means that on an event with large probability and for all sufficiently large $K$, there cannot be a point $x^*,$ such that $|\wh f_n(x^*)-f_0(x^*)|/t_n(x^*)> K,$ proving $\sup_x |\wh f_n(x)-f_0(x)|/t_n(x) \leq K$ with large probability.

\begin{wrapfigure}{r}{0.4\textwidth}
    \vspace{-\baselineskip}
    \centering
    \includegraphics[width=0.4\textwidth]{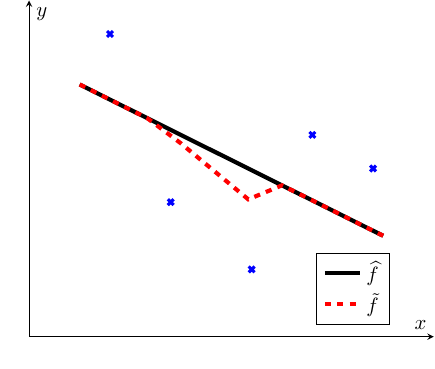}
    \vspace{-\baselineskip}
    \caption{\label{fig.1} If the LSE $\wh f$ would not have locally slope $=1$, then one could construct a perturbed version $\tilde f$ that better fits the data, implying that $\wh f$ cannot be a least squares fit.}
    \vspace{-1cm}
\end{wrapfigure}

There is still a major technical obstacle in the proof strategy, namely the choice of the local perturbation $g$. This construction appears to be one of the main difficulties of the proof. In fact, the empirical risk minimizer over $1$-Lipschitz functions will typically lie somehow on the boundary of the space $\Lip(1)$ in the sense that on small neighbourhoods, the Lipschitz constant of the estimator is exactly one.

To see this, assume the statement would be false. Then we could build tiny perturbations around the estimator that are $1$-Lipschitz and lead to a smaller least squares loss, which contradicts the fact that the original estimator is a least squares minimizer (see Figure \ref{fig.1}). This makes it tricky to construct a local perturbation $g$ of $\wh f$ that also lies in $\Lip(1)$ and satisfies the required conditions. To find a suitable perturbation, our approach is to introduce first $x^*$ as above and then define another point $\tilde x$ in the neighbourhood of $x^*$ with some specific properties. The full construction is explained in Figure \ref{fig.perturbation} and Lemma \ref{lem.g_function} in Section~\ref{sec.proof}.

\section{Applications to Transfer Learning}
\label{sec:TL}

Transfer Learning (TL) aims to exploit that an estimator achieving good performance on a certain task should also work well on similar tasks. This allows to emulate a bigger dataset and to save computational time by relying on previously trained models. In the supervised learning framework, we have access to training data generated from a distribution $\Q_{\sX,\sY}.$ Observing $X$ from a pair $(X,Y)\sim \Q_{\sX,\sY}$ we want to predict the corresponding value of $Y.$ To do so, we compute an estimate based on observing $m$ i.i.d.\ copies sampled from $\Q_{\sX, \sY}$. Assume now that we also have access to $n > m$ i.i.d.\ copies sampled from another distribution $\P_{\sX, \sY}$. The transfer learning paradigm states that, depending on some similarity criterion between $\P_\sX$ and $\Q_\sX$, fitting an estimator using both samples improves the predictive power. In other words, $\P_{\sX, \sY}$ contains information about $\Q_{\sX, \sY}$ that can be transferred to improve the fit. Two standard settings within TL are posterior drift and covariate shift. For posterior drift, one assumes that the marginal distributions are the same, that is, $\P_\sX = \Q_\sX,$ but $\P_{\sY|\sX}$ and $\Q_{\sY|\sX}$ may be different. On the contrary, TL with covariate shift assumes that $\P_{\sY|\sX} = \Q_{\sY|\sX},$ while $\P_\sX$ and $\Q_\sX$ can differ. Here, we address the covariate shift paradigm within the nonparametric regression framework. This means, we observe $n+m$ independent pairs $(X_1,Y_1),\ldots,(X_{n+m},Y_{n+m})\in [0,1]\times \mathbb{R}$ with 
\begin{align}
    X_i&\sim \P_\sX, \quad\quad\quad\quad \ \ \text{for} \ i=1,\ldots,n, \notag \\ 
    X_i&\sim \Q_\sX, \quad\quad\quad\quad \ \text{for} \  i=n+1,\ldots,n+m, \label{eq.nonp_regr_covariate_shift}\\
    Y_i&=f_0(X_i)+\varepsilon_i, \quad \, \text{for} \  i=1,\ldots,n+m,\notag
\end{align}
and independent noise variables $\varepsilon_1,\ldots,\varepsilon_{n+m}\sim \mathcal{N}(0,1).$

We now discuss estimation in this model, treating the cases $m=0$ and $m>0,$ separately. In both cases, the risk is the prediction error under the target distribution. For readability, we omit the subscript $\sX$ and write $\P$ and $\Q$ for $\P_\sX$ and $\Q_\sX$, respectively. Throughout the section, we assume global doubling, that is, $\P, \Q \in \mP_G(D)$ for some $D \geq 2.$

\subsection{Using LSE from source distribution to predict under target distribution}

As before, let $q$ denote the density of the target design distribution $\Q.$ Recall that we are considering the covariate shift model \eqref{eq.nonp_regr_covariate_shift} with $m=0$ and Lipschitz continuous regression functions. If the $n$ training data were generated from the target distribution $\Q_{\sX,\sY}$ instead, the classical empirical risk theory would lead to the standard nonparametric rate $n^{-2\beta/(2\beta+1)}$ with $\beta=1.$ More precisely, the statement would be that with probability tending to one as $n\to \infty,$ $\int_0^1 ( \wh f_n(x)-f_0(x))^2 q(x) \, dx \lesssim n^{-2/3}.$ 

For $n$ training samples from the source distribution $\P_{\sX,\sY}$, Theorem \ref{th.main_theorem} shows that $|\wh f_n(x)-f_0(x)|\leq K t_n^{\P}(x)$ for all $x,$ with high probability, where $t_n^{\P}$ denotes the spread function associated to the distribution $\P.$ This means that the prediction risk under the target marginal distribution $\Q_{\sX}$ with density $q$ is bounded by
\begin{align}
    \int_0^1 \big( \wh f_n(x)-f_0(x)\big)^2 q(x) \, dx
    \leq K^2 \int_0^1 t_n^{\P}(x)^2 q(x) \, dx,
    \label{eq.TL_1}
\end{align}
with probability converging to one as $n\to \infty.$ For a given source density $p,$ the main question is whether the right-hand side is of the order $n^{-2/3}$, up to $\log n$-factors. This would imply that there is no loss in terms of convergence rate (ignoring $\log n$-factors) due to the different sampling scheme. To get at least close to the $n^{-2/3}$-rate, some conditions on $p$ are needed. If $p$ is, for instance, zero on $[0,1/2],$ we have no information about the regression function $f$ on this interval and any estimator will be inconsistent on $[0,1/2]$. If we then try to predict with $\Q$ the uniform distribution, it is clear that $\int_0^1 ( \wh f_n(x)-f_0(x))^2 \, dx\geq \int_0^{1/2} ( \wh f_n(x)-f_0(x))^2 \, dx \gtrsim 1.$  

In the setting of sample size dependent densities $p_n$, Corollary \ref{cor.corollary_lip} shows that under the imposed conditions, there exists a constant $K'$ that does not depend on $n,$ such that
\begin{align*}
    \int_0^1 \big( \wh f_n(x)-f_0(x)\big)^2 q(x) \, dx
    \leq (K')^2 \Big(\frac{\log n}{n}\Big)^{2/3} \int_0^1 \frac{q(x)}{p_n(x)^{2/3}} \, dx,
\end{align*}
with probability tending to one as $n\to \infty.$ For instance, for the sequence of densities $p_n(x) := \phi_n(x) +16(1-\phi_n)\max(1/4-x,0,x-3/4)$ with $\phi_n=1\wedge n^{-1/4}\log(n),$ as considered in
\eqref{eq.Ex2_pn_def}, the right-hand side in the previous display is of the order $(\log n/n)^{2/3}\phi_n^{-2/3}\leq n^{-1/2}.$

For distributions satisfying the conditions of Theorem \ref{th.main_theorem}, we need to bound the more abstract integral $\int_0^1 t_n^{\P}(x)^2q(x) \, dx.$ The next result provides a different, sometimes simpler formulation. 

\begin{lemma}\label{lem.rate_rewritten}
In the same setting and for the same conditions as in Theorem \ref{th.main_theorem}, there exists a constant $K'',$ such that
\begin{align*}
    \int_0^1 \big( \wh f_n(x)-f_0(x)\big)^2 q(x) \, dx
    &\leq K'' \frac{\log n}{n} \int_0^1 \frac{\Q([x\pm t_n^{\P}(x)])}{t_n^{\P}(x)\P([x\pm t_n^{\P}(x)])}  \, dx \\
    &\leq 2^{1/3} K''\Big(\frac{\log n}{n}\Big)^{2/3}\|p\|_\infty^{1/3} \int_0^1 \frac{\Q([x\pm t_n^{\P}(x)])}{\P([x\pm t_n^{\P}(x)])}  \, dx,
\end{align*}
with probability tending to one as $n\to \infty.$
\end{lemma}

In \cite{kpotufe2021marginal}, a pair $(\P, \Q)$ is said to have transfer exponent $\gamma,$  if there exists a constant $0<C\leq 1,$ such that for all $0\leq x\leq 1$ and $0<\eta\leq 1,$ we have $\P([x \pm \eta]) \geq C\eta^\gamma\Q([x \pm \eta]).$ Combined with the previous lemma, we get for transfer exponent $\gamma,$
\begin{align*}
    \int_0^1 \big( \wh f_n(x)-f_0(x)\big)^2 q(x) \, dx
    &\leq 2^{1/3} \frac{K''}C \Big(\frac{\log n}{n}\Big)^{2/3}\|p\|_\infty^{1/3} \int_0^1 t_n^{\P}(x)^{-\gamma}  \, dx,
\end{align*}
with probability tending to one as $n\to \infty.$ Interestingly, the right-hand side does not depend on the target distribution $\Q$.

The next lemma provides an example of convergence rates.

\begin{lemma}\label{lem.TL_1}
Work in the nonparametric covariate shift model \eqref{eq.nonp_regr_covariate_shift} with $m=0.$ Let $\alpha>0.$ For source design density $p(x)=(\alpha+1)x^\alpha \mathds{1}(x\in [0,1])$ and uniform target design density $q(x)=\mathds{1}(x\in [0,1])$,
we have that 
\begin{align}\label{eq.main_ex_simple_TL}
    \int_0^1 \Big( \wh f_n(x)-f_0(x)\Big)^2 q(x) \, dx \lesssim (\log n)^{\mathds{1}(\alpha = 3/2)}\bigg[\Big(\frac{\log n}n\Big)^{3/(3 + \alpha)}\vee \Big(\frac{\log n}n\Big)^{2/3}\bigg]
\end{align}
with probability tending to one as $n\to \infty.$
\end{lemma}

The proof shows that the result follows for $0<\alpha\leq 1$ by a direct application of Lemma \ref{lem.rate_rewritten}. For general $\alpha>0,$ we prove the lemma by a more sophisticated analysis based on the bounds derived in Example \ref{ex.example2}.

The convergence rate is $(\log n/n)^{3/(\alpha + 3)}$ if $\alpha > 3/2$ and $(\log n/n)^{2/3}$ if $\alpha < 3/2$. For $\alpha=3/2$ an additional $\log n$-factor appears. This result shows that the low-density region near zero causes a slower convergence for $\alpha > 3/2.$

\subsection{Combining both samples to predict under the target distribution}

We now consider the nonparametric regression model under covariate shift \eqref{eq.nonp_regr_covariate_shift} with a second sample, that is, $m>0.$

In the first step, we construct an estimator combining the information from both samples. The main idea is to consider the LSEs for the first and second part of the sample and, for a given $x,$ pick the LSE with the smallest estimated local rate. For a proper definition of the estimator, some notation is required. Restricting to the first and second part of the sample, let $\wh f_n^{(1)}$ and $\wh f_m^{(2)}$ denote the corresponding LSEs taken over $1$-Lipschitz functions. Because the spread function is the local convergence rate, it is now natural to study $\wt f_{n,m}(x)=\wh f_n^{(1)}(x)\mathds{1}(t_n^{\P}(x)\leq t_m^{\Q}(x))+\wh f_m^{(2)}(x)\mathds{1}(t_n^{\P}(x)> t_m^{\Q}(x)).$ Because the spread functions are unknown, $\wt f_{n,m}(x)$ is not yet an estimator. Replacing $t_n^{\P}(x)$ and $t_m^{\Q}(x)$ by the estimators
\begin{align}
\begin{split}
    \wh t_{\sn}^{\SP}(x)&:= \inf \Big\{t: t^2\wh\P^n([x\pm t])\geq \frac{\log n}{n}\Big\}, \quad \wh\P^n([x\pm t]):=\frac{1}{n}\sum_{i=1}^n \mathds{1}(X_i\in [x\pm t]), \\
    \wh t_{\sm}^{\SQ}(x)&:= \inf \Big\{t: t^2\wh\Q^m([x\pm t])\geq \frac{\log m}{m}\Big\}, \quad \wh\Q^m([x\pm t]):=\frac{1}{m}\sum_{i=n+1}^{n+m} \mathds{1}(X_i\in [x\pm t]),
\end{split}
\end{align}
leads to the definition of our nonparametric regression estimator under covariate shift,
\begin{align}
    \wh f_{n,m}(x) 
    :=
    \wh f_n^{(1)}(x)\mathds{1}\big(\wh t_{\sn}^{\SP}(x)\leq \wh t_{\sm}^{\SQ}(x)\big)+\wh f_m^{(2)}(x)\mathds{1}(\wh t_{\sn}^{\SP}(x)> \wh t_{\sm}^{\SQ}(x)).
    \label{eq.cov_shift_nonp_est}
\end{align}

Let $p$ and $q$ be the respective Lebesgue densities of $\P$ and $\Q.$ We omit the dependence on $n,m$ and define by $\P_{\!f}$ the distribution of the data in model \eqref{eq.nonp_regr_covariate_shift}.

\begin{theorem}\label{thm.ub_cov_shift}
Consider the nonparametric regression model under covariate shift \eqref{eq.nonp_regr_covariate_shift}. Let $0<\delta<1$ and $D>0.$ If $\P,\Q \in \mP_G(D)$ and the estimator $\wh f_{n,m}$ is as in \eqref{eq.cov_shift_nonp_est}, then, for a sufficiently large constant $K,$
\begin{align*}
    \sup_{f_0\in \Lip(1-\delta)} \  \P_{\!f_0} \left(\sup_{x \in [0, 1]} \frac{|\wh f_{n,m}(x) - f_0(x)|}{t_n^{\P}(x) \wedge t_m^{\Q}(x)} > K\right) \to 0 \quad \text{as} \ n\to \infty \  \text{and} \ m\to \infty.
\end{align*}
\end{theorem}

The proof shows that to achieve the rate $t_n^{\P}(x) \wedge t_m^{\Q}(x),$ it is actually enough to estimate $t_n^{\P}(x)$ using $N$ data points $X_1,\ldots,X_N\sim \P,$ where $N$ is a sufficiently large number. Thus, instead of observing the full first dataset $(X_1,Y_1),\ldots,(X_n,Y_n)\sim \P$, the estimator only needs the LSE $\wh f_n^{(1)}$ and $N$ i.i.d. observations from the design distribution $\P.$

In the next step, we show that the rate $t_n^{\P}(x) \wedge t_m^{\Q}(x)$ is the local minimax rate. The design distributions $\P_\sX^n,\Q_\sX^n$ are allowed to depend on the sample size. The corresponding spread functions are denoted by $t_n^{\P}(x)$ and $t_m^{\Q}(x).$

\begin{theorem}
\label{th.lower_bound2}
Consider the nonparametric regression model under covariate shift \eqref{eq.nonp_regr_covariate_shift}. If $C_\infty$ is a positive constant, then there exists a positive constant $c,$ such that for any sufficiently large $n,$ and any sequences of design distribution $\P_\sX^n, \Q_\sX^n \in \mM$ with corresponding Lebesgue densities $p_n, q_n$ all upper bounded by $C_\infty,$ we have
\begin{align*}
    \inf_{\wh f_{n,m}} \, \sup_{f_0 \in \Lip(1)} \,  \P_{\! f_0}\bigg( 
    \sup_{x \in [0, 1]} \, \frac{|\wh f_{n,m}(x) - f_0(x)|}{t_n^{\P}(x) \wedge t_m^{\Q}(x)}
    \geq \frac {1}{12}\bigg) \geq c,
\end{align*}
where the infimum is taken over all estimators and $P_{f_0}$ is the distribution of the data in model \eqref{eq.nonp_regr_covariate_shift}.
\end{theorem}

Given the full dataset in model \eqref{eq.nonp_regr_covariate_shift}, an alternative procedure is to use the LSE over all data, that is, 
\begin{align*}
    \wh f_{n+m}\in \argmin_{f\in \Lip(1)}\sum_{i=1}^{n+m}\big(Y_i-f(X_i)\big)^2.
\end{align*}
Instead of analyzing this estimator in model \eqref{eq.nonp_regr_covariate_shift}, the risk can rather easily be controlled in the related model, where we observe $n+m$ i.i.d.\ observations $(X_1,Y_1),\ldots,(X_{n+m},Y_{n+m})$ with $X_i$ drawn from the mixture distribution $\wt{\P} := \tfrac m{m+n}\Q + \tfrac n{m+n} \P$ and $Y_i=f_0(X_i)+\varepsilon_i.$ In this model, we draw in average $n$ observations from $\P$ and $m$ observations from $\Q.$ Since $\mP_G(D)$ is convex, Theorem \ref{th.main_theorem} applies and, consequently, $t_{n+m}^{\wt \P}(x)$ is a local convergence rate. The spread function can be bounded as follows.

\begin{lemma}\label{lem.t_n_comparison_transfer} If $\wt \P:= \tfrac{n}{n+m}\P + \tfrac{m}{n+m}\Q,$ then,
\begin{align*}
     t_{n + m}^{\wt \P}(x) &\leq t_n^{\P}(x)\sqrt{\tfrac{\log(m+n)}{\log n}}
     \wedge 
     t_m^{\Q}(x)\sqrt{\tfrac{\log(m+n)}{\log m}}.
\end{align*}
If there are positive constants $C,\kappa$ such that $\sup_x t_n^{\P}(x)\leq Cn^{-\kappa}$ for all $n,$ then, there exists a constant $C',$ such that 
\begin{align*}
    t_{n + m}^{\wt \P}(x)&\leq C'\big(t_n^{\P}(x)\wedge t_m^{\Q}(x)\big),
    \quad \text{for all} \ n\geq m>2.
\end{align*}
\end{lemma}
One can see that the rate is at most a $\log$-factor larger than the local minimax rate $t_n^{\P}(x)\wedge t_m^{\Q}(x).$ Moreover, this additional $\log$-factor can be avoided in the relevant regime where the local rate $t_n^{\ \P_{\sX}}(x)$ decays with some polynomial rate uniformly over $[0,1].$ 

We now return to our leading example with source density $p(x)=(\alpha+1)x^{\alpha+1}\mathds{1}(x\in [0,1])$ and target density $q(x)=\mathds{1}(x\in [0,1]).$ Lemma \ref{lem.example_bounded_density} and Lemma \ref{lem.example_growing_density} show that if $\alpha>0,$ then, the assumptions of Theorem \ref{thm.ub_cov_shift} are satisfied and the local convergence rate of the combined estimator $\wh f_{n,m}$ is $t_n^{\P}(x)\wedge t_m^{\Q}(x).$ From \eqref{eq.main_ex_simple_TL}, we see that as long as $\alpha<3/2,$ the first sample is enough to achieve the rate $(\log n/n)^{2/3}.$ We therefore focus on the regime $3/2<\alpha.$

\begin{lemma}\label{lem.cov_shift_rate_ex}
Consider the nonparametric regression model under covariate shift \eqref{eq.nonp_regr_covariate_shift}. For $3/2< \alpha,$ let $\P$ be the distribution with Lebesgue density $p(x)=(\alpha+1)x^{\alpha+1}\mathds{1}(x\in [0,1])$ and $\Q$ be the uniform distribution on $[0,1]$. If $n^{3/(3+\alpha)}\log^{\alpha/(3+\alpha)} n \ll m 
\leq n,$ then, we have that for any $f_0\in \Lip(1-\delta),$ 
\begin{align}
    \int_0^1 \big(\wh f_{n,m}(x)-f_0(x)\big)^2 q(x) \, dx\lesssim \Big(\frac{\log n}{m}\Big)^{2/3}\Big(\frac{m}{n}\Big)^{1/{\alpha}},
\end{align}
with probability converging to one as $n\to \infty.$
\end{lemma}

We have $m\lesssim n^{3/(\alpha + 3)}(\log n)^{\alpha/(\alpha + 3)}$ if and only if $(\log n/m)^{2/3}(m/n)^{1/\alpha}\lesssim (\log n/n)^{3/(3+\alpha)}.$ Since the right-hand side is the rate obtained without a second sample in  \eqref{eq.main_ex_simple_TL}, the additional data $(X_{n+1},Y_{n+1}),\ldots,(X_{n+m},Y_{n+m})$ improve the convergence rate if $m\gg  n^{3/(3+\alpha)}(\log n)^{\alpha/(3+\alpha)}.$

\section{Discussion}
\label{sec:discussion}

\subsection{A brief review of convergence results for the least squares estimator in nonparametric regression}
\label{sec.review_LSE}

The standard strategy to derive convergence rates with respect to (empirical) $L^2$-type losses is based on empirical process theory and covering bounds. The field is well-developed, see e.g. \cite{MR1056343, van1996weak, MR1920390, MR2329442, wainwrighthighdimstat}. At the same time, it remains a topic of active research. A recent advance is to establish convergence rates of the LSE under heavy-tailed noise distributions \cite{MR3953452, 2019arXiv190902088K}.

Some convergence results are with respect to the squared Hellinger distance, see for instance \cite{MR1739079, MR1240719}. This is slightly weaker but essentially the same as convergence with respect to the prediction risk $\E[(\wh f_n(X)-f_0(X))^2].$ To see this, recall that for two probability measures $P,Q$ defined on the same measurable space, the squared Hellinger distance is defined as $H^2(P,Q)= \tfrac 12 \int (\sqrt{dP}-\sqrt{dQ})^2$ (some authors do not use the factor $1/2$). Denote by $\Q_f$ the distribution of $(X_1,Y_1)$ in the nonparametric regression model \eqref{eq.mod} with regression function $f.$ It can be shown that 
\begin{align*}
    H^2(\Q_f,\Q_g) = 1 -\int e^{-\tfrac 18 (f(x)-g(x))^2} p(x) \, dx.
\end{align*}
In view of the formula $1-e^{-u}\leq u,$ it follows that $H^2(\Q_f,\Q_g)\leq \tfrac 18 \E[(f(X)-g(X))^2].$ Thus, the squared Hellinger loss is weaker than the squared prediction loss.

Concerning estimation rates, the LSE achieves the rate $n^{-2\beta/(2\beta+d)}\vee n^{-\beta/d}$ over balls of $\beta$-smooth H\"older functions. To see this, observe that if $\mF$ denotes a H\"older ball and $\|g\|_n:=(\tfrac 1n \sum_{i=1}^n g(X_i)^2)^{1/2}$ is the empirical $L^2$ norm, the metric entropy is $\log N\big(r, \mF, \|\cdot\|_n\big)\lesssim r^{-d/\beta},$ see Corollary 2.7.2 in \cite{van1996weak}. Any solution $\delta^2$ of the inequality $\int_{\delta^2}^\delta \sqrt{\log N\big(r, \mF, \|\cdot\|_n\big)} \, dr\lesssim \delta^2 \sqrt{n}$ is then a rate for the LSE, see Corollary 13.1 in \cite{wainwrighthighdimstat}. It is now straightforward to check that this yields the convergence rate $\delta^2 \asymp n^{-2\beta/(2\beta+d)}\vee n^{-\beta/d}.$ While $n^{-2\beta/(2\beta+d)}$ is the optimal convergence rate, Theorem 4 in \cite{MR1240719} shows that for $d=1$ and a specifically designed subset of the H\"older ball with index $\beta<1/2,$ the LSE cannot achieve a faster rate than $n^{-\beta/2}$ (up to a possibly non-optimal logarithmic factor in $n$). The (sub)optimality of the LSE over H\"older balls in the non-Donsker regime $\beta <d/2$ remains an open problem. Considering shape-constrained problems, \cite{convkur} shows that for different classes of convex functions, the LSE is suboptimal for dimensions $d \geq 5,$ while \cite{MR4338378} proves that the LSE can still achieve near-optimal convergence rates in the non-Donsker regime.

To the best of our knowledge, the only sup-norm rate result for the LSE is \cite{Saumard}. In this work, the LSE is studied for $\mF$ the linear space spanned by a nearly orthogonal function system. For this setting, the LSE has an explicit representation that can be exploited to prove sup-norm rates. 
 
For a number of other settings, a more explicit characterization of the LSE is available from which local properties can be inferred. \cite{MR1429931} shows this for least squares penalized regression with total variation penalties. In this case, the LSE can be linked to splines, and this is exploited in their Proposition 8 to provide a local characterization of the LSE.

More explicit characterizations of the LSE are also available for several shape-constrained estimation problems. In isotonic regression, the regression function is non-decreasing, and the LSE admits an explicit expression. Let $(X_{(1)}, Y_{(1)}), \dots, (X_{(n)}, Y_{(n)})$ be a reordered version of the dataset such that $X_{(1)} \leq X_{(2)} \leq \dots \leq X_{(n)}$ and for all $0 \leq k \leq n$, define the $k$-th partial sum of $Y$ as $S_k := \sum_{i = 1}^k Y_{(i)}$. Additionally, for $x \in [0, 1]$, set $k_-(x) := \max \{0 \leq k \leq n: X_{(k)} < x\}$ and $k_+(x) := \min\{0 \leq k \leq n: X_{(k)} \geq x\}$. The LSE for isotonic regression is then piecewise constant on $[0, 1]$ and is given by
\begin{align*}
    \wh f_n(x) := \min_{k_+(x) \leq i \leq n} \ \max_{0 \leq j \leq k_-(x)}\frac{S_i - S_j}{i - j},
\end{align*}
see for instance \cite{brunk1955maximum, Brunk_1958, wright_isotonic_1981, hanson1973consistency} and Lemma 2.1 in \cite{soloff_isotonic_regression_lse}. Moreover, the pointwise limiting distributional results are available for the isotonic LSE. For the purposes of this discussion, we provide the following, adapted version of the main theorem in \cite{wright_isotonic_1981}.
\begin{theorem}
\label{th.adapted_wright}
Let $A >0, \alpha > 1/2$ and $x_0 \in [0, 1]$ with $p(x_0)>0.$ Suppose that there exists an open neighbourhood $U$ of $0$ and a continuous function $g\colon U \mapsto \RR$ such that $\lim_{x \to 0} g(x)/x^\alpha = 0$ and for all $x \in U$, $|f_0(x) - f_0(x_0)| = A|x - x_0|^\alpha+g(x - x_0)$. Then
\begin{align*}
    \Big(\frac{\alpha + 1}{A}\Big)^{1/(2\alpha + 1)} \big(np(x_0)\big)^{\tfrac{\alpha}{2\alpha + 1}}\big(\wh f_n(x_0) - f_0(x_0)\big) \overset{d}{\longrightarrow} Z,
\end{align*}
with $Z$ a random variable distributed as the slope at zero of the greatest convex minorant of $W_t + |t|^{\alpha + 1}$ and $(W_t)_t$ a two-sided Brownian motion.
\end{theorem}

If $f_0$ is Lipschitz continuous, then $\alpha = 1$ and $Z$ is known to follow Chernoff's distribution. Assuming moreover that $p(x_0) > 0$, leads to $|\wh f_n(x_0) - f(x_0)| \asymp (np(x_0))^{-1/3}.$ This agrees with the local rate $t_n(x_0) \asymp (\log n/(np(x_0)))^{1/3}$ obtained in Corollary \ref{cor.corollary_lip} up to the $\log n$-factor that emerges due to the uniformity of the local rates.

For isotonic regression in $d$ dimensions, the recent article \cite{MR3988762} shows that the LSE achieves the minimax estimation rate $n^{- \min(2/(d+2),1/d)}$ up to $\log n$-factors. For $d\geq 3,$ is it known that $\log N(r,\mF,\|\cdot\|_2)\asymp r^{-2(d-1)}.$ Since for uniform design, the norms $\|\cdot\|_2$ and $\|\cdot\|_n$ are close, the standard approach to derive convergence rates via the entropy integral is then expected to yield no convergence rate faster than $n^{-1/(2d-2)}$. Since this rate is slower than the actual convergence rate of the LSE, this provides another instance where the entropy integral approach is suboptimal. Interestingly, \cite{MR3988762} proves, moreover, that if the isotonic function is piecewise constant with $k$ pieces, the LSE adapts to the number of pieces and attains the optimal adaptive rate $(k/n)^{\min(1,2/d)}$ up to $\log n$-factors. More on adaptation and the pointwise behaviour of the LSE in isotonic regression and other shape-constrained estimation problems can be found in the survey article \cite{MR3881209}.

For penalized LSEs, an alternative to the concentration bounds in empirical process methods is the recently developed proof strategy based on the small ball method, \cite{pmlr-v35-mendelson14, 8200211, MR3782379, JMLR:v18:16-422, MR4474492}.

\subsection{Related work on transfer learning}

\label{sec:discussion_TL}

From a theory perspective, the key problem in TL is to quantify the information that can be carried over from one task to another \cite{caruana1997multitask, micchelli2004kernels, baxter1997bayesian, ben2003exploiting}. Among the mathematical statistics articles, \cite{sugiyama2007direct} proposes unbiased model selection procedures and \cite{shimodaira_consistency} considers re-weighting to improve the predictive power of models based on likelihood maximization. The nonparametric TL literature mainly focuses on classification. Minimax rates are derived under posterior drift by \cite{cai2019transfer}, under covariate shift by \cite{kpotufe2021marginal}, and in a general setting by \cite{MR4352543}.

The closest related work is the recent preprint \cite{pathak2022new}. While we consider the LSE, this article proves minimax convergence rates for the Nadaraya-Watson estimator in nonparametric regression under covariate shift. The proofs differ, as one can use the closed-form formula for the Nadaraya-Watson estimator (NW). The rates are proven uniformly over two different sets of distribution pairs. Let $\rho_\eta(\P_\sX, \Q_\sX) := \int_0^1 \P_\sX([x \pm \eta])^{-1}\, d\Q_\sX(x),\ \gamma, C \geq 1$ and denote by $\mS(\gamma, C)$ the set of all pairs $(\P_\sX, \Q_\sX),$ such that $\sup_{\eta \in (0, 1]} \eta^\gamma \rho_\eta(\P_\sX, \Q_\sX) \leq C$.

To discuss the connection of this class to our approach, observe that, in our framework, bounding the prediction risk of the LSE with regards to some target distribution amounts to bounding the quantity $\int_0^1 t_n^{\P_\sX}(x)^2\, d\Q_\sX(x).$ Using the definition of the spread function and assuming $(\P_\sX, \Q_\sX) \in S(\gamma, C)$, we obtain
\begin{align*}
    \int_0^1t_n^{\P_\sX}(x)^2\, d\Q_\sX(x) = \frac{\log n}n\int_0^1 \frac{d\Q_\sX(x)}{\P_\sX([x \pm t_n^{\P_\sX}(x)])}
    \leq \frac{\log n}n \left(\inf_{x \in [0, 1]}t_n^{\P_\sX}(x)\right)^{-\gamma}.
\end{align*}

In some cases, faster rates for the prediction error can be obtained for the LSE using our results. As an example, consider again the case that the source density is $p(x)=(\alpha+1)x^\alpha \mathds{1}(x\in [0,1])$ and the target distribution is uniform on $[0,1].$ For the nonparametric regression model with covariate shift \eqref{eq.nonp_regr_covariate_shift}, Lemma \ref{lem.TL_1} shows that for the LSE $\wh f_n,$
\begin{align}\label{eq.rate_recall}
        \int_0^1 \Big( \wh f_n(x)-f_0(x)\Big)^2 q(x) \, dx \lesssim (\log n)^{\mathds{1}(\alpha = 3/2)}\bigg[\Big(\frac{\log n}n\Big)^{3/(3 + \alpha)}\vee \Big(\frac{\log n}n\Big)^{2/3}\bigg]
\end{align}
with probability tending to one as $n\to \infty.$ For the Nadaraya-Watson estimator, Lemma 31 in Appendix E of the supplement \cite{supplement} shows that if $\alpha \geq 1,$ there exists a $C> 0,$ such that for any $\varepsilon \in (0, \alpha)$, $(\P_\sX,\Q_\sX)\in \mS(\alpha, C)\setminus \mS(\alpha - \varepsilon, C).$ According to Corollary 1 in \cite{pathak2022new}, for $\wh f_{NW}$ the Nadaraya-Watson estimator with suitable bandwidth choice, we then have for any $\alpha \geq 1,$
\begin{align*}
    E\Big[\int_0^1\big(\wh f_{NW}(x) - f(x)\big)^2q(x) \, dx\Big] \lesssim n^{-2/(2 + \alpha)}.
\end{align*}
This is a slower rate than \eqref{eq.rate_recall}. The convergence rate of the LSE becomes slower than $(\log n/n)^{2/3}$ for $\alpha > 3/2$, while for the Nadaraya-Watson estimator, this happens already for $\alpha > 1$. We believe that the loss in the rate is due to the lack of local adaptivity of kernel smoothing with fixed bandwidth, as discussed in Section \ref{sec.main}.

\subsection{Extensions and open problems}

For machine learning applications, we are, of course, interested in multivariate nonparametric regression with $d$-dimensional design vectors $X_i$ and arbitrary H\"older smoothness $\beta.$ To extend the result, the definition of the spread function has to be adjusted. If $\beta>d/2,$ the LSE converges with the rate $n^{-2\beta/(2\beta+d)}$ (see the discussion above) and we believe that the local rate $t_n$ is now determined by the solution of the equation
\begin{align}
    t_n(x)^{2} \P_\sX\big(y:|x-y|_\infty\leq t_n(x)^{1/\beta}\big)=\frac{\log n}{n},
\end{align}
where $|v|_\infty$ denotes the largest absolute value of the components of the vector $v.$ In the case $d=1,$ this coincides with the minimax rate found in \cite{gaiffas2009uniform}. Observe moreover that for the uniform design distribution, $\P_\sX(y:|x-y|_\infty\leq t_n(x)^{1/\beta})\asymp t_n(x)^{d/\beta}$ and we obtain $t_n(x) \asymp (\log n/n)^{\beta/(2\beta+d)}.$ To show that $t_n$ is a lower bound on the local convergence rate, we believe that the lower bound in Theorem \ref{th.lower_bound} can be generalized without too much additional effort. But the upper bound is considerably harder than the case $\beta=d=1$ we considered in this work. The main reason is that the local perturbation in the proof also needs to be $\beta>1$ smooth, thus a piecewise approach as in \eqref{eq.local_pert_define} does not work anymore, and one needs to have tight control of the derivatives of the LSE. 

It is unclear what the local convergence rate is in the regime  $\beta < d/2.$

Throughout this work, we assumed data from the nonparametric regression model 
$Y_i = f_0(X_i)+ \eps_i$ with independent noise variables $\eps_i \sim \mN(0,1).$ If instead, we have $\eps_i \sim \mN(0,\sigma^2),$ the spread function should be determined by 
\begin{align*}
    t_n(x)^2 \, \P_\sX\big(\big[x\pm t_n(x)\big]\big) = \frac{\sigma^2\log n}n.
\end{align*}
That this is the right scaling has already been observed in the article \cite{gaiffas2009uniform}.

In Theorem \ref{th.main_theorem}, we assume that the regression function is $(1-\delta)$-Lipschitz for some positive $\delta$. Another interesting question is whether the local convergence result can be extended to a regression function that is itself $1$-Lipschitz. Again, the main complication arises in constructing the local perturbation in Lemma \ref{lem.g_function}. 
One might also wonder whether the same rates can be achieved if, instead of the global minimizer, we take any estimator $\wh f$ satisfying 
\begin{align*}
    \sum_{i=1}^n \big(Y_i-\wh f(X_i)\big)^2
    \leq \inf_{f\in \Lip(1-\delta)}\sum_{i=1}^n \big(Y_i-f(X_i)\big)^2+\tau_n
\end{align*}
for a pre-defined rate $\tau_n.$ In particular, it is of interest to determine the largest $\tau_n$ such that the optimal local rates can still be obtained.

The proposed approach might be used to derive theoretical guarantees for transfer learning based on deep ReLU networks, extending the earlier analysis of the prediction risk \cite{schmidt2020nonparametric, bauer2019deep, kohler2021rate}. Here we briefly illustrate this using shallow ReLU networks of the form 
\begin{align*}
    f(x)=\sum_{i=1}^N a_i(w_i x-v_i)_+ \quad \text{with} \ (u)_+:=\max(u,0) \quad \text{and} \ a_i,w_i,v_i \in \mathbb{R}.
\end{align*}
Denote by $\ReLU_N(1-\delta)$ the function class of all such shallow ReLU networks that are moreover $(1-\delta)$-Lipschitz. 

\begin{lemma}
\label{lem.ReLU_minim}
If $N\geq n-1\geq 1,$ then, 
\begin{align*}
    \wh f_n\in \argmin_{f\in \, \ReLU_N(1-\delta)} \sum_{i=1}^n \big(Y_i-f(X_i)\big)^2,
\end{align*}
implies that $\wh f_n$ is also a minimizer over all $\Lip(1-\delta)$ functions, that is,
\begin{align*}
    \wh f_n\in \argmin_{f\in \, \Lip(1-\delta)} \sum_{i=1}^n \big(Y_i-f(X_i)\big)^2.
\end{align*}
\end{lemma}

The result shows that a global minimizer over all ReLU networks in the class $\ReLU(1-\delta)$ is also an empirical risk minimizer over all $\Lip(1-\delta)$-functions. In particular, this means that all the results derived in this paper can be immediately applied, leading to local convergence rates and theoretical guarantees in the case of transfer learning.

Another possible future direction is to use the refined analysis and the local convergence of the LSE to prove distributional properties similar to those established for the least squares procedure under shape constraints. See also the discussion in Section \ref{sec.review_LSE}.

\section{Proof of the local convergence rate for the LSE}

\label{sec.proof}
We now describe the construction of the local perturbation and give the proof of Theorem \ref{th.main_theorem}.

\subsection*{Preliminary: concentration of histogram}

For sufficiently large sample size $n$, we can find an integer sequence $(N_n)_n$ satisfying $$1 \leq \frac{N_n}{16} \sqrt{\frac{\log n}n} \leq 2.$$ For discretization step size
\begin{align}
    \Delta_{n} := \frac 1{N_n}
    \label{eq.Deltan_def}
\end{align}
we show that $n^{-1}\sum_{i=1}^n \mathds{1}(X_i\in [j\Delta_{n},k\Delta_{n}])$ concentrates around its expectation $\int_{j\Delta_{n}}^{k\Delta_{n}} p(u)\, du.$ For this purpose, we first recall the classical Bernstein inequality for Bernoulli random variables.

\begin{lemma}[Bernstein inequality]
\label{lem.prob_bound}
Let $p \in [0, 1]$ and $V_1, \dots, V_n$ be $n$ independent Bernoulli variables with success probability $p$, then,
\begin{align*}
    \P\bigg(\frac 12 p \leq \frac 1n \sum_{i=1}^n V_i \leq \frac 32 p\bigg) \geq 1 - 2\exp\left(-\frac{np}{10}\right).
\end{align*}
\end{lemma}

\begin{proof}
Let $U_i = V_i - p.$ We have, $|U_i|\leq 1,$ $\E[U_i] = 0$ and $\E[U_i^2] = \Var(V_i)=p(1-p)\leq p.$ By Bernstein's inequality,
\begin{align*}
    \P\left(\Big|\sum_{i=1}^n U_i\Big| \geq \frac n2 p\right) &\leq 2\exp\left(-\frac{n^2p^2/8}{np + np/6}\right)
    \leq 2\exp\left(-\frac{np}{8 + 4/3}\right)
    < 2\exp\left(-\frac{np}{10}\right).
\end{align*}
\end{proof}

Set $\ol p_{j,k} := \int_{[j\Delta_n, k\Delta_n]} p(u) \, du$ and define $\Gamma_n(\alpha)$ as the event
\begin{align}
    \Gamma_n(\alpha) := \mathlarger{\bigcup}_{\substack{j,k = 1,\dots, N_n\\ \ol p_{j, k}\geq \alpha\frac{\log^2 n}{n}}}\!\left\{X_1,\ldots,X_n:\Big|\frac 1n \sum_{i=1}^n \mathds{1}(X_i \in [j\Delta_{n}, k\Delta_{n}]) -  \ol p_{j, k}\Big| > \frac {\ol p_{j, k}}2\right\}.
    \label{eq.Gamman_def}
\end{align}

Roughly speaking, this set consists of all samples for which the histogram does not concentrate well around its expectation.

\begin{lemma}
\label{lem.prob.to.zero}
If $\alpha >0,$ then, the probability of the event $\Gamma_n(\alpha)$ vanishes as the sample size grows. More precisely,
\begin{gather*}
    \P\left(\Gamma_n(\alpha)\right)
    \leq \frac{1024 n}{\log n}\exp\left(-\frac{\alpha \log^2n}{10}\right) \to 0 \quad \text{as } n\to \infty.
\end{gather*}
\end{lemma}

\begin{proof}
Use the union bound and apply Lemma \ref{lem.prob_bound}. Since $N_n \leq 32\sqrt{n/\log n},$ we obtain the inequality. The convergence to zero follows from $\alpha > 0.$
\end{proof}

The previous result allows us to work on the event $\Gamma_n(\alpha)^c.$ On this event, the random quantity $n^{-1}\sum_{i=1}^n\mathds{1}(X_i \in [j\Delta_n, k\Delta_n])$ is the same as its expectation $\int_{j\Delta_n}^{k\Delta_n} p(u)\, du$
up to a factor two. In particular, we will apply this to random integers $j, k$ depending on the sample $ (X_1,Y_1),\ldots,(X_n,Y_n)$. We frequently use that for an $X$ that is independent of the data,
\begin{align}
    \P_\sX(X\in A)=\P\big(X\in A \, | \,(X_1,Y_1),\ldots,(X_n,Y_n)\big).
    \label{eq.PX_cond}
\end{align}
The proof outline in Section \ref{sec.proof_strategy} assumes that there exists a function $g$ lying at all data points between $\wh f_n$ and $f_0$ in the sense that $(\wh f_n(X_i)-g(X_i) )(g(X_i)-f_0(X_i)) \geq 0$ for all $i=1,\ldots,n.$ The next lemma guarantees the existence of a sufficiently large local perturbation $g$ of $\wh f_n$ with this property. This will allow us later to carry out the proof strategy sketched in Section \ref{sec.proof_strategy}.

\begin{lemma}[Construction of a local perturbation]
\label{lem.g_function}
Let $\psi \in \Lip(1)$ and $f\in \Lip(1-\delta)$. Define $x^* \in \argmax_{x \in [0, 1]} (\psi(x) - f(x))/t_n(x)$ and $\tilde x \in \argmax_{x \in [0, 1]}\psi(x) - f(x) - \tfrac{\delta}2|x - x^*|.$ Assume the existence of some $K >0,$ such that
\begin{align*}
    \frac{\psi(x^*) - f(x^*)}{t_n(x^*)} \geq K,
\end{align*}
and set 
\begin{align}
    s_n := 2Kt_n(\tilde x) \wedge \Big(2Kt_n(x^*) + \frac{\delta}{2}|x^* - \tilde x|\Big).
    \label{eq.sn_def}
\end{align}
Then there exists a function $g_n$ and two real numbers $0\leq x_\ell\leq x_u\leq 1,$ such that
\begin{enumerate}
    \item [$(i)$] $g_n \in \Lip(1)$ and $\supp(\psi - g_n) = [x_\ell, x_u]$. \label{i}
    \item [$(ii)$] $f \leq g_n \leq \psi$ on $[x_\ell,x_u]$. \label{ii}
    \item [$(iii)$] $\tilde x - s_n/\delta \leq x_\ell \leq \big(\tilde x - s_n/4\big)\vee 0 \leq \big(\tilde x + s_n/4\big) \wedge 1 \leq x_u \leq \tilde x + s_n/\delta$. \label{iii}
    \item [$(iv)$] the inequality $\psi(x) - g_n(x)\geq s_n/4$ holds for all $x\in [\tilde x \pm s_n/8]\cap [0, 1].$ \label{iv}
\end{enumerate}
\end{lemma}

\begin{proof}
We construct a function $g_n$ satisfying all claimed properties. The construction requires several steps and can be understood best through the visualization in Figure \ref{fig.perturbation}. Recall that $\tilde x \in \argmax_{x \in [0, 1]} \psi(x) - f(x) - \tfrac{\delta}2|x - x^*|.$ Hence, $\psi(\tilde x) - f(\tilde x) - \tfrac{\delta}2|\tilde x - x^*|\geq \psi(x^*) - f(x^*).$ By assumption, we have $\psi(x^*) - f(x^*) > t_n(x^*)K$ and thus
\begin{align*}
    \psi(\tilde x) - f(\tilde x) > t_n(x^*)K + \frac{\delta}2|\tilde x - x^*| 
    \geq \frac{s_n}2.
\end{align*}
Define the function
\begin{align*}
    h_n(x) := \psi(\tilde x) - f(\tilde x) + \delta|x - \tilde x| + f(x) -\frac{s_n}2.
\end{align*}
Since $f \in \Lip(1 - \delta)$ and $\delta|\cdot| \in \Lip(\delta)$, we have that $h_n \in \Lip(1).$ By construction, $h_n(\tilde x)=\psi(\tilde x) -s_n/2 < \psi(\tilde x).$ Denote by $x_\ell$ the largest $x$ below $\tilde x$ satisfying $h_n(x)=\psi(x).$ If no such $x$ exists, set $x_\ell:=0.$ Similarly, define $x_u$ as the smallest $x$ above $\tilde x$ satisfying $h_n(x)=\psi(x)$ and set $x_u:=1$ if this does not exist. Define
\begin{align}
    g_n(x) := \psi(x) \mathds{1}(x \in [0, 1]\setminus [x_\ell, x_u]) + h_n(x)\mathds{1}(x \in [x_\ell, x_u]).
    \label{eq.local_pert_define}
\end{align}
By construction, $g_n \in \Lip(1)$ and $\supp(\psi - g_n) = [x_\ell, x_u]$. Thus $(i)$ holds. Also, $(ii)$ follows directly from the inequalities above. 

We now prove $(iii)$. Applying triangle inequality yields
\begin{align*}
    h_n(x) - f(x) &= \psi(\tilde x) - f(\tilde x) - \frac{\delta}2|\tilde x - x^*| + \frac \delta 2|\tilde x - x^*| + \delta|x - \tilde x|-\frac 12 s_n\\
    &\geq \psi(x) - f(x) - \frac \delta 2|x - x^*| + \frac \delta 2|\tilde x - x^*|  + \delta |x - \tilde x|-\frac 12 s_n\\
    &\geq \psi(x) - f(x) + \frac \delta 2 |x - \tilde x|-\frac 12 s_n.
\end{align*}
From the last inequality, we deduce that for any $x \in [0, 1]$ with $|x - \tilde x| \geq s_n/\delta$, we have $h_n(x) - f(x) \geq \psi(x) - f(x).$ Thus, 
\begin{align}
\tilde x - \frac{s_n}\delta \leq x_\ell
\leq x_u \leq \tilde x + \frac{s_n}{\delta},\label{eq.outer_bounds}
\end{align}
proving the first and last inequality in $(iii)$.

To prove the remaining inequalities in $(iii)$, we use $\psi - f \in \Lip(2 - \delta)$ to deduce $\psi(x) - f(x) \geq \psi(\tilde x) - f(\tilde x) - (2 - \delta)|x - \tilde x|.$ By definition, we have moreover $h_n(x) - f(x) = \psi(\tilde x) - f(\tilde x)  + \delta|x - \tilde x|-s_n/2$ and therefore $\psi(x) - f(x) +2|x - \tilde x| \geq h_n(x) - f(x) +s_n/2,$ which can be rewritten into
\begin{align}
    \psi(x) - h_n(x) \geq \frac{s_n}2 - 2|x - \tilde x|. \label{ineq.ineq}
\end{align}
The right-hand side of this inequality is $>0$ for all $x \in [\tilde x \pm s_n/4]\cap [0,1]$. The definition of $x_\ell$ and $x_u$ implies then that $x_\ell \leq 0 \vee \big(\tilde x - s_n/4 \big)\leq 1 \wedge \big(\tilde x + s_n/4 \big) \leq x_u.$ 

We now establish $(iv)$. One can use the lower bound from Equation \eqref{ineq.ineq} to obtain that for any $x \in [\tilde x \pm s_n/8] \cap [0,1]$, 
\begin{align*}
    \psi(x) - f(x) &\geq \frac{s_n}2 - 2|x-\tilde x|+ g_n(x) - f(x) \geq \frac{s_n}4 + g_n(x) - f(x) \geq \frac{s_n}4,
\end{align*}
applying $(ii)$ for the last inequality. This proves $(iv).$
\end{proof}

\begin{figure}[H]

  \begin{minipage}[c]{0.49\textwidth}
    \vspace{-.5\baselineskip}
    \includegraphics[width=.9\textwidth]{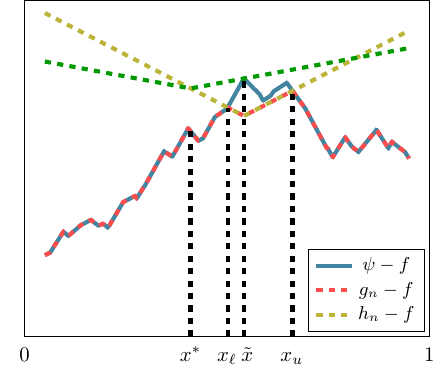}
    \vspace{-.5\baselineskip}
  \end{minipage}\hfill
  \begin{minipage}[c]{0.49\textwidth}
    \vspace{-.5\baselineskip}
    \caption{(Construction of the local perturbation) The variables $x^*$ and $\tilde x$ are as in Lemma \ref{lem.g_function}. From the construction of $\tilde x,$ we know that the function $\psi-f$ (plotted in blue) cannot lie above the green dashed curve with slope $\delta/2.$ The yellow function is $h_n - f.$ Since this function has slope $\delta$, it will hit the green dashed curve in a neighbourhood of $\tilde x$. This also implies that $h_n - f$ intersects for the first time with $\psi-f$ (blue curve) in this neighbourhood and provides us with control for the hitting points $x_\ell$ and $x_u.$ The (shifted) perturbation $g_n - f$ is given by the red curve.} \label{fig.perturbation}
    \vspace{-.5\baselineskip}
  \end{minipage}
\end{figure}

\begin{lemma}\label{lem.sn2PX_lb}
For given $K>1/2$ and $0<\delta<1,$ let $s_n$ and $\tilde x$ be as defined in Lemma \ref{lem.g_function}. Moreover, let $\P_\sX\in \mP_n(D)$ for some $D\geq 2.$ If $n \geq \exp(4K^2\vee 36(1 \vee\log^2D)))$ and $0 < c\delta \leq 2$ then,
\begin{enumerate}
    \item[$(i)$]$P_\sX\big([\tilde x\pm cs_n]\big)> D^{-\lceil \log_2(1/(\delta c))\rceil-1} \frac{\log^2 n}{n},$
    \item[$(ii)$] $s_n^2\P_{\sX}\big([\tilde x \pm cs_n]\big) \geq D^{-\lceil\log_2(1/(\delta c))\rceil - 1}4K^2\frac{\log n}n.$
\end{enumerate}
\end{lemma}

\begin{proof}
Recall that $s_n = 2Kt_n(\tilde x) \wedge (2Kt_n(x^*) + \delta|x^* - \tilde x|/2).$ Since $n \geq \exp(4K^2), \ s_n \leq 2Kt_n(\tilde x) \leq \sqrt{\log n} \,t_n(\tilde x) \leq \sqrt{\log n}\sup_{x \in [0, 1]} t_n(x).$ This allows to apply now the local doubling property of $\P_\sX$ to intervals of length up to $s_n$.

By assumption $K\geq 1/4$ and $\delta<1.$ Hence $2Kt_n(\tilde x) \geq \delta t_n(\tilde x)/2.$ Moreover, using the fact that $t_n$ is $1$-Lipschitz from Lemma 27 (Appendix A of the supplement \cite{supplement}), $2Kt_n(x^*) + \delta|x^* - \tilde x|/2 \geq \frac \delta{2}\left[ t_n(x^*) + |x^* - \tilde x|\right] \geq \delta t_n(\tilde x)/2.$
Combining the two previous bounds, we obtain $cs_n \geq \delta c t_n(\tilde x)/2$. Using the latter and applying \eqref{eq.LDP} in total $\lceil \log_2(1/(\delta c))\rceil+1$ times to increase the constant $\delta c/2$ to $\delta c2^{\lceil \log_2(1/(\delta c))\rceil} \in [1,2],$ we find whenever $\delta c \leq 2,$
\begin{align*}
    \P_\sX\big(\big[\tilde x\pm cs_n \big]\big)&\geq 
    \P_\sX\big(\big[\tilde x \pm \delta ct_n(\tilde x)/2\big]\big)\\
    &\geq D^{-\lceil \log_2(1/(\delta c))\rceil-1} 
    \P_\sX\big(\big[\tilde x \pm \delta c2^{\lceil \log_2(1/(\delta c))\rceil} t_n(\tilde x)\big]\big) \\
    &\geq D^{-\lceil \log_2(1/(\delta c))\rceil-1} 
    \P_\sX\big(\big[\tilde x \pm t_n(\tilde x)\big]\big).
\end{align*}
Applying Remark 2 (Appendix A of the supplement \cite{supplement}) completes the proof of $(i)$. To prove the second claim, we once again lower bound $cs_n$. 

We first consider the case $s_n = 2Kt_n(\tilde x).$ If $c\geq1,$ then, we get $cs_n \geq s_n \geq 2Kt_n(\tilde x) \geq t_n(\tilde x)$ and
\begin{align*}
    s_n^2\P_\sX\big(\big[\tilde x \pm cs_n]\big]\big) &\geq 4K^2t_n(\tilde x)^2\P_\sX\big(\big[\tilde x \pm t_n(\tilde x)\big]\big)\\
    &= 4K^2\frac{\log n}n \\
    &\geq \big(1 \wedge D^{-\lceil\log_2(1/(\delta c))\rceil - 1}\big)4K^2\frac{\log n}n.
\end{align*}
Otherwise, if $c< 1,$ then, we can apply \eqref{eq.LDP} in total $k = \lceil\log_2(1/c)\rceil$ times, so that $2^kc \geq 1,$ and obtain
\begin{align*}
    s_n^2\P_\sX\big(\big[\tilde x \pm cs_n]\big]\big) &\geq 4K^2D^{-\lceil\log_2(1/c)\rceil}t_n(\tilde x)^2\P_\sX\big(\big[\tilde x \pm 2^kct_n(\tilde x)\big]\big)\\
    &= 4K^2D^{-\lceil\log_2(1/c)\rceil}\frac{\log n}n \\
    &\geq \big(1 \wedge D^{-\lceil\log_2(1/(\delta c))\rceil - 1}\big)4K^2\frac{\log n}n.
\end{align*}

We now consider the case $cs_n = 2Kct_n(x^*) + \delta c|x^* - \tilde x|/2$. Suppose without loss of generality that $x^* \leq \tilde x$. If $c \geq 2/\delta > 1,$ then we get
\begin{align*}
    \big[\tilde x \pm cs_n\big] &\supset \big[\tilde x - |x^* - \tilde x| - 2Kt_n(x^*), \tilde x + |x^* - \tilde x| + 2Kt_n(x^*)\big]\\
    &\supset \big[x^* - 2Kt_n(x^*), x^* + 2Kt_n(x^*)\big] \\
    &\supset \big[x^* \pm t_n(x^*)\big],
\end{align*}
which implies
\begin{align*}
    s_n^2\P_\sX\big(\big[\tilde x \pm cs_n\big]\big) &\geq 4K^2t_n(x^*)\P_\sX\big(\big[x^* \pm t_n(x^*)\big]\big)\\
    &= 4K^2\frac{\log n}n \\
    &\geq \big(1 \wedge D^{-\lceil\log_2(1/(\delta c))\rceil - 1}\big)4K^2\frac{\log n}n.
\end{align*}
Otherwise, if $c < 2/\delta$, we can apply \eqref{eq.LDP} in total $k = \lceil\log_2(1/(\delta c))\rceil + 1$ times, so that $2^k\delta c/2 \geq 1$ to obtain
\begin{align*}
    s_n^2\P_\sX\big(\big[\tilde x \pm cs_n]\big]\big) \geq 4K^2D^{-\lceil\log_2(1/(\delta c))\rceil - 1}t_n(x^*)^2\P_\sX\big(\big[\tilde x \pm 2^kcs_n\big]\big).
\end{align*}
Since $2^kc \geq 2/\delta$, we proceed as in the previous case and find
\begin{align*}
    s_n^2\P_\sX\big(\big[\tilde x \pm cs_n\big]\big) \geq 4K^2D^{-\lceil\log_2(1/(\delta c))\rceil - 1}\frac{\log n}n 
    \geq \big(1 \wedge D^{-\lceil\log_2(1/(\delta c))\rceil - 1}\big)4K^2\frac{\log n}n.
\end{align*}
Combining both cases, for any $K >1/2,$ any $0<\delta <1$ and any $c > 0$,
\begin{align*}
    s_n^2\P_\sX\big(\big[\tilde x \pm cs_n\big]\big) \geq \big(1 \wedge D^{-\lceil\log_2(1/(\delta c))\rceil - 1}\big)4K^2\frac{\log n}n.
\end{align*}
Finally, if $c < 4/\delta$, then $\lceil \log_2(1/(\delta c) \rceil \geq -1$ and $(ii)$ follows.
\end{proof}

\subsection*{Proof of the main theorem}

\begin{proof}[\textbf{\textit{Proof of Theorem \ref{th.main_theorem}}}] The proof follows the steps outlined in Section \ref{sec.proof_strategy}. Because of $|z| = \max(z, -z)$ and since all arguments carry over to the other case, it is enough to show that
\begin{align*}
    \sup_{\P_\sX \in \mP_n(D)} \ \sup_{f_0\in \Lip(1-\delta)} \, \P_{\!f_0} \left(\sup_{x \in [0, 1]}\frac{\wh f_n(x) - f_0(x)}{t_n(x)} \geq K\right)\to 0 \quad \text{as } n\to \infty.
\end{align*}
We will derive a contradiction by considering
\begin{align}
    K> K_*:=\frac 12 \vee 3 \big(1 \vee \log(D)\big) \vee \frac{2^{11/2}D^2}{21^3}\delta^{-\log_2(D)/2} \vee 
    (32 \chi D^5)^{3/4} D^{1+\frac 12 \log_2(1/(2\delta))}
    \label{eq.K*_def}
\end{align}
where 
\begin{align}
    \chi := 2^{17/3}\cdot 21^2 D^{\lceil \log_2(\delta^{-1})\rceil/3}\delta^{-2/3}.
    \label{eq.chi_def}
\end{align}
The condition $K > 3(1 \vee \log(D))$ ensures that if $n \geq \exp(4K^2)$, then, $n\geq \exp(4K^2\vee(36(1 \vee \log^2(D))))$. Therefore, we can apply Lemma \ref{lem.sn2PX_lb} under the simplified condition $n \geq \exp(4K^2)$.

Inequality \eqref{eq.argmin_inequality_PS} states that if $g_n \in \Lip(1)$ satisfies $(\wh f_n(X_i)-g_n(X_i))(g_n(X_i)-f_0(X_i)) \geq 0$ for all $i=1,\ldots,n$, then,
\begin{align}
    \sum_{i=1}^n(\wh f_n(X_i) - g_n(X_i))^2 \leq 2\sum_{i=1}^n \varepsilon_i(\wh f_n(X_i) - g_n(X_i)).
    \label{eq.argmin_inequality_recall}
\end{align}
Let $g_n$ be the local perturbation constructed in Lemma \ref{lem.g_function} with $\psi$ and $f$ replaced by $\wh f_n$ and $f_0,$ respectively. In particular, Lemma \ref{lem.g_function} $(ii)$ ensures that $(\wh f_n(x)-g_n(x))(g_n(x)-f_0(x)) \geq 0$ for all $x\in [0,1].$ As indicated in Section \ref{sec.proof_strategy}, we begin by lower bounding the left-hand side of inequality \eqref{eq.argmin_inequality_recall}.

\medskip

\noindent
{\textit{Lower bound for the left-hand side of \eqref{eq.argmin_inequality_recall}}:} Work on the event $\Gamma_n(D^{-4})^c$ defined in \eqref{eq.Gamman_def}. Let $x^*, \tilde x$ and $s_n$ be defined as in Lemma \ref{lem.g_function}, that is, $x^* \in \argmax_{x \in [0, 1]} t_n(x)^{-1} (\wh f_n(x) - f_0(x)),$ $\tilde x \in \argmax_{x \in [0, 1]} \wh f_n(x) - f_0(x) - \tfrac{\delta}2 |x - x^*|$ and $s_n = 2Kt_n(\tilde x) \wedge (2Kt_n(x^*) + \tfrac{\delta}2|x^* - \tilde x|).$ It is sufficient to show the result for all sufficiently large $n.$ In particular, it is enough to consider $n \geq \exp(4K^2/\delta^2),$ ensuring that
\begin{align}
    \frac{s_n}{\delta} \leq \frac{2K}{\delta} t_n(\tilde x) \leq \sqrt{\log n}\sup_{x \in [0, 1]}t_n(x).
    \label{eq.sn_ub3}
\end{align}
Lemma \ref{lem.g_function} $(iii)$ shows existence of an interval $I := [\tilde x \pm s_n/8]\cap [0, 1]$ with length $\geq  (s_n/8)\wedge 1/2$ such that for all $x \in I,\ \wh f_n(x) - g_n(x) \geq s_n/4.$ By restriction of the sum to $\{i:X_i\in I\}$ and using Lemma \ref{lem.g_function} $(iv)$, we find 
\begin{align}
\begin{split}
    \sum_{i=1}^n\big(\wh f_n(X_i) - g_n(X_i)\big)^2
    &\geq \Big(\frac{s_n}{4}\Big)^2\sum_{i=1}^n\mathds{1}(X_i \in  I).\label{***}
\end{split}    
\end{align}
We now relate the right-hand side of \eqref{***} to our discretization of $[0, 1]$ with step size $\Delta_n$ defined in \eqref{eq.Deltan_def}. By \eqref{eq.K*_def}, $K \geq 1/2$. Thus, by Lemma 27 $(ii)$ (Appendix A of the supplement \cite{supplement}), $s_n \geq t_n(\tilde x) \geq \sqrt{\log n/n}$ and 
\begin{align}
    \Delta_n \leq \frac 1{16} \sqrt{\frac{\log n}n} \leq \frac{s_n}{16}. 
    \label{eq.44537}
\end{align}
Hence, there exist two random integers $0 \leq \ell_1 < k_1 \leq N_n$ satisfying
\begin{align*}
    \Big(\tilde x - \frac {s_n}8\Big)\vee 0 \leq \ell_1\Delta_n \leq \Big(\tilde x - \frac{s_n}{16}\Big)\vee 0 < \Big(\tilde x + \frac{s_n}{16}\Big)\wedge 1 \leq k_1\Delta_n \leq \Big(\tilde x + \frac{s_n}8\Big)\wedge 1.
\end{align*}
This implies $(k_1 - \ell_1)\Delta_n \geq s_n/16> 0$ and  $[\ell_1\Delta_n, k_1\Delta_n] \subseteq I=[\tilde x \pm s_n/8]\cap [0, 1].$ Applying the lower bound $(i)$ in Lemma \ref{lem.sn2PX_lb}, we find
\begin{align*}
    \P_\sX\big([\ell_1\Delta_n, k_1\Delta_n]\big) &\geq \P_\sX\Big(\Big[\tilde x \pm \frac{s_n}{16}\Big]\Big)
    \geq \frac{\log^2 n}{D^4 n},
\end{align*}
with $\P_\sX$ the conditional distribution defined in \eqref{eq.PX_cond}. By the definition of the event $\Gamma_n(D^{-4})^c$ in \eqref{eq.Gamman_def}, we have $n^{-1}\sum_{i=1}^n \mathds{1}(X_i \in I) \geq  n^{-1}\sum_{i=1}^n \mathds{1}(X_i \in [\ell_1\Delta_n, k_1\Delta_n]) \geq \P_\sX([\ell_1\Delta_n, k_1\Delta_n])/2$. This and our choice of $\ell_1, k_1$ means that, on $\Gamma_n(D^{-4})^c,$ inequality \eqref{***} implies
\begin{align*}
\begin{split}
    \sum_{i=1}^n\left(\wh f_n(X_i) - g_n(X_i)\right)^2
    &\geq \Big(\frac{s_n}4\Big)^2\frac n{2} \P_\sX\big([\ell_1\Delta_n,k_1\Delta_n]\big)
    \geq \frac{s_n^2n}{32} \P_\sX\Big(\Big[\tilde x \pm \frac{s_n}{16}\Big]\Big).
\end{split}
\end{align*}
By \eqref{eq.sn_ub3}, $s_n \leq \sqrt{\log n}\sup_{x \in [0, 1]}t_n(x).$ This allows to apply \eqref{eq.LDP} in total five times to obtain
\begin{align}
\begin{split}
    \sum_{i=1}^n\left(\wh f_n(X_i) - g_n(X_i)\right)^2
    &\geq \frac{s_n^2n}{32D^{5}} \P_\sX\big(\big[\tilde x \pm 2s_n\big]\big).
\end{split}\label{eq.final_lb}    
\end{align}

\noindent
\textit{Upper bound for the right-hand side of \eqref{eq.argmin_inequality_recall}:} We now derive an upper bound for $2\sum_{i=1}^n \varepsilon_i\big(\wh f(X_i) - g_n(X_i)\big).$ Since $\wh f_n - g_n$ is supported on a small subset of $[0, 1]$, it is advantageous to study the sum over $X_i$ in the support. For $0 \leq a < b\leq 1,$ denote the class of $1$-Lipschitz functions supported on the interval $[a,b]$ by
\begin{align*}
    \mE_{[a, b]} := \{h \in \Lip(1): \supp(h) \subset [a, b]\}.
\end{align*}
Define $m(X)$ as the cardinality of the set $\{i \in \{1, \ldots n\}: X_i \in [a, b]\}$ and write  $Z_1, \ldots, Z_{m(X)}$ for the $m(X)$ variables $X_i$ that fall into the interval  $[a, b].$ For a function $h \in \mE_{[a, b]},$ we define the effective empirical $L_2$-norm,
\begin{align*}
    \|h\|_{m(X)} := \bigg(\frac 1{m(X)} \sum_{i =1}^{m(X)} h(Z_i)^2\bigg)^{1/2}.
\end{align*}
Here, effective refers to the fact that the semi-norm is computed based on the 'effective' sample $Z_1, \ldots, Z_{m(X)}.$ The normalization is chosen such that $n\|h\|_n^2=m(X)\|h\|_{m(X)}^2$ for all $h\in \mE_{[a, b]}.$ From now on, we follow the same steps as in Chapter 13 from \cite{wainwrighthighdimstat} with the sample size $n$ replaced by the effective sample size $m(X).$ The so-called critical inequality with $\sigma = 1$ is
\begin{align}
    \frac{\mG_n(\eta, \mE_{[a, b]})}\eta\leq \frac \eta{2},
    \label{eq.cricital_inequality}
\end{align}
with $\mG(\eta, \mE_{[a, b]})$ the Gaussian complexity of the set $\mE_{[a, b]}$, that is,
\begin{align*}
    \mG_n(\eta, \mE_{[a, b]}) := \E_\varepsilon\Bigg[\sup_{\substack{h \in \mE_{[a, b]}, \, \|h\|_{m(X)} \leq \eta}} \ \frac 1{m(X)} \bigg|\sum_{i=1}^{m(X)} \varepsilon_i h(Z_i)\bigg|  \Bigg].
\end{align*}
Recall that a function class $\mF$ is called star-shaped if $f\in \mF$ implies $\alpha f \in \mF$ for all $0\leq \alpha \leq 1.$ Observing that $\mE_{[a, b]}$ is star-shaped, one can derive the following modified version of Theorem 13.1 in \cite{wainwrighthighdimstat}.
\begin{theorem}
\label{th.adapted_wainwright}
If $\eta$ is a solution of the critical inequality \eqref{eq.cricital_inequality}, then for any $t \geq \eta$,
\begin{align*}
    \P\Big(\big\{\|\wh f_n - g_n\|_{m(X)}^2 \geq 16t\eta\big\} \cap \big\{\supp(\wh f_n - g_n) \subset [a, b]\big\} \, \Big|\,  X_1, \ldots, X_n\Big) \leq e^{ -\frac{m(X)t\eta}{2}}.
\end{align*}
\end{theorem}
The covering number of $\mF^*$ with sup-norm balls of radius $r$ is denoted by $N(r,\mF^*, \|\cdot\|_\infty).$ Along with Theorem \ref{th.adapted_wainwright} comes a modified version of Corollary 13.1 in \cite{wainwrighthighdimstat} stating a sufficient condition for $\eta$ to solve the critical inequality.
\begin{corollary}
\label{cor.adapted_wainwright}
Set $B_n(\eta, \mE_{[a, b]}) := \{h \in \mE_{[a, b]}: \|h\|_{m(X)} \leq \eta\}$. Under the conditions of Theorem \ref{th.adapted_wainwright}, any $\eta \in (0, 1]$ satisfying
\begin{align}
    \frac{16}{\sqrt{m(X)}}\int_{\frac{\eta^2}{4}}^\eta \sqrt{\log N\big(t, B_n(\eta, \mE_{[a, b]} , \|\cdot\|_\infty})\big) \, d t \leq \frac{\eta^2}{4}
    \label{eq.covering_int_condition}
\end{align}
also satisfies the critical inequality and can be used in the conclusion of Theorem \ref{th.adapted_wainwright}.
\end{corollary}

We now derive a bound for the covering number of the class $\mE_{[a, b]}$ by slightly refining the proof of classical results such as Corollary 2.7.10 in \cite{van1996weak}, see Appendix B of the supplement \cite{supplement} for a full proof.

\begin{lemma}
\label{lem.cov_num}
Given two real numbers $a<b,$ let $\mathcal{E}_{[a,b]}:= \{u\in \Lip(1): \supp(u) \subset [a,b]\}$. Then, for any positive $r$,
\begin{align*}
    \log N\big(r, \mathcal{E}_{[a,b]}, \|.\|_\infty\big) &\leq \frac{b-a}{r}
    \log 3.
\end{align*}
\end{lemma}
This allows us to upper bound the left-hand side of \eqref{eq.covering_int_condition} by $32\sqrt{m(X)^{-1}\eta (b-a) \log 3}$ and, for $n \geq 9$, by $32\sqrt{m(X)^{-1}\eta (b-a) \log(n)/2}$. Hence, if $\eta$ satisfies 
\begin{align*}
    \Big(2^{13}\frac{(b-a) \log(n)}{m(X)}\Big)^{1/3} \leq \eta,
\end{align*}
then it also satisfies \eqref{eq.covering_int_condition}. The latter inequality holds for $\eta = 21\big(\tfrac{(b - a)\log n}{m(X)}\big)^{1/3}.$ 

We further work on the classes $\mE_{[\ell\Delta_n, k\Delta_n]}$ with $0 \leq \ell < k \leq N_n$ and adapt the notation by defining $m_{k,\ell}(X)$ as the cardinality of the set $\{i \in \{1,\ldots,n\}: X_i \in [\ell\Delta_n, k\Delta_n] \}$, and 
\begin{align*}
    \eta_{n,k,\ell} := 21\Big(\tfrac{(k - \ell)\Delta_n\log n}{m_{k, \ell}(X)}\Big)^{1/3}.
\end{align*}
Set
\begin{align*}
    \mathcal{S}_{k, \ell} := \big\{(X_1,Y_1),\ldots,(X_n,Y_n): \supp(\wh f_n - g_n) \subseteq [\ell\Delta_n, k\Delta_n]\big\}.
\end{align*}
On $\mathcal{S}_{k, \ell},$ we have by Lemma \ref{lem.g_function} $(iii)$, that $[\tilde x \pm s_n/4]\subseteq [x_\ell,x_u]=\supp(\wh f_n - g_n) \subseteq [\ell\Delta_n, k\Delta_n].$ Using the first claim of Lemma \ref{lem.sn2PX_lb} and the fact that $D\geq 2$, we find that 
\begin{align}
    \P_\sX\big([\ell\Delta_n, k\Delta_n]\big)
    > \frac{\log^2 n}{D^4 n}.
    \label{eq.65554}
\end{align}
The second claim of Lemma \ref{lem.sn2PX_lb} combined with $P([\tilde x \pm a])= 1$ for $a\geq 1$, yields
\begin{align}
\begin{split}
    \P_\sX([\ell\Delta_n, k\Delta_n])(k - \ell)^2\Delta_n^2
    &\geq \P_\sX\Big(\Big[\tilde x\pm \frac{s_n}4\Big]\Big)\Big(\frac{s_n}{4} \wedge \frac 12 \Big)^2 \\
    &\geq \frac 14 \P_\sX\Big(\Big[\tilde x \pm \frac {s_n}4\Big]\Big)\Big(\frac{s_n}{4} \wedge 1 \Big)^2 \\
    &\geq \frac 14 \wedge \bigg(\frac{s_n^2}{64} \P_\sX\Big(\Big[\tilde x \pm \frac {s_n}4\Big]\Big)\bigg) \\
    &\geq \frac 14 \wedge \frac{K^2}{16}\frac{\log n}n D^{\log_2(\delta) -4}.
    \end{split}
    \label{eq.1111}
\end{align}
Define the set $\mathcal{T}=\{(k,\ell):\P_\sX([\ell\Delta_n, k\Delta_n]) > \log^2 n/(D^4 n)\}$ and note that by \eqref{eq.65554}, $(k,\ell) \notin \mathcal{T}$ implies that $\mathcal{S}_{k,\ell}$ is the empty set. Applying Corollary \ref{cor.adapted_wainwright}, we obtain
\begin{align}
    \begin{split}
    \P\Big( \Big\{\frac 1{m_{k,\ell}(X)} \sum_{i=1}^n \Big(\wh f_n(X_i) - g_n(X_i)\Big)^2 \geq 16t\eta_{n,k,\ell}\Big\} \cap \mathcal{S}_{k, \ell} \, \Big|& X_1, \dots, X_n\Big)\\ 
    &\leq e^{-\frac{m_{k,\ell}(X)t\eta_{n,k,\ell}}{2}}\mathds{1}\big((k,\ell)\in \mathcal{T}\big).
    \end{split}
    \label{eq.7888}
\end{align}
Recall that we work on the event $\Gamma(D^{-4})^c$. For any pair $(k,\ell)\in \mathcal{T},$ we have by \eqref{eq.65554}, $\tfrac n2 \P_\sX([\ell\Delta_n, k\Delta_n]) \newline \leq m_{k,\ell}(X) \leq 2n\P_\sX([\ell\Delta_n, k\Delta_n])$. Multiplying by $\eta_{n,k,\ell}^2$ and rearranging the terms yields 
\begin{align}
\begin{split}
    \frac 1{2^{1/3}} \leq \frac{m_{k,\ell}(X)\eta_{n,k,\ell}^2}{21^2P_{k, \ell}^{1/3}} \leq 2^{1/3}, \quad \text{with} \ \ P_{k, \ell} := n \P_\sX([\ell\Delta_n, k\Delta_n])(k - \ell)^2\Delta_n^2\log^2 n.
\end{split}    
    \label{eq.gamma_ineq_applied}
\end{align} 
Consider the event
\begin{align*}
    \mD_{k,\ell} &:=\underbrace{\Big\{\sum_{i=1}^n \Big( \wh f(X_i) - g(X_i)\Big)^2\geq 16\cdot 21^2(2P_{k, \ell})^{1/3} \Big\}}_{:= A} \cap \mathcal{S}_{k, \ell}\cap \Gamma_n(D^{-4})^c \\
    &\subseteq 
    \Big\{\frac 1{m_{k,\ell}(X)} \sum_{i=1}^n \Big( \wh f(X_i) - g(X_i)\Big)^2 \geq 16\eta_{n,k,\ell}^2\Big\}\cap \mathcal{S}_{k, \ell} \cap \Gamma_n(D^{-4})^c.
\end{align*}
Let $A,B$ be measurable sets and assume that $\P(A|X)\leq a(X).$ If $B$ only depends on $X,$ then, we have that $\P(A\cap B)=\E[\P(A\cap B|X)]=\E[\P(A|X)\mathds{1}(X\in B)]\leq \E[a(X)\mathds{1}(X\in B)].$ Below we apply this inequality for $X$ the sample $(X_1,Y_1),\ldots,(X_n,Y_n),$ $A$ the event defined in the previous display, $B=S_{k, \ell} \cap \Gamma_n(D^{-4})^c$ and $a(X)=\exp(-\tfrac 12 m_{k,\ell}(X)\eta_{n,k,\ell}^2).$ By \eqref{eq.K*_def}, $K > 2^{11/2}21^{-3} D^2 \delta^{-\log_2(D)/2}.$ Choosing $t = \eta_{n,k,\ell}$ and using \eqref{eq.7888} as well as \eqref{eq.1111}, we find for any $(k,\ell)\in \mathcal{T},$
\begin{align}
\begin{split}
    \P(\mD_{k,\ell})
    &\leq \E\Big[e^{-\frac{m_{k,\ell}(X)\eta_{n,k,\ell}^2}{2}}\mathds{1}\big(X\in \mathcal{S}_{k, \ell}\cap \Gamma_n(D^{-4})^c\big)\Big] \\
    &\leq \E\bigg[\exp\Big(-\frac {21^2}{2^{4/3}}\Big( n\log^2 n\P_\sX\big( [\ell\Delta_n, k\Delta_n]\big) (k - \ell)^2\Delta_n^2\Big)^{1/3}\Big) \bigg] \\
    &\leq \exp\bigg(-\frac{21^2}{2^{4/3}}\bigg[\Big(\frac{n\log^2n}4\Big)^{1/3} \wedge \Big(\frac{K^2}{16}D^{\log_2(\delta) - 4}\log^3 n\Big)^{1/3}\bigg]\bigg)\\
    &\leq  \exp\bigg(-\frac{21^2}{2^{4/3}}\Big(\frac{n\log^2n}4\Big)^{1/3}\bigg) \vee 
    n^{-2},
    \end{split}\label{eq.23455}
\end{align}
using $\exp(-a\log n)=n^{-a}$ for the last step. (Choosing the lower bound $K_*$ in \eqref{eq.K*_def} large enough, one can modify the previous argument and achieve polynomial decay in $n$ of any order.) If $(k,\ell)\notin \mathcal{T},$ then \eqref{eq.7888} implies $\P(\mD_{k,\ell})=0.$ 

Define the random variables $0 \leq \wh \ell < \wh k\leq N_n$ such that
\begin{align*}
    (\wh k,\wh \ell) \in \argmin_{(k,\ell): 0 \leq \ell < k \leq N_n}\big\{(k-\ell) \Delta_n : \supp(\wh f_n - g_n) \subset [\ell\Delta_n, k\Delta_n]\big\}.
\end{align*}
With $\mD_{k,\ell}$ as defined above, applying \eqref{eq.23455}, $N_n\leq 32\sqrt{n/\log n}$ and the union bound yields
\begin{align}
\begin{split}
    \P(\mD_{\wh k, \wh \ell}) 
    &\leq \sum_{0 \leq \ell < k \leq N_n} \P(\mD_{k,\ell}) \\ 
    &\leq
    \sum_{(k,\ell)\in \mathcal{T}} 
    \P(\mD_{k,\ell}) \\
    &\leq \frac{1024n}{\log n} \bigg(\exp\Big(-\frac{21^2}{2^{4/3}}\big(n\log^2 n/4\big)^{1/3}\Big) \vee 
    n^{-2}\bigg) \to 0 \quad \text{as }n\to \infty.
\end{split}\label{eq.Dkl_prob_bd}    
\end{align}
The convergence is uniform over $f_0 \in \Lip(1)$ and $\P_\sX\in \mP_n(D).$

In a next step of the proof, we provide a simple upper bound of the least squares distance on the set $\mD_{\wh k,\wh \ell}.$ Inequality \eqref{eq.44537} shows $\Delta_n\leq s_n/16\leq s_n/\delta$ and Lemma \ref{lem.g_function} $(iii)$ yields
\begin{align*}
    \tilde x - 2\frac{s_n}{\delta} \leq \tilde x - \frac{s_n}{\delta} - \Delta_n \leq \wh\ell\Delta_n \leq x_\ell < x_u \leq \wh k \Delta_n \leq \tilde x + \frac{s_n}{\delta} + \Delta_n
    \leq \tilde x + 2\frac{s_n}{\delta},
\end{align*}
implying $\P_\sX([\wh \ell\Delta_n, \wh k\Delta_n]) \leq \P_\sX([\tilde x \pm 2s_n/\delta])$ and $(\wh k - \wh\ell)\Delta_n \leq 4s_n/\delta$. This allows to further upper bound $P_{\wh k, \wh \ell}$ by $n\P_\sX([\tilde x \pm 2s_n/\delta])(4s_n/\delta)^2$. Using this, rearranging the rightmost inequality in \eqref{eq.gamma_ineq_applied} and applying the local doubling property \eqref{eq.LDP}  $\lceil\log_2(\delta^{-1})\rceil$ times recalling the inequality $s_n/\delta\leq \sqrt{\log n} \sup_x t_n(x)$ derived in \eqref{eq.sn_ub3}, shows that, on $\mD_{\wh k,\wh \ell}^c \cap \Gamma_n(D^{-4})^c,$
\begin{align*}
    \sum_{i=1}^n \big( \wh f_n(X_i) - g_n(X_i)\big)^2
    &\leq 16\cdot 21^2\Big(2n \P_\sX\big([\tilde x \pm 2s_n/\delta]\big)\Big(\frac{4s_n}{\delta}\Big)^2\log^2 n\Big)^{1/3}\\
    &\leq 2^{17/3}\cdot 21^2\delta^{-2/3}\Big(nD^{\lceil\log_2(\delta^{-1})\rceil}\P_{\sX}\big([\tilde x \pm 2s_n]\big)s_n^2\log^2 n\Big)^{1/3}\\
    &\leq \chi\Big(n\P_\sX\big([\tilde x \pm 2s_n]\big)s_n^2\log^2n\Big)^{1/3},
\end{align*}
with $\chi = 2^{17/3}\cdot 21^2 D^{\lceil \log_2(\delta^{-1})\rceil/3}\delta^{-2/3}$ as defined in \eqref{eq.chi_def}.

\medskip

\noindent
\textit{Combining the bounds for \eqref{eq.argmin_inequality_recall}:}  Using the lower bound \eqref{eq.final_lb} and the upper bound derived above, \eqref{eq.argmin_inequality_recall} implies that on the event $\mD_{\wh k,\wh \ell}^c \cap \Gamma_n(D^{-4})^c,$
\begin{align}
\begin{split}
    \label{eq.final_lb2}
    \frac{s_n^2n}{32D^5} \P_\sX\big(\big[\tilde x \pm 2s_n\big]\big)
    \leq \chi\Big(n \P_\sX\big([\tilde x \pm 2s_n]\big)s_n^2\log^2 n\Big)^{1/3}.
\end{split} 
\end{align}
Rearranging the terms in \eqref{eq.final_lb2}, and raising both sides to the power $3/2$ gives
\begin{align*}
    B s_n^{2}n\P_\sX\big([\tilde x\pm 2s_n]\big)\leq \log n,
\end{align*}
with $B:= (32\chi D^5)^{-3/2}$. Since $\chi$ is an absolute constant independent of $K$, $B$ is also independent of $K$. Using the second claim of Lemma \ref{lem.sn2PX_lb} and dividing by $\log n$ on both sides, we obtain on the event $\mD_{\wh k,\wh \ell}^c \cap \Gamma_n(D^{-4})^c,$
\begin{align*}
     4K^2BD^{-\lceil\log_2(1/(2\delta))\rceil -1} \leq 1.
\end{align*}
But since by \eqref{eq.K*_def}, $K>B^{-1/2} D^{1+\frac 12 \log_2(1/(2\delta))},$ we have derived a contradiction. Because $K$ was chosen to be any number larger than $K_*,$ we must have, on the event  $\mD_{\wh k,\wh \ell}^c \cap \Gamma_n(D^{-4})^c,$
\begin{align*}
    \sup_{x \in [0, 1]}\frac{\wh f_n(x) - f_0(x)}{t_n(x)}\leq K_*.
\end{align*}
The probability of the exceptional set tends to zero because by \eqref{eq.Dkl_prob_bd} and Lemma \ref{lem.prob.to.zero}, 
\begin{align*}
    \P\big((\mD_{\wh k,\wh \ell}^c \cap \Gamma_n(D^{-4})^c)^c\big)\leq \P(\mD_{\wh k,\wh \ell})+\P\big(\Gamma_n(D^{-4})\big)\to 0\text{ as }n\to \infty.
\end{align*}
The convergence can be checked to be uniform over $f_0\in \Lip(1-\delta)$ and $\P_\sX\in \mP_n(D).$ The proof is complete.
\end{proof}

\begin{acks}[Acknowledgments]
 The authors would like to thank the editor, the AE and three anonymous referees for helpful comments and suggestions.
\end{acks}

\begin{supplement}

\stitle{Supplement to "Local convergence rates of the nonparametric least squares estimator with applications to transfer learning"}

\sdescription{Additional proofs and technical lemmas.}

\end{supplement}

\begin{funding}
The research has been supported by the NWO/STAR grant 613.009.034b and the NWO Vidi grant
VI.Vidi.192.021.
\end{funding}

\bibliographystyle{imsart-number} 
\bibliography{references.bib} 

\newpage

\begin{frontmatter}

\title{Supplement to "Local convergence rates of the nonparametric least squares estimator with applications to transfer learning"}
\runtitle{Local Convergence Rates}

\begin{aug}

\author[A]{\fnms{Johannes}~\snm{Schmidt-Hieber}\ead[label=e1]{a.j.schmidt-hieber@utwente.nl}}
\author[A]{\fnms{Petr}~\snm{Zamolodtchikov}\ead[label=e2]{p.zamolodtchikov@utwente.nl}}

\address[A]{Department of Applied Mathematics,
University of Twente\printead[presep={,\ }]{e1,e2}}

\end{aug}

\end{frontmatter}

\begin{appendix}

\section{Proofs for the properties of the spread function}

\begin{proof}[\textbf{Proof of Lemma 1}]
Let $x \in [0,1]$ and define $h(t) := t^2\P_\sX([x - t, x+ t]).$ We have $h(0) = 0$ and $h(1) = 1$. Also, since $\P_\sX$ admits a Lebesgue density, $h$ is continuous. As a consequence of the mean value theorem, there is at least one $t \in [0, 1]$ such that $h(t)=\log n/n.$ For $n>1,$ $\log n/n>0$ and hence $t>0$ as well as $\P_\sX([x - t, x+ t])>0.$ Thus, if $t_0$ denotes the smallest solution, then it follows that $h$ is strictly increasing on $[t_0,\infty).$ Thus, there cannot be a second solution for $h(t) = \log n/n,$ proving uniqueness.
\end{proof}

\begin{lemma}
\label{lem.tn_diff}
If $\P_\sX \in \mM$ then,
\begin{enumerate}
\item[$(i)$] $x \mapsto t_n(x)$ is $1$-Lipschitz,
\item[$(ii)$] for all $n \geq 3, \ \inf_{x \in [0, 1]}t_n(x) \geq \sqrt{\log n/n},$
\item[$(iii)$] if they exist, the solutions $x_1, x_2$ to, respectively, $t_n(x) = x$ and $t_n(x) = 1-x$ are unique; moreover, for $n \geq 9,$ both solutions exist and verify $0 < x_1 < 1/2 < x_2 < 1,$ and
\item[$(iv)$] if, additionally, $p$ is continuous on $[0, 1],$ then $t_n$ is differentiable on $(0,1)\setminus\{x_1,x_2\}$ and its derivative is given by 
\begin{align*}
    t_n^\prime(x) = \frac{p(x-t_n(x))\mathds{1}(t_n(x) \leq x) - p(x+t_n(x))\mathds{1}(t_n(x) \leq 1-x)}{2\tfrac{\log n}{nt_n(x)^3} + p(x+t_n(x))\mathds{1}(t_n(x) \leq 1-x) + p(x - t_n(x))\mathds{1}(t_n(x) \leq x)}.
\end{align*}
\end{enumerate}
\end{lemma}

\begin{remark}
    \label{rem.monotonicity_tn_borders}
    In fact, we always have $t_n(0) > 0,$ and $t_n(1)>0.$ This means that one can partition $[0, 1]$ in three intervals $I_1 := [0, x_1),\ I_2 := [x_1, x_2]$ and $I_3 := (x_2, 1],$ such that on $I_1,\ t_n(x) > x,$ on $I_2,\ t_n(x) \leq x\wedge(1-x),$ and on $I_3,\ t_n(x) \geq 1-x.$ From the expression of $t_n^\prime,$ it follows that $t_n$ is strictly decreasing on $I_1$ and strictly increasing on $I_3$.
\end{remark}

\begin{proof}

{\it Proof of $(i)$:} Assume there are points $0\leq x,  y\leq 1$ such that $t_n(x)> t_n(y)+|x-y|.$ Then, $[y\pm t_n(y)]\subseteq [x\pm t_n(x)]$ and, therefore,
\begin{align*}
    \frac{\log n}{n}=t_n(y)^2\P_{\sX}\big([y\pm t_n(y)]\big)
    < t_n(x)^2\P_{\sX}\big([x\pm t_n(x)]\big)
    = \frac{\log n}{n}.
\end{align*}
Since this is a contradiction, we have $t_n(x)\leq  t_n(y)+|x-y|$ for all $0\leq x,y\leq 1.$ Hence $t_n$ is $1$-Lipschitz. 

{\it Proof of $(ii)$:} The inequality is an immediate consequence of the definition of the spread function. For any $x \in [0, 1], \ \P_\sX([x \pm t_n(x)]) \leq 1,$ implying
\begin{align*}
    t_n(x) = \sqrt{\frac{\log n}{n\P_\sX([x \pm t_n(x)])}} \geq \sqrt{\frac{\log n}n}. 
\end{align*}

{\it Proof of $(iii)$:} Assuming one solution exists, we first show its uniqueness. Then we prove that for all $n > 9,$ one solution exists. Suppose that the equation $t_n(x) = x$ admits at least one solution $x_1$ and suppose that there is another solution $y_1$ such that $0<x_1<y_1<1.$ Then $x_1^2P_X([0, 2x_1]) = y_1^2P([0, 2y_1]),$ which means that
\begin{align*}
    0 > (x_1^2 - y_1^2)\P_\sX([0, 2x_1]) = y_1^2P_X((2x_1, 2y_1]) \geq 0.
\end{align*}
This is a contradiction. Therefore, $x_1 = y_1,$ and the solution must be unique.

Next, we assume $n \geq 9,$ and we prove the existence of a solution $x_1\in (0,1/2)$ of the equation $t_n(x)=x$. The function $h\colon x \mapsto x^2\P([0, 2x])$ is continuous and increasing with $h(0) = 0$ and $h(1/2) \geq P([0,1])/4\geq 1/4$. Since $n\geq9,$ we have $0<\log n/n<1/4$ and the mean value theorem guarantees the existence of $x_1\in (0,1/2)$ such that $h(x_1) = \log n/n,$ that is, $x_1 = t_n(x_1).$

For the solution $x_2,$ one can apply the same reasoning to the distribution with density function $\tilde p\colon x \mapsto p(1-x).$

{\it Proof of $(iv)$:} Consider the function $f\colon (x, t) \mapsto t^2\P_{\sX}([x \pm t]) - \frac{\log n}n$. Denote by $D_1f(x, y)$ and $D_2f(x, y)$ the partial derivatives of $f$ with respect to its first and second variable evaluated at $(x, y)$. Since $\P_\sX$ admits a continuous Lebesgue density $p$ on $(0,1)$ and $t \mapsto t^2$ is smooth, $f$ is continuously differentiable on $S := (0, 1)\times (0,1)\setminus (\{(x,t): x=t\}\cup \{(x,t): x+t=1\}),$ and
\begin{align*}
    D_1f(x, t_n(x)) &= t_n(x)^2\left[p(x + t_n(x))\mathds{1}(t_n(x) \leq 1-x) - p(x - t_n(x))\mathds{1}(t_n(x) \leq x)\right].\\
    D_2f(x, t_n(x)) &=
    2t_n(x)\P_{\sX}([x\pm t_n(x)]) \\
    &\quad\! + t_n(x)^2\!\left[\!\Big(p(x+t_n(x))\mathds{1}(t_n(x) \leq 1-x) + p(x-t_n(x))\mathds{1}(t_n(x) \leq x)\Big)\!\right].
\end{align*}
For any $x \in [0, 1], \, f(x, t_n(x)) = 0$. Therefore, if $D_2f(x, t_n(x)) \neq 0,$ then one can apply the implicit function theorem, which states the existence of an open neighbourhood $U$ of $x$ and a continuously differentiable function $g\colon U \mapsto \RR$ satisfying $f(x, g(x)) = 0$ for all $x \in U$. Moreover, for all $x \in U,$ the derivative of $g$ is given by
\begin{align*}
    g^\prime(x) = -\left[D_2(x, g(x))\right]^{-1}\left[D_1(x, g(x))\right].
\end{align*}
Since $t_n(x)$ satisfies $f(x, t_n(x)) = 0$ for all $x \in U$ and $g$ is unique, we must have $t_n(x) = g(x)$. Moreover, since $t_n(x) > 0$ for all $x \in S,$ we have $D_2f(x, t_n(x)) > 0$ for all $x \in S$. Using $\P_\sX([x \pm t_n(x)]) = \log n/(nt_n(x)^2),$ we have, for all $x \in S,$ 
\begin{align*}
    t_n^\prime(x) &=
    \frac{t_n(x)\left[p(x + t_n(x))\mathds{1}(t_n(x) \leq 1-x) - p(x - t_n(x))\mathds{1}(t_n(x) \leq x)\right]}{2\P_{\sX}([x\pm t_n(x)]) + t_n(x)\!\left[\!\Big(p(x+t_n(x))\mathds{1}(t_n(x) \leq 1-x) + p(x-t_n(x))\mathds{1}(t_n(x) \leq x)\Big)\!\right]}\\
    &= \frac{p(x-t_n(x))\mathds{1}(t_n(x) \leq x) - p(x+t_n(x))\mathds{1}(t_n(x) \leq 1-x)}{2\tfrac{\log n}{nt_n(x)^3} + p(x+t_n(x))\mathds{1}(t_n(x) \leq 1-x) + p(x - t_n(x))\mathds{1}(t_n(x) \leq x)}.
\end{align*}
\end{proof}

\begin{proof}[\textbf{Proof of Lemma 3}]
Using the definition of the spread function, the inequality 
\begin{align*}
    \sqrt{\log n}\sup_{x\in [0,1]} t_n(x)<\varepsilon
\end{align*}
is equivalent to 
\begin{align}
    \inf_{x\in [0,1]} \, \P_\sX\big([x\pm t_n(x)]\big)> \frac{\log^2 n}{n\varepsilon^2}.
    \label{eq.tnPX_equiv}
\end{align}
To prove the lemma, we now argue by contradiction. Assume existence of an $x^* \in [0, 1],$ such that $\P_\sX([x^* \pm t_n(x^*)]) \leq \log^2 n/(n\varepsilon^2)$. We show that for all sufficiently large $n,$ this implies that $\P_\sX([0, 1]) < 1$ and thus $\P_\sX$ is not a probability measure. 

In a first step, we prove by induction that for $k\geq 1,$
\begin{align}
    \P_\sX\big(\big[x - (k + 1) t_n(x^*), x - (k-1)t_n(x^*)\big]\big) \leq D^k \P_\sX\big(\big[x \pm t_n(x^*)\big]\big). 
    \label{eq.induction_hypothesis_design}
\end{align}
For $k=1,$ using the local doubling condition, we have for any $x \in [0, 1],$
\begin{align*}
    \P_\sX\big(\big[x - 2t_n(x^*), x\big]\big) \leq \P_\sX\big(\big[x\pm 2t_n(x^*)\big]\big) \leq DP_X\big(\big[x\pm t_n(x^*)\big]\big).
\end{align*}
If \eqref{eq.induction_hypothesis_design} holds for some positive integer $k,$ then we have
\begin{align*}
    \P_\sX\big(\big[x - (k+2) t_n(x^*), x - kt_n(x^*)\big]\big) &\leq \P_\sX\big(\big[x - kt_n(x^*) \pm 2t_n(x^*)\big]\big)\\
    &\leq \P_\sX\big(\big[x - (k+1)t_n(x^*), x - (k-1) t_n(x^*)\big]\big)\\
    &\leq D^{k+1}\P_\sX\big(\big[x\pm t_n(x^*)\big]\big),
\end{align*}
proving the induction step. Thus \eqref{eq.induction_hypothesis_design} holds for all positive integers $k.$ By symmetry, one can prove that for all positive integers $k, \ \P_\sX([x + (k-1)t_n(x^*), x + (k+1)t_n(x^*)]) \leq D^k\P_\sX([x\pm t_n(x^*)]).$ Since $t_n(x^*) > \varepsilon/\sqrt{\log n},$ we can cover each of the intervals $[0,x^*]$ and $[x^*,1]$ by at most $\lceil \sqrt{\log n}/\varepsilon \rceil$ intervals of length $2t_n(x^*)$. This implies 
\begin{align*}
    \P_\sX([0, 1]) &\leq \sum_{k = 1}^{\lceil \sqrt{\log n}/\varepsilon \rceil} \P_\sX\big(\big[x^* - (k+1)t_n(x^*), x^* - (k-1)t_n(x^*)\big]\big) \\
    &\qquad\qquad + \P_\sX\big(\big[x^* + (k-1) t_n(x^*), x^* + (k+1)t_n(x^*)\big]\big)\\
    &\leq 2D^{\lceil \sqrt{\log n}/\varepsilon \rceil}\P_\sX\big(\big[x^*\pm t_n(x^*)\big]\big)\\ 
    &\leq 2\frac{D^{\lceil \sqrt{\log n}/\varepsilon \rceil} \log^2n}{n} \to 0 \quad \text{as $n\to \infty$}.
\end{align*}
Hence, there exists an $N = N(\varepsilon, D)$ such that for all $n \geq N, \P_{\sX}([0, 1]) < 1$. This is a contradiction. Hence, the claim holds.
\end{proof}

\begin{remark}
\label{rem.tnPX}
For $\varepsilon = 1,$ one can take $N(1, D) = \exp(36(1 \vee \log^2 D))$ in the previous result. To see this, consider for $\varepsilon=1$ the last upper bound in the proof of \eqref{eq.tnPX_equiv},
\begin{align*}
    \P_X([0, 1]) \leq 2\frac{D^{\lceil \sqrt{\log n}\rceil}\log^2n}{n}.
\end{align*}
Substituting $n = \exp(t^2),$ using $t^4 \leq \exp(4t)$ and $\lceil t \rceil \leq t + 1$ yields,
\begin{align*}
    2\frac{D^{\lceil t\rceil }t^4}{\exp(t^2)} \leq  2\frac{D^{t + 1} t^4}{\exp(t^2)} \leq \exp\big(-t^2 + t(\log(D) + 4) + \log(2) + \log(D)\big) = \exp(Q(t)).
\end{align*}
The largest root of the polynomial $Q(t) := -t^2 + t(\log(D) + 4) + \log(2) + \log(D)$ is 
\begin{align*}
    t_+ := \frac{\log(D) + 4 + \sqrt{(\log(D) + 4)^2 + 4(\log(2) + \log(D))}}{2}.
\end{align*}
Using $\log(2) \leq 7/4,$ we obtain that for $D \leq e, \ \log(D)\leq 1$ and $t_+ < (5 + \sqrt{25 + 11})/2 < 6.$ If $D > e,$ then $a \leq a\log(D)$ for all $a>0,$ and $t_+ \leq 5\log(D) + \sqrt{25\log(D) + 11\log(D)} \leq 6\log(D)$. Since $n = \exp(t^2),$ we find for any $n \geq N(1, D)=\exp(36(1 \vee \log(D)^2)),$ that $\P_X([0, 1])<1,$ completing the argument.
\end{remark}

\begin{proof}[\textbf{Proof of Lemma 6}]
Since $p(x_0) = 0,$ for all $x \in U,$ we have $A^{-1}|x - x_0|^\alpha \leq p(x) \leq A|x - x_0|^\alpha$. Consequently,
\begin{align}\label{eq.zero_behaviour_ineq}
    \frac{t^{\alpha+3}}{(\alpha + 1)A} \leq t^2\P_\sX([x_0 \pm t]) \leq \frac A{\alpha + 1}t^{\alpha + 3}.
\end{align}
By Lemma 3 and Remark \ref{rem.tnPX}, $\sup_{x \in [0, 1]}t_n(x) < \log^{-1/2}n$ for all $n \geq N(1, D)$. Hence, there exists a positive integer $N,$ only depending on $D$ and $U,$ such that for all $n \geq N,\ t_n(x_0) \in U$. Combining inequality \eqref{eq.zero_behaviour_ineq} with $t = t_n(x_0),$ using the definitions of $t_n$ and some calculus yields
\begin{align*}
    \Big(\frac{(\alpha + 1)\log n}{An}\Big)^{1/(\alpha + 3)} \leq t_n(x_0) \leq \Big(\frac{(\alpha + 1)A\log n}{n}\Big)^{1/(\alpha + 3)}.
\end{align*}
\end{proof}

\section{Proofs for upper bounds in Section 2}
\label{app:main_theorem}

\subsection*{Covering bound}
\label{app.cov_num}

\noindent\textbf{Proof of Lemma 26.}
The proof follows standard arguments. For the sake of completeness, all details are provided. As $a, b$ are fixed, we simply write $\mE$ instead of $\mE_{[a, b]}$. Let $g \in \mE$. We construct $3^{\lfloor (b-a)/r\rfloor}$ functions serving as the centers for the $\|\cdot\|_\infty$-norm covering balls.

The idea of the proof is to construct a finite subset of $\mE$ containing an element $h$ that lies less than $\|\cdot\|_\infty$-distance $r$ away from $g.$ The upper bound on the metric entropy is then given by the cardinality of this finite set. 
For $k = \lfloor \frac{b-a}{r} \rfloor$ and $x_i=ik,$ $i=0,\ldots,k,$ define $h$ by induction as follows. Set $h(x) = 0 \text{ if } x\in [0, r],$ and for any $1 \leq i \leq k$ and any $x \in (x_i, x_{i+1}],$ set
\begin{align*}
    h(x) &= 
    \begin{cases}
    h(x_i) + |x - x_i|, \text{ if for some }y \in [x_i, x_{i+1}], g(y) - h(x_i) > r,\\
    h(x_i) - |x - x_i|, \text{ if for some }y \in [x_i, x_{i+1}], g(y) - h(x_i) < -r,\\
    h(x_i), \text{ otherwise}.
    \end{cases}
\end{align*}
\begin{wrapfigure}{r}{.45\textwidth}
    \centering
    \includegraphics[width=.4\textwidth]{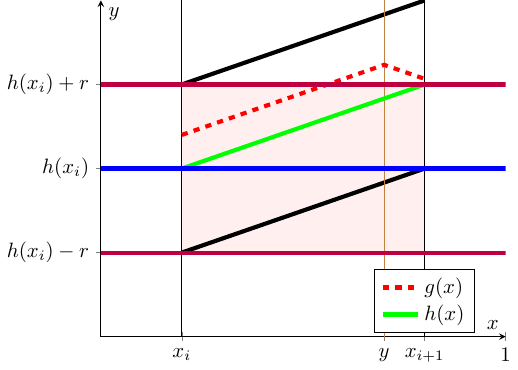}
    \caption{If one point of $g$ is more than $r$ away from $h(x_i)$ on the interval $[x_i, x_{i+1}),$ then, by construction, $h$ moves in this direction on the interval $[x_{i+1},x_{i+2})$.}
    \label{fig.covering_lip}
\end{wrapfigure}
Denote by $H$ the set of all possible functions $h$ constructed as above.

Intuitively, if $|g(x_i) - h(x_i)| \leq r$ and there is some $y \in [x_i, x_{i+1})$ such that $g(y)$ is more than $r$ away from $h(x_i),$ say $g(y) \geq h(x_i) + r,$ then, by the Lipschitz property, $g([x_i, x_{i+1}))$ will be a subset of $[h(x_i), h(x_i) + 2r]$ and the function $x \mapsto h(x_i) + |x - x_i|$ is not more than $r$ away from $g$ on the interval $[x_i, x_{i+1}),$ see Figure \ref{fig.covering_lip} for more details.

We prove by induction that for any $i \in \{0, \dots, k\}$ and for any $x \in [x_i, x_{i+1}], \ |g(x) - h(x)|\leq r$.

First, since $g(0) = 0$ and $g \in \Lip(1),$ it follows from the definition of $h$ that $\sup_{x \in [0, r]}|g(x) - h(x)| \leq r.$ This proves the property for $i = 0$. Next, assume that for some $1 \leq i \leq k-1,$ we have $\sup_{x \in [x_{i-1}, x_i]}|g(x) - h(x)| \leq r.$ We prove that $|g(x) - h(x)|\leq r$ for all $x \in [x_{i}, x_{i+1}].$ 

We distinguish two cases. First, if there exists a $y \in [x_i, x_{i+1}]$ such that $|g(y) - h(x_i)| > r,$ then, by symmetry, it suffices to prove the induction step assuming $g(y) - h(x_i) > r$. By construction of $h$, we then have that for any $x \in [x_i, x_{i+1}], \ h(x) = h(x_i) + |x - x_i|$. Since $g \in \Lip(1)$ we have, for any $x \in [x_i, x_{i+1}],$
\begin{gather*}
    -r \leq g(y) - |x - y| - h(x_i) - |x - x_i| \\
    \leq g(x) - h(x)\\
    \leq g(x_i) + |x - x_i| - h(x_i) - |x - x_i| \leq r,
\end{gather*}
where the last inequality applies the induction assumption. This proves the first case of the induction step. Next, assume that for any $x \in [x_i, x_{i+1}],\ g(x) \in [h(x_i) - r, h(x_i) + r],$ then for any $x \in [x_i, x_{i+1}],\ h(x) = h(x_i)$ and $|g(x) - h(x)| \leq r,$ proving the induction step for the second case. We conclude that for any $g \in \mathcal{E},$ there exists a function $h \in H$ such that $\|g - h\|_{\infty} \leq r$. The set $H$ has cardinal $\leq 3^k$. Hence, 
\begin{align*}
    N(r, \mathcal{E}, \|.\|_\infty) \leq 3^{\lfloor \frac{b - a}{r}\rfloor}.
\end{align*}
\qed

\subsection*{Proof of Corollary 5}

In a first step, we prove the claimed inequalities for all $n \geq \exp(4\kappa) \vee 9.$ Let $\underline p_n := \inf_{x \in [0, 1]} p_n(x)$. We can see that
\begin{align}
    \underline p_n\geq (4\kappa)^{\frac 3{3+\beta}}\Big(\frac{\log n}{n}\Big)^{\frac{\beta}{3+\beta}}.
    \label{eq.pn_lb_large_n}
\end{align}
Because $n \geq 9,$ by virtue of Lemma \ref{lem.tn_diff}, $t_n(x) < 1/2$ and therefore $[x \pm t_n(x)]$ contains an interval of length $t_n(x)$ for all $x \in [0, 1]$. By definition of the spread function $t_n,$ we have
\begin{align*}
    \frac{\log n}{n} = t_n(x)^2 \P_\sX^n\big([x\pm t_n(x)]\big)\geq t_n(x)^3 \underline p_n.
\end{align*}
Since $\underline p_n\geq n^{-1/(3+\beta)} \log n,$ it follows that
\begin{align}
    t_n(x) \leq \Big(\frac{\log n}{ \underline p_n n}\Big)^{1/3}\leq \Big(\frac{\underline p_n}{4\kappa}\Big)^{1/\beta}.
    \label{eq.1357}
\end{align}

Since $p_n$ can be extended to a function $q_n:\mathbb{R}\to \mathbb{R}$ with $|q_n|_\beta \leq \kappa,$ Theorem 4 and (2.1) in \cite{MR3714756} show that $|p_n^\prime(x)|^\beta\leq \kappa |p_n(x)|^{\beta-1},$ which is the same as 
\begin{align}
    |p_n^\prime(x)|\leq \kappa^{1/\beta} p_n(x)^{(\beta-1)/\beta}.
    \label{eq.7531}
\end{align}

We now prove 
\begin{align}
    \frac {t_n(x)^3p_n(x)}2\leq \frac{\log n}n \leq 3t_n(x)^3p_n(x).
    \label{eq.2468}
\end{align}
To do so, we distinguish two cases, either $t_n(x) \geq x \wedge (1 - x)$ or $t_n(x) \leq x \wedge (1 - x).$ Moreover, we treat the cases $\beta \leq 1$ and $1<\beta \leq 2,$ separately. 

From Lemma \ref{lem.tn_diff} and Remark \ref{rem.monotonicity_tn_borders}, we know that one can partition $[0, 1]$ in three intervals $I_1 := [0, a[,\ I_2 := [a, b], \ I_3 := ]b, 1],$ such that 
\begin{align*}
    \begin{cases}
    t_n(x) > x, &\text{ for all } x \in I_1,\\
    t_n(x) \leq x \wedge (1 - x), &\text{ for all } x \in I_2, \text{ and,}\\
    t_n(x) > 1-x, &\text{ for all } x \in I_3.
    \end{cases}
\end{align*}
We distinguish two cases, either $x \in I_2$ or $x \in I_1 \cup I_3$.

\noindent{\underline{$x \in I_2$}:} For $x \in I_2,$ observe that
\begin{align}
    \label{eq.case1_corollary}
    \P_\sX^n([x \pm t_n(x)]) = 2t_n(x)p_n(x) + \int_{x - t_n(x)}^{x + t_n(x)}(p_n(u) - p_n(x)) \, d u.
\end{align}

{\it \underline{$0 < \beta \leq 1$}:} In this case, $|p_n(u)-p_n(x)|\leq \kappa |u-x|^\beta.$ Plugging this in \eqref{eq.case1_corollary}, using $t_n(x)^2 \P_\sX^n([x \pm t_n(x)])=\log n/n$ and \eqref{eq.1357} leads to \eqref{eq.2468}.

{\it \underline{$1 < \beta \leq 2$}:} A first order Taylor expansion yields, for all $u \in [0, 1], \ p_n(u)=p_n(x)+p_n^\prime(\xi)(u-x)$ for a suitable $\xi$ between $x$ and $u.$ Since $\int_{x-t_n(x)}^{x+t_n(x)} p_n^\prime(x)(u-x) \, du=0,$ we have $\P_\sX^n([x \pm t_n(x)]) = 2t_n(x)p_n(x) + \int_{x - t_n(x)}^{x + t_n(x)}(p_n^\prime(\xi) - p_n^\prime(x))(u-x) \, \, d u.$ Using the definition of the H\"older semi-norm, we find $|\int_{x - t_n(x)}^{x + t_n(x)}(p_n^\prime(\xi) - p_n^\prime(x))(u-x) \, \, d u|\leq 2\kappa t_n(x)^{1+\beta}$ and also obtain \eqref{eq.2468}.

\noindent{\underline{$x \in I_1 \cup I_3$}:}
By symmetry, it is sufficient to prove the inequality for $x \in I_1$ (we can apply the same to the density $x \mapsto p(1 - x)$ and obtain the results on $I_2$). For $x \in I_1,$ we have 
\begin{align}
\label{eq.case2_corollary}
    \P_\sX^n([x \pm t_n(x)]) = (x + t_n(x))p(x) + \int_0^{x + t_n(x)}(p_n(u) - p_n(x))\, du.
\end{align} 

{\it \underline{$0 < \beta \leq 1$}:} In this case, $|p_n(u) - p_n(x)| \leq \kappa |u - x|^\beta$. Plugging this in \eqref{eq.case2_corollary}, using $t_n(x)^2 \P_\sX^n([x \pm t_n(x)])=\log n/n$ and \eqref{eq.1357} leads to \eqref{eq.2468}.

{\it \underline{$1 < \beta \leq 2$}:} A first order Taylor expansion yields, for all $u \in [0, 1], \ p_n(u)=p_n(x)+p_n^\prime(\xi)(u-x)$ for a suitable $\xi$ between $x$ and $u.$ Since $\int_{0}^{x+t_n(x)} p_n^\prime(x)(u-x) \, du=p^\prime(x)(t_n(x)^2 - x^2)/2,$ we have 
\begin{align*}
    \P_\sX^n([x \pm t_n(x)]) = (x + t_n(x))p_n(x) &+ \int_{0}^{x + t_n(x)}(p_n^\prime(\xi) - p_n^\prime(x))(u-x) \, d u \\
    &+ \frac{p^\prime(x)}2(t_n(x)^2 - x^2).
\end{align*} 
Therefore,
\begin{align*}
    \left|\P_\sX^n([x \pm t_n(x)]) - (x + t_n(x))p_n(x)\right| \leq (x + t_n(x))\big(\underbrace{|\kappa t_n(x)^{\beta}|}_{:= A}  + \underbrace{\left|p^\prime(x)(t_n(x) - x)/2\right|}_{:= B}\big).
\end{align*}
We now bound $A$ and $B$. Because of \eqref{eq.1357}, $A \leq \ul p_n/4 \leq p_n(x)/4$. Now, using \eqref{eq.1357} and \eqref{eq.7531}, we have
\begin{align*}
    B \leq \frac 12 t_n(x)p^\prime(x) \leq \frac 12\kappa^{1/\beta}p_n(x)^{(\beta - 1)/\beta}\Big(\frac{\ul p_n}{4\kappa}\Big)^{1/\beta}\leq \frac{1}{4}p_n(x).
\end{align*}
Consequently,
\begin{align*}
    \frac{t_n(x) p_n(x)}2 \leq \frac{x + t_n(x)}2 p_n(x) \leq \P_\sX^n([x \pm t_n(x)]) \leq \frac{3(x + t_n(x))}2 p_n(x) \leq 3t_n(x) p_n(x).
\end{align*}
Multiplying both sides by $t_n(x)^2$ and using the definition of $t_n$ yields \eqref{eq.2468}. Ultimately, (5) is a direct implication of \eqref{eq.2468} so that for all $n \geq 9,$ all $0 < \beta \leq 2,$ and all $x \in [0, 1]$,
\begin{align}
    \label{eq.3579}
    \Big(\frac{\log n}{3np_n(x)}\Big)^{1/3} \leq t_n(x) \leq \Big(\frac{2\log n}{np_n(x)}\Big)^{1/3}.
\end{align}

In a next step, we prove that $\P_\sX^n \in \mP_n(2 + 3^\beta \kappa + 3\kappa^{1/\beta}).$ For fixed $x\in [0,1]$ and fixed $\eta \in [0,\sqrt{\log n} \sup_{x \in [0, 1]}t_n(x)],$ let $\rho := \min_{u \in [x \pm \eta]}p_n(u).$

If $\beta \leq 1,$ then we have that for all $u\in [x - 2\eta, x + 2\eta]$ that $p_n(u)\leq \rho+\kappa 2^\beta \eta^\beta.$ If $1<\beta \leq 2,$ then the assumption $|p_n|_\beta \leq \kappa$ ensures that for all $u,v\in [0,1], \ |p_n^\prime(u)-p_n^\prime(v)|\leq \kappa |u-v|^{\beta-1}.$ Using \eqref{eq.7531}, we find that by first order Taylor expansion for any $v\in [x\pm \eta],$ there exists $\xi(u)$ between $u$ and $v$ such that
\begin{align*}
    \P_\sX^n\big([x\pm 2\eta]\!\!\setminus \!\![x\pm \eta]\big)
    &\leq \int_{[x\pm 2\eta]\setminus [x\pm \eta]} \!\!p_n(v)\!+\!\big(p_n^\prime(\xi(u))-p_n^\prime(v)\big)(u-v)+p_n^\prime(v) (u-v) \, du \\
    &\leq 2\eta p_n(v)+\kappa (3\eta)^\beta 2\eta+\kappa^{1/\beta} p_n(v)^{(\beta-1)/\beta} 6\eta^2.
\end{align*}
Since $v$ was arbitrary in $[x\pm \eta],$ we can replace $p_n(v)$ by $\rho.$

Thus for all $0<\beta \leq 2,$ decomposing $[x - 2\eta, x + 2\eta]=[x - \eta, x + \eta] \cup ([x - 2\eta, x + 2\eta]\setminus [x - \eta, x + \eta])$ and using that $\rho\geq \underline p_n,$ we find
\begin{align*}
    \frac{\P_\sX^n\big([x - 2\eta, x + 2\eta]\big)}{\P_\sX^n\big([x - \eta, x + \eta]\big)} &\leq 1 + \frac{2\eta(\rho +
    \kappa (3\eta)^\beta+3\kappa^{1/\beta}\rho^{(\beta-1)/\beta}\eta)}{2\eta\rho}
    \leq 2+ \kappa 3^\beta \frac{\eta^\beta}{\underline p_n}+ 3\kappa^{\frac{1}{\beta}} \frac{\eta}{\underline p_n^{1/\beta}}.
\end{align*}
Using \eqref{eq.3579}, $\beta\leq 2$ and 
$\min_{x \in [0, 1]} p_n(x) \geq n^{-\beta/(3+\beta)} \log n,$ we find 
\begin{align}
    \sqrt{\log(n)}\, t_n(x)\leq \log^{1/\beta}(n)\Big(\frac{2\log n}{n p_n(x)}\Big)^{1/3}\leq \underline p_n^{1/\beta}.
\end{align}
Combined, the two previous displays finally yield for $0\leq \eta \leq \sqrt{\log(n)}\sup_x t_n(x),$
\begin{align*}
    \frac{\P_\sX^n\big([x - 2\eta, x + 2\eta]\big)}{\P_\sX^n\big([x - \eta, x + \eta]\big)}
    \leq 2+ \kappa 2^{\beta/3}3^\beta + 2^{1/3}3\kappa^{1/\beta},
\end{align*}
proving $\P_\sX^n \in \mP_n(2 + 2^{\beta/3}3^\beta \kappa + 2^{1/3}3\kappa^{1/\beta}).$

To prove the local convergence rate, we now apply Theorem 4. Equation (5) yields $t_n(x)\leq (\log n/(np_n(x)))^{1/3}$ for all $n \geq \exp(4\kappa)\vee 9.$ Hence,
\begin{align*}
    \sup_{x \in [0, 1]}\frac{|\wh f_n(x) - f_0(x)|}{t_n(x)} \geq \Big(\frac n{\log n}\Big)^{1/3}\!\!\!\!\! \sup_{x \in [0, 1]} p(x)^{1/3} \big|\wh f_n(x) - f_0(x)\big|,
\end{align*}
and
\begin{align*}
    \P\bigg(\sup_{x \in [0, 1]} \Big( p(x)^{1/3}|\wh f_n(x) - f_0(x)|\Big) \geq K&\Big(\frac{\log n}n\Big)^{1/3}\bigg)
    \leq \P\bigg(\sup_{x \in [0, 1]}\Big(\frac{|\wh f_n(x) - f_0(x)|}{t_n(x)}\Big) \geq K \bigg). 
\end{align*}
Since $\P_\sX^n \in \mP_n(2 + 2^{\beta/3}3^\beta \kappa + 2^{1/3}3\kappa^{1/\beta}),$ we can apply Theorem 4 and the claim follows. \qed

\subsection*{Heuristic argument for (6)}

\label{app.heuristic}
To see (6), consider the kernel smoothing estimator $\wh f_{nh}(x)=(nhp_n(x))^{-1}\sum_{i=1}^n Y_iK(\tfrac{X_i-x}h)$ with positive bandwidth $h$ and a kernel function $K$ supported on $[-1,1].$ This is a simplification of the Nadaraya-Watson estimator that replaces the density $p_n$ by a kernel density estimator. Observe that using $Y_i=f_0(X_i)+\eps_i,$

\begin{gather*}
    \big|\wh f_{nh}(x) -f_0(x)\big|
    \leq
    \underbrace{\Big|\frac{1}{nhp_n(x)} \sum_{i=1}^n \eps_i K\Big(\frac{X_i-x}{h}\Big)\Big|}_{\text{stochastic error}}
    +\underbrace{\Big|\frac{1}{nhp_n(x)} \sum_{i=1}^n f_0(X_i) K\Big(\frac{X_i-x}{h}\Big)-f_0(x)\Big|}_{\text{deterministic error}}.
\end{gather*}
In the stochastic error term, the sum is over $O(nhp_n(x))$ many variables since $K$ has support on $[-1,1].$ By the central limit theorem, this means that this sum is of the order $O(\sqrt{nhp_n(x)}).$ To obtain a uniform statement in $x,$ yields an additional $\sqrt{\log n}$-factor and together with the normalization $1/(nhp_n(x)),$ the stochastic error is of the order $\sqrt{\log n/(nhp_n(x))}.$

Thus, to heuristically verify (6), it remains to bound the deterministic error term. This term is close to its expectation 
\begin{align*}
    &\Big|\frac{1}{hp_n(x)} \int_0^1 f_0(u)K\Big(\frac{u-x}{h}\Big) p_n(u) \, du -f_0(x)\Big| \\
    &= \Big|\frac 1{p_n(x)} \int \big(f_0(x+vh)p_n(x+vh)-f_0(x)p_n(x)\big) K(v) \, dv \Big|,
\end{align*}
where we used substitution $v=(u-x)/h$ and $\int K(v) \, dv=1.$ If $p_n$ is sufficiently smooth and $K$ has enough vanishing moments, then the Lipschitz property of $f_0$ shows that this term can be bounded by $h,$ completing the argument for (6).

\subsection*{Estimation of spread functions}
The symmetric difference of two sets $C,D$ is $C\triangle D:=(C\setminus D)\cup (D\setminus C).$

\begin{theorem}[Theorem 4.1 in \cite{MR890285}] \label{thm.4.1inMR890285}
For independent $d$-dimensional random vectors $X_1,\ldots,X_n \sim \P$ and a class of sets $\mathcal{C}$ with finite Vapnik-Chervonenkis (VC) dimension, let $\P_n$ be the empirical distribution of $X_1, \dots, X_n, \ \nu_n:=\sqrt{n}(\P_n-\P)$ and $\sigma(C):=\sqrt{\P(C)(1-\P(C))}$ for all $C\in \mathcal{C}.$ Then, for $\psi(t)=t\sqrt{\log (t^{-2})}$ and any null sequences $\gamma_n,\alpha_n$ satisfying $n\alpha_n \to \infty$ and $n^{-1}\log n \ll \gamma_n \leq \alpha_n,$ we have that 
\begin{align}
    \limsup_n \ \sup 
    \Big\{ \frac{|\nu_n(C)-\nu_n(D)|}{\psi(\sigma(C\triangle D))}:C,D\in \mathcal{C}, \P(C\triangle D)\leq \frac 12, \gamma_n\leq \sigma^2(C\triangle D)\leq \alpha_n\Big\}<\infty,
\end{align}
almost surely.
\end{theorem}
\begin{remark}
Note that in (4.3) of \cite{MR890285}, $n\alpha_n \downarrow$ should be $n\alpha_n \uparrow$ and in (4.4) of \cite{MR890285}, $\alpha_n$ has to be replaced by $\gamma_n,$ see also the proof of Theorem 4.1 on p. 417 of \cite{MR890285}. A consequence of the previous result is that because of $\sigma^2(C\triangle D) \geq \gamma_n \geq 1/n,$ for all sufficiently large $n,$ we have $\psi(\sigma(C\triangle D))\leq \sqrt{\P(C\triangle D)\log n}$ and hence
\begin{align}
    \limsup_n \ \sup 
    \Big\{ \frac{|\nu_n(C)-\nu_n(D)|}{\sqrt{\P(C\triangle D)\log n}}:C,D\in \mathcal{C}, \P(C\triangle D)\leq \frac 12, \gamma_n\leq \sigma^2(C\triangle D)\leq \alpha_n\Big\}<\infty,
    \label{eq.thm4.1.reform}
\end{align}
almost surely.
\end{remark}

\begin{proof}[\textbf{Proof of Theorem 8}]

It is enough to show that the statement holds for $\max_{n>1}$ replaced by $\limsup_n$ if we also can prove that for any finite $N,$
\begin{align}
    \max_{1<n\leq N} \, \sup_{x\in[0,1]} \, \sqrt{\log n} \, \Big| \frac{\wh t_n(x)}{t_n(x)}-1\Big|<\infty.
    \label{eq.thm_estima_to_show_1}
\end{align}
To see this observe that by definition $\sup_x \wh t_n(x)\leq 1.$ By definition, the spread function solves $t_n^2(x)\P_\sX([x\pm t_n(x)])=\log n/n.$ Combining this with the inequality $\P_\sX([x\pm t_n(x)])\leq 2\|p\|_\infty t_n(x)$ yields
\begin{align}
    \inf_x t_n(x)\geq \Big(\frac{\log n}{2n\|p\|_\infty}\Big)^{1/3}.
    \label{eq.tn_ub}
\end{align}
Therefore, 
\begin{align*}
    \max_{1<n\leq N} \, \sup_{x\in[0,1]} \, \sqrt{\log n} \, \Big| \frac{\wh t_n(x)}{t_n(x)}-1\Big|
    \leq \sqrt{\log N}\Big| \big(2N\|p\|_\infty \big)^{1/3}+1\Big|<\infty,
\end{align*}
proving \eqref{eq.thm_estima_to_show_1}. 

It thus remains to prove the statement with $\max_{n>1}$ replaced by $\limsup_{n}.$ The class of all half intervals $\{(-\infty,u]:u\in \mathbb{R}\}$ is VC. Thus, \eqref{eq.thm4.1.reform}, $\gamma_n=\log^2 n/n$ and $\alpha_n = n^{-1/4}$ applied to the class of half intervals gives that
\begin{align}
    \limsup_n \ \sup\bigg\{\sqrt{n}\frac{|\wh \P_\sX^n([a,b])-\P_\sX([a,b])|}{\sqrt{\P_\sX([a,b])\log n}}:0\leq a<b \leq 1,\frac{\log^2 n}{n}\leq \P_\sX([a,b]) \leq n^{-1/4}\bigg\}
    \label{eq.9631}
\end{align}
is almost surely finite. Since $\P_\sX$ is doubling, Lemma 3 and Remark \ref{rem.tnPX} ensure that $\P_\sX([x\pm t_n(x)])\geq \log^2 n/n$ for all $n \geq \exp(36(1\vee\log^2(D))),$ with $D$ the doubling constant of $\P_\sX$. Due to \eqref{eq.tn_ub} and the definition of the spread function, $\P_\sX([x\pm t_n(x)])=\log n/(nt_n(x)^2)\leq (\log n/n)^{1/3}\|p\|_\infty^{2/3}.$ Thus for all sufficiently large $n,$ we have $\log^2 n/n\leq \P_\sX([x\pm t_n(x)])\leq n^{-1/4}.$ Applying \eqref{eq.9631} shows that there exist $a_n(x)$ and a constant $A$ independent of $n,$ such that for any $x\in [0,1],$
\begin{align*}
    \P_\sX([x\pm t_n(x)])=\wh \P_\sX^n([x\pm t_n(x)])+a_n(x)\sqrt{\frac{\P_\sX([x\pm t_n(x)])\log n}{n}}
\end{align*}
and $\sup_x |a_n(x)|\leq A,$ almost surely. 

Suppose now that $t_n(x)> Q\wh t_n(x)$ with $Q:=(1+A/\sqrt{\log n}).$ By Lemma 3, Remark \ref{rem.tnPX}, and by picking $n$ large enough, $t_n(x)< 1/\sqrt{\log n}.$ Using the definitions of $t_n(x)$ and $\wh t_n(x),$ we have 
\begin{align*}
    \frac{\log n}{n}
    &=t_n(x)^2\P_\sX([x\pm t_n(x)]) \\
    &=
    t_n(x)^2\wh \P_\sX^n([x\pm t_n(x)])+a_n(x)t_n(x)\sqrt{\frac{t_n(x)^2\P_\sX([x\pm t_n(x)])\log n}{n}}\\
    &> Q^2 \wh t_n(x)^2\wh \P_\sX^n([x\pm \wh t_n(x)])
    -A\frac{\sqrt{\log n}}{n} \\
    &\geq \Big(Q^2-\frac{A}{\sqrt{\log n}}\Big) \frac{\log n}{n} \\
    &> \frac{\log n}{n},
\end{align*}
almost surely. This is a contradiction. Hence $t_n(x)\leq Q\wh t_n(x),$ almost surely.

Now let $t(x):= Rt_n(x)$ with $R:=(1-A/\sqrt{\log n})^{-1/2}>1.$ Arguing similarly as above, we find that 
\begin{align*}
    \frac{\log n}{n}
    &=t_n(x)^2\P_\sX([x\pm t_n(x)]) \\
    &=
    t_n(x)^2\wh \P_\sX^n([x\pm t_n(x)])+a_n(x)t_n(x)\sqrt{\frac{t_n(x)^2\P_\sX([x\pm t_n(x)])\log n}{n}}\\
    &\leq \frac {t(x)^2}{R^2} \wh \P_\sX^n([x\pm t(x)])
    +A \frac{\sqrt{\log n}}{n},
\end{align*}
almost surely. Rewriting this gives 
\begin{align*}
    t(x)^2\wh \P_\sX^n([x\pm t(x)])\geq R^2\Big(1-\frac{A}{\sqrt{\log n}}\Big)\frac{\log n}{n}= \frac{\log n}{n},
\end{align*}
almost surely. Applying the definition of the estimator $\wh t_n(x)$ yields $\wh t_n(x)\leq t(x)$ and thus $t_n(x)\geq R^{-1}\wh t_n(x),$ almost surely.

Combined with $t_n(x)\leq Q\wh t_n(x),$ this proves 
\begin{align*}
    \sup_x \, \Big|\frac{\wh t_n(x)}{t_n(x)}-1\Big| \leq \Big|1-\frac 1Q\Big|\vee |R-1|,
\end{align*}
almost surely. For any positive number $u, \ |1/\sqrt{u}-1|=|1-\sqrt{u}|/\sqrt{u}=|1-u|/(\sqrt{u}(1+\sqrt{u}))\leq |1-u|/\sqrt{u}.$ Without loss of generality, we can assume that $n$ is sufficiently large such that $\sqrt{\log n}\geq 2A.$ Choosing $u:=1-A/\sqrt{\log n},$ we then have $u\geq 1/2$ and $|R-1|=|1/\sqrt{u}-1|\leq \sqrt{2}A/\sqrt{\log n}.$ Since also $|1-1/Q|=|Q-1|/Q\leq |Q-1|=A/\sqrt{\log n},$ this proves that for any sufficiently large $n,$
\begin{align*}
    \sup_x \, \Big|\frac{\wh t_n(x)}{t_n(x)}-1\Big| \leq \Big|1-\frac 1Q\Big|\vee |R-1|\leq \frac{\sqrt{2}A}{\sqrt{\log n}},
\end{align*}
almost surely. The proof is complete.
\end{proof}
\section{Proofs for Section 3}
\label{annex:simple_distrib}

\begin{lemma}
\label{lem.local_doubling_nonzero_mass}
    If $\P_\sX \in \mP_n(D)$ for some $n > 0$ and $D \geq 2,$ then for any $x \in [0, 1]$ and any $t> 0,\ \P_\sX([x \pm t]) > 0$.
\end{lemma}
\begin{proof}
    Suppose that for some $x_0 \in [0, 1]$ and some $t>0, \  \P_\sX([x_0 \pm t]) = 0$. We prove that this implies the contradiction $1=\P_\sX([0, 1]) = 0.$ By considering sub-intervals, one can, without loss of generality, assume that $t\leq \sqrt{\log n}\sup_{x\in [0,1]} t_n(x).$ Therefore one can apply (LDP) in Definition 2 to obtain that $\P_\sX([x_0 \pm 2t]) \leq D\P_\sX([x_0 \pm t]) = 0.$ We have $[x_0 - t \pm t] \subset [x_0 \pm 2t]$ and $[x_0 +t \pm t] \subset [x_0 \pm 2t]$. Therefore $\P_\sX([x_0 -t \pm t]) = \P_\sX([x_0 + t \pm t]) = 0$. One can then repeat the previous step $\lceil 1/t \rceil$ times to obtain that 
    \begin{align*}
        1=\P_\sX([0, 1]) \leq \sum_{i = 0}^{\lceil t^{-1}\rceil} \P_\sX([x - it \pm t]) + \P_\sX([x + it \pm t]) = 0.
    \end{align*}
    Since this is a contradiction, the proof is complete.
\end{proof}

\begin{proof}[\textbf{Proof of Lemma 9}]
Let $u_n := \sqrt{\log n}\sup_{x \in [0, 1]}t_n(x)$. Since $\P_\sX$ admits a Lebesgue density, $x, t \mapsto \P_\sX([x \pm t])$ is continuous on $[0, 1]\times(0, 1]$. By virtue of Lemma \ref{lem.local_doubling_nonzero_mass}, $x \mapsto \P_\sX([x \pm t])$ is strictly positive on $[0, 1]\times(0, 1]$. Therefore $f\colon (x, \eta) \mapsto \P_\sX([x \pm 2\eta])/\P_\sX([x \pm \eta])$ is continuous on $[0, 1] \times (0, 1].$ In particular $f$ is continuous on the compact set $[0, 1]\times [u_n, 1].$ Therefore $f(x, \eta)$ is upper bounded on $[0, 1]\times [u_n, 1]$ by some constant $D'$ depending on $\P_\sX$. Also, for all $\eta \leq u_n, \ f(x, \eta) \leq D$ by assumption. Combining both, means that for all $(x, \eta) \in [0, 1]\times (0, 1], \ f(x, \eta) \leq D_n(\P_\sX):=D'\vee \sup_{\eta\leq u_n}\sup_{x\in [0,1]} f(x, \eta)\leq D'\vee D$ and hence $\P_\sX \in \mP_G(D_n(\P_\sX))$.
\end{proof}

The following lemma provides a sufficient condition for a distribution $\P_\sX \in \mM$ with monotone density $p$ to be doubling.
\begin{lemma}[Lemma 3.2 from \cite{cruz1996piecewise}]
    \label{lem.char_doubling_monotonic}
    Let $w$ be a locally integrable, monotonic function on $\RR_+$ and define $F\colon x \mapsto \int_0^{x}w(u)\, du.$ Denote by $\mu$ the measure $A \in \mB(\RR_+) \mapsto \int_A w(u) \, \dd u \in \RR_+.$
    If $w$ is increasing, then $\mu$ is doubling if and only if there exists a constant $\gamma \in (0, 1/2)$ such that for all $x,$ 
    \begin{align*}
        \gamma F(2x) \leq F(x).
    \end{align*}
    If $w$ is decreasing, then $\mu$ is doubling if and only if there exists a constant $\gamma \in (1, 2)$ such that, for all $x,$ 
    \begin{align*}
        \gamma F(x) \leq F(2x),
    \end{align*}
    where doubling means that there exists a constant $D' \geq 2$ such that for all $x >0$ and all $\eta > 0, \ \mu([x \pm 2\eta]) \leq D'\mu([x \pm \eta]).$
\end{lemma}

\begin{proof}[\textbf{Proof of Lemma 10}]
We first prove that $\P_\sX \in \mP_G(4\bar p/\underline p)$. To see that, for any $x \in [0, 1],$ consider $\eta > 0$. We have $(4\eta\wedge 1) \ol p \geq \P_\sX([x \pm 2\eta]) \geq \P_\sX([x \pm \eta]) \geq (\eta\wedge 1) \underline p.$ Therefore
\begin{align*}
    \frac{\P_\sX([x \pm 2\eta])}{\P_\sX([x \pm \eta])} \leq \frac{(4\eta\wedge 1) \bar p}{(\eta\wedge 1) \underline p} \leq 4\frac{\ol p}{\underline p}.
\end{align*}
We now prove (9). For all $n > 2$ and $x \in [0, 1],\ t_n(x) < 1$ and $t_n(x) \ul p \leq \P_\sX([x \pm t_n(x)]) \leq 2t_n(x)\ol p$. By using the definition of the spread function $t_n,$ we obtain
\begin{align*}
    t_n(x)^3 \ul p \leq \frac{\log n}n \leq 2t_n(x)^3 \ol p,
\end{align*}
which, once rearranged, yields the desired inequality. 
\end{proof}

\begin{proof}[\textbf{Proof of Lemma 11}]
To simplify the notation, we write $\P=\P_{\sX}.$ Let $\alpha > 0$ and $p \colon x \mapsto (\alpha + 1)x^\alpha\mathds{1}(x \in [0, 1])$.

The function $w \colon x \mapsto (\alpha + 1)x^\alpha$ is increasing on $(0,\infty)$ and for any $x>0,$
\begin{align*}
    \, \P([0, x])/\P([0, 2x]) =  1/2^{\alpha + 1} \in (0, 1/2).
\end{align*}
Therefore a measure $\mu$ with Lebesgue density $w$ is doubling in the sense of Lemma \ref{lem.char_doubling_monotonic}. Denote the doubling constant by $D.$ The restriction of $\mu$ to the interval $[0, 1]$ yields the distribution $\P.$ We show now that $\P \in \mP_G(D^3).$ To see this, we distinguish two cases, either $\eta > 1$ or $0< \eta \leq 1$. If $\eta > 1$ then $\P([x \pm 2\eta]) = 1 = \P_\sX([x \pm \eta]) \leq (D\vee 2)\P_\sX([x \pm \eta])$ which proves (LDP) of Definition 2 in this case. Else, if $\eta \in (0, 1],$ for any $x \in [0, 1],$ we have $\P_\sX([x \pm \eta]) = \mu([x \pm \eta]\cap[0, 1]) \geq D^{-2}\mu([x \pm \eta]).$ Consequently,
\begin{align*}
    \frac{\P_\sX([x \pm 2\eta])}{\P_\sX([x \pm \eta])} = \frac{\mu([x \pm 2\eta]\cap[0, 1])}{\mu([x \pm \eta]\cap[0, 1])} \leq D^2\frac{\mu([x \pm 2\eta])}{\mu([x \pm \eta])} \leq D^3
\end{align*}
and thus $\P \in \mP_G(D^3).$

By assumption $n \geq 9$. According to Lemma \ref{lem.tn_diff} and Remark \ref{rem.monotonicity_tn_borders}, we can split $[0, 1]$ into three intervals $I_1 := [0, a_n),\ I_2 := [a_n, b_n]$ and $I_3 := (b_n, 1]$ such that $t_n(a_n) = a_n, \ t_n(b_n) = 1-b_n,$ on $I_1,\ t_n(x) > x,$ on $I_2,\ t_n(x) \leq x\vee(1-x)$ and on $I_3,\ t_n(x) \geq 1-x.$ Moreover, from the expression for the derivative $t_n^\prime$ in Lemma \ref{lem.tn_diff} $(iii)$ and from the fact that $p$ is strictly increasing, we know that $t_n$ is strictly decreasing on $[0, b_n)$ and $t_n$ is strictly increasing on $(b_n, 1]$. We now derive expressions for the boundaries $a_n, b_n$.\\
\textit{\underline{Derivation of $a_n$}}: The solution $a_n$ of the equation $t_n(x) = x$ must satisfy $2^{\alpha+1}a_n^{\alpha + 3} = \frac{\log n}n,$ which can be rewritten as
\begin{align*}
    a_n = \Big(\frac{\log n}{2^{\alpha+1}n}\Big)^{1/(\alpha+3)}.
\end{align*}
\textit{\underline{Derivation of $b_n$}}: We solve the equation $t_n(b_n) = 1-b_n$ for $b_n \in [0, 1]$. Using the definition of the spread function, we obtain
\begin{align}
    \label{eq.derivation_b_n}
    t_n(b_n)^2\big(1 - (1 - 2t_n(b_n))^{\alpha + 1}\big) = \frac{\log n}n.
\end{align}
By Lemma \ref{lem.tn_diff} $(ii),$ we know that for $n>9, \ b_n>1/2$ and therefore $t_n(b_n)<1/2.$ Since $\alpha>0,$ Bernoulli's inequality guarantees $1 - 2(\alpha+1)t_n(b_n) \leq (1 - 2t_n(b_n))^{\alpha + 1}\leq 1 - 2t_n(b_n).$ Rewriting yields
\begin{align*}
    \Big(\frac{\log n}{2(\alpha + 1)n}\Big)^{1/3} \leq t_n(b_n) = 1-b_n \leq \Big(\frac{\log n}{2n}\Big)^{1/3}.
\end{align*}
We now derive the behaviour of $t_n$ in each of those regimes.\\
\underline{\textit{First regime, $t_n(x) \geq x$}}: Simple calculus leads to
\begin{align}
    t_n(0) &= \Big(\frac{\log n}n\Big)^{1/(\alpha + 3)} \ \text{ and} \  \ 
    t_n(a_n) = \Big(\frac{\log n}{2^{\alpha+1}n}\Big)^{1/(\alpha+3)}.
    \label{eq.1479}
\end{align}
As mentioned before, the spread function $t_n$ is strictly decreasing on $I_1$ and for all $x \in I_1,\ t_n(a_n) \leq t_n(x) \leq t_n(0).$ Hence,
\begin{align}
    \label{eq.bounds_first_regime}
    \Big(\frac{\log n}{2^{\alpha+1}n}\Big)^{1/(\alpha+3)} \leq t_n(x) \leq \Big(\frac{\log n}n\Big)^{1/(\alpha + 3)},
\end{align}
completing the proof of (10).\\
\textit{\underline{Second regime, $t_n(x) \leq x \wedge (1 - x)$}}: For $x \in I_2,$ we have 
\begin{align*}
    \P\big([x \pm t_n(x)]\big) &= \int_{x - t_n(x)}^{x + t_n(x)} (\alpha + 1)u^{\alpha}\, du.
\end{align*}
Using that in this regime $t_n(x) \leq x$ implies $\int_{x-t_n(x)}^{x+t_n(x)}u^{\alpha}\, d u \leq \int_{x-t_n(x)}^{x + t_n(x)}(x + t_n(x))^{\alpha} \, d u\leq 2t_n(x) (2x)^\alpha.$ Similarly, $\int_{x-t_n(x)}^{x+t_n(x)}u^{\alpha}\, d u\geq \int_x^{x+t_n(x)}x^{\alpha}\, d u=t_n(x) x^\alpha.$ Using the definition of $t_n$ and the previous inequalities, we get
\begin{align*}
    (\alpha+1)x^{\alpha}t_n(x)^3 \leq\frac{\log n}n &\leq 2^{\alpha + 1}(\alpha + 1)x^{\alpha}t_n(x)^3,
\end{align*}
which, once rearranged, yields, for all $x \in I_2,$
\begin{align}
    \label{eq.bounds_second_regime}
    \Big(\frac{\log n}{2^{\alpha+1}(\alpha+1)nx^{\alpha}}\Big)^{1/3} \leq t_n(x) \leq \Big(\frac{\log n}{(\alpha + 1)nx^{\alpha}}\Big)^{1/3}.
\end{align}
\underline{\textit{Third regime, $t_n(x) \geq 1- x$}}: As mentioned before, in this regime $t_n$ is strictly increasing. We already have suitable bounds for the value of $t_n(b_n)$ and now need an upper bound for the value of $t_n(1)$. We have $\P_\sX([1 - t_n(1), 1]) = 1 - (1-t_n(1))^{\alpha + 1}\geq 1-(1-t_n(1))=t_n(1)$. Therefore, $t_n(1)^3\leq \log n/n$ and $t_n(1)\leq (\log n/n)^{1/3}.$ Combining the bounds for $t_n(b_n)$ and $t_n(1)$ yields that for any $x \in I_3,$
\begin{align}
    \label{eq.bounds_third_regime}
    \Big(\frac{\log n}{2(\alpha + 1)n}\Big)^{1/3} \leq t_n(x) \leq \Big(\frac{\log n}{n}\Big)^{1/3}.
\end{align}
Using that $b_n>1/2$ and $2^\alpha x^\alpha \geq 1$ for $x\geq 1/2,$ the bounds in the second and third regime can be combined into
\begin{align*}
    \Big(\frac{\log n}{2^{\alpha+1}(\alpha+1)nx^{\alpha}}\Big)^{1/3} \leq t_n(x) \leq \Big(\frac{\log n}{nx^{\alpha}}\Big)^{1/3}  \ \ \text{for } a_n \leq x \leq 1.
\end{align*}
This proves (11).

Therefore, Theorem 4 applies.
\end{proof}

\section{Proofs for Section 5}
\label{supp.sec5}

\begin{proof}[\textbf{Proof of Lemma 12}]
By Theorem 4, $\int_0^1(\wh f_n(x) - f_0(x))^2q(x)\, d x \leq K\int_0^1t_n^{\P}(y)^2q(y) \, dy$ with probability tending to one as $n\to \infty$. To obtain the first inequality with $K''=4K^2,$ it is, therefore, enough to show that 
\begin{align}
    \int_0^1 t_n^{\P}(y)^2 q(y) \, dy
    \leq 4\int_0^1 t_n^{\P}(x) \Q\big([x\pm t_n^{\P}(x)]\big)\, dx.
    \label{eq.TL_2}
\end{align}
To verify \eqref{eq.TL_2}, observe that for any $y, \ t_n^{\P}(y)\leq 1$ and hence
\begin{align}
    \int_0^1 t_n^{\P}(y)^2 q(y) \, dy
    \leq
    2\int_0^1 \int_{[y\pm t_n^{\P}(y)/2]\cap [0,1]} \, dx \, t_n^{\P}(y) q(y) \, dy.
    \label{eq.TL_3}
\end{align}
By construction of the second integral, we have that $|x-y|\leq t_n^{\P}(y)/2.$ From Lemma \ref{lem.tn_diff}, we know that $t_n^{\P}$ is $1$-Lipschitz and, therefore, $|t_n^{\P}(x)-t_n^{\P}(y)|\leq |x-y|\leq t_n^{\P}(y)/2,$ implying $t_n^{\P}(y)/2\leq t_n^{\P}(x)\leq 3t_n^{\P}(y)/2$ so that $x \in [y \pm t_n^{\P}(y)/2]$ implies $y \in [x \pm t_n^{\P}(x)]$. This proves the inclusions
\begin{align*}
    I:=\Big\{(x,y): (x,y)\in \big[y\pm t_n^{\P}(y)/2\big]\cap [0,1] \times [0,1]\Big\}
    &\subseteq 
    \Big\{(x,y): (x,y)\in [y\pm t_n^{\P}(x)]\cap [0,1] \times [0,1]\Big\} \\
    &\subseteq 
    \Big\{(x,y): (x,y)\in [0,1] \times [x\pm t_n^{\P}(x)]\cap [0,1] \Big\}.
\end{align*}
Combined with \eqref{eq.TL_3} and $t_n^{\P}(y)/2\leq t_n^{\P}(x),$ it follows that
\begin{align*}
    \int_0^1 t_n^{\P}(y)^2 q(y) \, dy
    &\leq
    4\int_I t_n^{\P}(x) q(y) \, dy \, dx \\
    &=4\int_0^1 t_n^{\P}(x) \int_{[x\pm t_n^{\P}(x)]\cap [0,1]} q(y) \,dy \, dx \\
    &= 4\int_0^1 t_n^{\P}(x) \Q\big([x\pm t_n^{\P}(x)]\big) \, dx,
\end{align*}
completing the proof for \eqref{eq.TL_2} and therefore establishing the first inequality of the lemma. 

To prove the second inequality, observe that $$\frac 1{t_n^{\P}(x)}=\sqrt{\frac n{\log n} \P([x\pm t_n^{\P}(x)])}\leq \sqrt{2\frac n{\log n} t_n^{\P}(x) \|p\|_\infty},$$ which can be rewritten into $1/t_n^{\P}(x)\leq (2 \|p\|_\infty n/{\log n} )^{1/3}.$ 
\end{proof}

\begin{proof}[\textbf{Proof of Lemma 13}]

In Lemma 11, we already checked the conditions of Theorem 4.

We first prove the result for $\alpha<1$ using Lemma 12. For $\alpha \in (0,1),$ notice that since $t_n^{\P}(x)\leq 1,$ we have for any $x\leq 1/2,$ $\P([x\pm t_n^{\P}(x)])\geq \int_x^{x+t_n^{\P}(x)/2} (\alpha+1) u^\alpha \, du\geq x^\alpha t_n^{\P}(x)/2.$ Similarly, for any $x\geq 1/2,$ $\P([x\pm t_n^{\P}(x)])\geq \P([1/2\pm t_n^{\P}(x)])\geq (1/2)^\alpha t_n^{\P}(x)/2.$ Combining both cases gives the lower bound $\P([x\pm t_n^{\P}(x)])\geq 2^{-\alpha} x^\alpha t_n^{\P}(x)/2.$ Since also $\Q([x\pm t_n^{\P}(x)])\leq 2t_n^{\P}(x)$ and $\|p\|_\infty=(1+\alpha),$ Lemma 12 yields that, for $\alpha <1,$
\begin{align*}
    \int_0^1 \big( \wh f_n(x)-f_0(x)\big)^2 q(x) \, dx
    &\leq 2^{1/3} K''\Big(\frac{\log n}{n}\Big)^{2/3}\|p\|_\infty^{1/3} \int_0^1 \frac{\Q([x\pm t_n^{\P}(x)])}{\P([x\pm t_n^{\P}(x)])}  \, dx \\
    &\leq 2^{\alpha+7/3} K''\Big(\frac{\log n}{n}\Big)^{2/3} (1+\alpha)\int_0^1 x^{-\alpha} \, dx \\
    &\lesssim \Big(\frac{\log n}{n}\Big)^{2/3},
\end{align*}
with probability tending to one as $n\to \infty.$

Due to the infinite integral, the previous argument does not extend to the case $\alpha\geq 1.$ Using a refined argument based on the bounds for the spread function in Example 2, we now prove the claim for any $\alpha>0.$

Following Example 2, for any $\alpha > 0,$ let $a_n = (\log n/n2^{\alpha + 1})^{1/(\alpha + 3)} \asymp (\log n/n)^{1/(\alpha + 3)},$ and observe that
\begin{align*}
\begin{split}
    \int_0^1 \Big( \wh f_n(x)-f_0(x)\Big)^2 q(x) \, dx &\lesssim \int_0^1 t_n^{\P}(x)^2 \, d x\\
    & \lesssim a_n\Big(\frac{\log n}n\Big)^{2/(\alpha + 3)} + \int_{a_n}^{1} \Big(\frac{\log n}{nx^\alpha}\Big)^{2/3} \, d x\\
    & \lesssim \Big(\frac{\log n}n\Big)^{3/(\alpha + 3)} + \Big(\frac{\log n}n\Big)^{2/3}\underbrace{\int_{a_n}^1 x^{-2\alpha/3}\, dx}_{A_n},
\end{split}
\end{align*}
with probability tending to one as $n\to \infty.$ If $\alpha = 3/2,$ then $A_n = -\log(a_n) \asymp \log n$; if $\alpha \neq 3/2,$ then $A_n \asymp |1 - a_n^{1 - 2\alpha/3}|$. Combining the two cases yields the claim.
\end{proof}

\begin{proof}[\textbf{Proof of Theorem 14}]
Theorem 8 shows that there are numbers $N, M$ and an event $A$ that has probability one, such that on $A,$ we have $\max_{n\geq N,m\geq M}\sup_{x\in [0,1]}|\wh t_n^{\ \P}(x)/t_n^{\P}(x)-1|\vee |\wh t_m^{\ \Q}(x)/t_m^{\Q}(x)-1| \leq 1/2.$ Let us now always work on the event $A$. Note that
\begin{align*}
    \max_{n\geq N,m\geq M}\, \sup_{x\in [0,1]} \  \frac{\wh t_n^{\ \P}(x)}{t_n^{ \P}(x)}
    \vee 
    \frac{t_n^{\P}(x)}{\wh t_n^{\ \P}(x)}
    \vee 
    \frac{\wh t_m^{\ \Q}(x)}{t_m^{\Q}(x)}
    \vee 
    \frac{t_m^{\Q}(x)}{\wh t_m^{\ \Q}(x)}
    \leq 2.
\end{align*}
 In particular, we have that $\tfrac 12 (t_n^{\P}(x) \wedge t_m^{\Q}(x))\leq \wh t_n^{\ \P}(x) \wedge \wh t_m^{\ \Q}(x)\leq 2(t_n^{\P}(x) \wedge t_m^{\Q}(x)).$ By Theorem 4, there exists a $K'$ such that the event $A_{n,m}(f),$ defined through
 \begin{align*}
     \sup_x \frac{|\wh f_n^{(1)}(x) - f(x)|}{t_n^{\P}(x)}
     \vee \frac{|\wh f_m^{(2)}(x) - f(x)|}{t_m^{\Q}(x)}\leq K',
 \end{align*}
 satisfies $\sup_{f\in \Lip(1-\delta)} \P_{\!f}(A_{n,m}(f))\to 1$ as $n\to \infty$ and $m\to \infty.$ Thus on the event $A_{n,m}(f)\cap A,$
 \begin{align*}
     \frac{|\wh f_{n,m}(x) - f(x)|}{t_n^{\P}(x) \wedge t_m^{\Q}(x)}&\leq 2\frac{|\wh f_n^{(1)}(x) - f(x)|\mathds{1}_{\P, \Q}(x)+|\wh f_m^{(2)}(x) - f(x)|(1 - \mathds{1}_{\P, \Q}(x))}{\wh t_n^{\ \P}(x) \wedge \wh t_m^{\ \Q}(x)} \\
     &\leq 2K'\frac{t_n^{\P}(x)}{\wh t_n^{\ \P}(x)}+2K'\frac{t_m^{\Q}(x)}{\wh t_m^{\ \Q}(x)}\\
     &\leq 8K',
 \end{align*}
 where $\mathds{1}_{\P, \Q}(x) := \mathds{1}(\wh t_n^{\ \P}(x)\leq \wh t_m^{\ \Q}(x)).$ Taking $K=8K'$ shows then that
 \begin{align*}
      \sup_{f\in \Lip(1-\delta)} \ \P_{\!f} \left(\sup_{x \in [0, 1]} \frac{|\wh f_{n,m}(x) - f(x)|}{t_n^{\P}(x) \wedge t_m^{\Q}(x)} > K\right)
      \leq \sup_{f\in \Lip(1-\delta)} 
      \ \P_{\!f}(A_{n,m}(f)^c)+\ \P_{\!f}(A^c) \to 0,
 \end{align*}
 as $n\to \infty$ and $m\to \infty.$
\end{proof}

\begin{proof}[\textbf{Proof of Theorem 15}]
To shorten the notation, we suppress the dependence of $\P_\sX^n$ and $\Q_\sX^n$ on $n,$ write $N:=n+m$ for the size of the combined sample, and set $t_{n,m}(x):=t_n^{\P}(x)\wedge t_m^{\Q}(x)$. In a first step of the proof, we construct a number of disjoint intervals on which the mixture distribution $\ol \P =\tfrac nN\P_\sX+\tfrac mN\Q_\sX$ assigns sufficiently much mass. Let $\psi_N = (\log N/N)^{1/3}$ and $M_N=\lceil 1/(2\psi_N) \rceil.$ Furthermore, let $N_0$ be the smallest positive integer, such that for all $N\geq N_0,$
\begin{align}
    N^{1/4}\leq \frac{1}{6(2C_\infty-1)}\Big(\frac{N}{\log N}\Big)^{1/3}
    \quad \text{and} \ \ \psi_N\leq \frac 14.
    \label{eq.345678}
\end{align}
Clearly, $N_0$ only depends on $C_\infty.$ We now prove the theorem for all $N\geq N_0.$ The second constraint implies that $M_N\psi_N \leq (1/(2\psi_N)+1)\psi_N\leq 3/4.$

For $1 \leq k \leq M_N,$ consider the intervals $I_k := [(2k-1)\psi_N \pm \psi_N].$ By construction, $\cup_{k=1}^{M_N} I_k\supseteq [0,1].$ Define moreover $s_N := |\{k \in \{1,\ldots, M_N\}: \ol \P(I_k) \geq \psi_N\}|$ and $\alpha_N := s_N/M_N.$ We now derive a lower bound for the fraction $\alpha_N.$ Using $M_N\psi_N \leq 3/4$ in the last step, we have
\begin{align*}
    1 \leq \sum_{k = 1}^{M_N} \ol \P(I_k) &\leq \sum_{k:\, \ol \P(I_k) \leq \psi_N} \psi_N + \sum_{k:\, \ol \P(I_k) > \psi_N} \ol \P(I_k)\\
    &\leq (M_N - s_N)\psi_N + 2s_N C_\infty \psi_N\\
    &= (1-\alpha_N)M_N \psi_N + 2\alpha_N M_NC_\infty \psi_N \\
    &\leq \big(1+\alpha_N (2C_\infty-1)\big)\frac 34.
\end{align*}
Rewriting this, we find
\begin{align*}
    \alpha_N \geq \frac{1}{3(2C_\infty-1)}=: 2b.
\end{align*}
Thus, there exist at least $b(N/\log N)^{1/3}$ intervals $I_k$ such that $\ol \P(I_k) \geq \psi_N$. By a slight abuse of notation, denote by $k_1, \dots, k_{s_N}$ the indexes for which $\ol \P(I_k) \geq \psi_N.$ Together with \eqref{eq.345678}, $s_N\geq b(N/\log N)^{1/3}\geq N^{1/4}$ for all $N\geq N_0.$ Write $x_j:=(2k_j-1)\psi_N$ for the centre of the interval $I_{k_j}$ and observe that using the definition of the spread function,
\begin{align*}
 t_{n,m}(x_j)^2\ol \P(I_k)
 &\leq t_n^{\P}(x_j)^2 \frac{n}{N}\P\big([x_j\pm t_n^{\P}(x_j)]\big)
 + t_m^{\Q}(x_j)^2 \frac{n}{N}\Q\big([x_j\pm t_m^{\Q}(x_j)]\big) \\
 &\leq \frac{\log n+\log m}{N} \\
 &\leq 2\frac{\log N}{N}.
\end{align*}
If $t_{n,m}(x_j)> \sqrt{2}\psi_N,$ then, $t_{n,m}(x_j)\ol \P(I_k) > 2 \psi_N^3=2\log N/N,$ which is a contradiction. Therefore, $t_{n,m}(x_j)\leq \sqrt{2}\psi_N$ for all $j=1,\ldots,s_N.$

We now apply the multiple testing lower bound together with Theorem 2.5 from \cite{nonparest}. In order to do so, we construct $s_N+1 >0$ hypotheses $f_0, \dots, f_{s_N} \in \Lip(1)$ such that,
\begin{enumerate}
    \item[$(i)$] for all $i,j \in \{0, \dots, s_N\},$ with $i\neq j,$ we have $$\Big\|\frac{f_i - f_j}{t_{n,m}}\Big\|_{\infty} \geq \frac {1}6,$$
    \item[$(ii)$] $$\frac 1{s_N} \sum_{i=1}^{s_N} \KL(\P_{f_i}||\P_{f_0}) \leq \frac {\log s_N}9,$$
\end{enumerate}
where $\KL(\cdot ||\cdot)$ denotes the Kullback-Leibler divergence and $\P_{\!f}$ is the distribution of the data in the nonparametric regression model with covariate shift (19). Since the weighted sup-norm loss $\ell(f,g)=\|(f-g)/t_{n,m}\|_\infty$ defines a distance, the result follows from Theorem 2.5 in \cite{nonparest} once we have established $(i)$ and $(ii).$

Recall that $x_j:=(2k_j-1)\psi_N$ are the centers of the $s_N$ intervals $I_{k_1},\ldots,I_{k_{s_N}}.$ Set $f_0 = 0$ and for any $i = 1,\dots, s_N$, consider the function $f_i \colon x\mapsto (t_{n,m}(x_i)/6 - |x - x_i|)_+,$ with $(z)_+:=\max(z,0).$ By construction, $f_i \in \Lip(1).$ As shown above, $t_{n,m}(x_i)\leq 2\psi_N.$ This implies that the support of $f_i$ is contained in the interval $[x_i - t_{n,m}(x_i)/6, x_i + t_{n,m}(x_i)/6] \subset I_{k_i}.$ Thus, for different $i,j,$ the functions $f_i$ and $f_j$ have disjoint support. Since $\|f_i\|_\infty =  t_{n,m}(x_i)/6,$ we obtain for $i\neq j,$
\begin{align*}
    \Big\|\frac{f_i - f_j}{t_{n,m}}\Big\|_\infty \geq \frac{f_i(x_i)}{t_{n,m}(x_i)} = \frac{1}{6},
\end{align*}
and this proves $(i).$

To verify $(ii),$ we first show the known formula 
\begin{align}
    \KL(\P_{\!f},\P_g)=\frac n{2} E_{X\sim \P}[(f(X)-g(X))^2]
    +\frac m{2} E_{X\sim \Q}[(f(X)-g(X))^2].
    \label{eq.KL_basic_id}
\end{align}
that holds for all $L^2[0,1]$-functions $f,g.$ To see this, observe that for any function $h,$ 
\begin{align*}
    \P_h=\bigotimes_{i=1}^n \P_{h,i}\otimes \bigotimes_{j=1}^m \Q_{h,j},    
\end{align*}
 where $\P_{h,i}$ denotes the distribution of the $i$-th observation $Y_i=h(X_i)+\varepsilon_i$ with $X_i\sim \P$ and $Q_{h,j}$ denotes the $(j+n)$-th observation $Y_{j+n}=h(X_{j+n})+\varepsilon_{j+n}$ with $X_{j+n}\sim \Q.$ If $R_1, S_1$ are two probability measures defined on the same measurable space and $R_2, S_2$ are two probability measures defined on the same measurable space, then $\KL(R_1\otimes R_2, S_1 \otimes S_2)=\KL(R_1, S_1)+\KL(R_2, S_2),$ see e.g.\ $(iii)$ on p.85 in \cite{nonparest}. Therefore, $\KL(\P_{\!f},\P_g)=\sum_{i=1}^n \KL(\P_{f,i},\P_{g,i})+\sum_{j=1}^m \KL(\Q_{f,j},\Q_{g,j}).$ If $\P_{h,i|X_i}$ denotes the conditional distribution of $Y_i=h(X_i)+\varepsilon_i$ given $X_i,$ it follows from the definition of the Kullback-Leibler divergence that $\KL(\P_{f,i},\P_{g,i})=E_{X\sim \P}[\KL(\P_{f,i|X_i},\P_{g,i|X_i})].$ By using that $\varepsilon_i \sim \mN(0, 1),$ some elementary computations show moreover that $\KL(\P_{f,i|X_i},\P_{g,i|X_i})=(f(X_i)-g(X_i))^2/2,$ see also Corollary 2.1 in \cite{nonparest}. By following the same arguments, $\KL(\Q_{f,j},\Q_{g,j})=E_{X\sim \Q}[(f(X)-g(X))^2]/2.$ Putting everything together proves the identity \eqref{eq.KL_basic_id}.

Now, we use that $f_0=0, \ \|f_j\|_\infty \leq t_{n,m}(x_j)/4,$ $t_{n,m}(x)=t_n^{\P}(x)\wedge t_m^{\Q}(x),$ and the definition of the spread function, to obtain for $j = 1, \dots, M_N,$
\begin{align*}
    \KL(\P_{f_j}||\P_{f_0}) &= \frac n {2}\E_{X\sim \P}\big[f_j(X)^2\big]+\frac m {2}\E_{X\sim \Q}\big[f_j(X)^2\big]\\
    &\leq \frac n{2}\int_{x_j - t_n^{\P}(x_j)}^{x_j + t_n^{ \P}(x_j)}f_j(t)^2p(t) \, \, d t
    + \frac m{2}\int_{x_j - t_m^{\Q}(x_j)}^{x_j + t_m^{\Q}(x_j)}f_j(t)^2p(t) \, \, d t\\
    &\leq \frac{n}{72} t_n^{\P}(x_j)^{2}\P\big([x_j \pm t_n^{\P}(x_j)]\big)+\frac{m}{72} t_m^{\Q}(x_j)^{2}\P\big([x_j \pm t_m^{\Q}(x_j)]\big)\\
    &\leq \frac{\log n+\log m}{72} \\
    &\leq \frac{\log N}{36}.
\end{align*}
Recall that $s_N \geq N^{1/4}.$ Averaging over the $s_N$ hypotheses and using that $\log s_N\geq \tfrac 14 \log N,$ we find
\begin{align*}
    \frac 1{s_N} \sum_{j=1}^{s_N} \KL(\P_{f_j}||\P_{f_0}) \leq \frac{\log N}{36}\leq \frac {\log s_N}9.
\end{align*}
Therefore, the conditions of Theorem 2.5 in \cite{nonparest} are satisfied with $M=s_N,$ $s=1/12,$ and $d(\wh \theta,\theta)=\|(\wh f-f)/t_{n,m}\|_\infty.$ From the conclusion of this theorem, the claim follows. 
\end{proof}

\begin{proof}[\textbf{Proof of Theorem 7}]
The theorem can be viewed as a special case of Theorem 15 by taking $m=0$ and $t_m^{\Q}(x)=\infty,$ such that $t_{n,m}(x)=t_n^{\P}(x)$ for all $x.$
\end{proof}

\begin{proof}[\textbf{Proof of Lemma 16}]
We write $\P=\P_{\sX}$ and $\Q=\Q_{\sX}.$ Suppose that 
\begin{align*}
    t_{m+n}^{\wt \P}(x) > t_m^{\Q}(x)\sqrt{\frac{\log(m+n)}{\log m}} > t_m^{\Q}(x),
\end{align*} 
then
\begin{align*}
    \frac{\log(n+m)}{\log m}\frac{\log m}{m} &= \frac{\log(n+m)}{\log m}t_m^{\Q}(x)^2 \Q([x \pm t_n^{\Q}(x)])\\
    &< \frac{n + m}{m}\bigg[\frac{t_{m+n}^{\wt \P}(x)^2}{n+m}\Big( n\P\big([x \pm t_{m+n}^{\wt \P}(x)]\big) + m\Q\big([x \pm t_{n+m}^{\wt \P}(x)]\big)\Big)\bigg] \\
    &= \frac{\log(n+m)}{m}
\end{align*}
which is a contradiction. Therefore $t_{m+n}^{\wt \P}(x) \leq t_m^{\Q}(x)\sqrt{\frac{\log(m+n)}{\log m}}.$ By symmetry, one concludes that $t_{m+n}^{\wt \P}(x) \leq \sqrt{\tfrac{\log(m+n)}{\log n}}t_n^{\P}(x)$. This proves $t_{n + m}^{\wt \P}(x) \leq t_m^{\Q}(x)\sqrt{\tfrac{\log(m+n)}{\log m}}\wedge t_n^{\P}(x)\sqrt{\tfrac{\log(m+n)}{\log n}}$.

To prove the second part of the claim, we show that there are positive constants $C''$ and $M,$ such that $t_{m+n}^{\wt P}(x)\leq C''(t_n^{\P}(x)\wedge t_m^{\Q}(x))$ for all $n\geq m \geq M.$  Since $\log(n),\log(m)>1$ for $n,m>2,$ the result holds then with $C'=C''\vee \sqrt{\log(2M)}$ for all $n\geq m >2.$ 

Observe that if $t_m^{\Q}(x)\sqrt{\log(m+n)/\log m}\leq t_n^{\P}(x)\sqrt{\log(m+n)/\log n},$ then,
\begin{align*}
    \sqrt{\frac{\log m}{m}}
    \leq t_m^{\Q}(x)
    \leq \sqrt{\frac{\log m}{\log n}}
    t_n^{\P}(x)
    \leq 
    \sqrt{\frac{\log m}{\log n}} Cn^{-\kappa},
\end{align*}
where the first inequality follows from the definition of the spread function, which implies $t_m^{\Q}(x)^2\geq t_m^{\Q}(x)^2\Q([x\pm t_m^{\Q}(x)])=\log m/m.$ The last display can be rewritten into the form $n^{2\kappa}\log n\leq Cm,$ implying that $\log m=\log\log n +2\kappa \log n-\log C.$ Therefore, we have that, for sufficiently large $M,C'',$ and all $n\geq m\geq M,$ $\sqrt{\log(n+m)/\log(m)} \vee \sqrt{\log(n+m)/\log(m)}\leq \sqrt{\log(2n)/\log(m)}\leq C''.$ Together with the first inequality of the lemma, we then have 
\begin{align*}
    t_{n + m}^{\wt \P}(x) 
    \leq t_m^{\Q}(x)\sqrt{\tfrac{\log(m+n)}{\log m}}\wedge t_n^{\P}(x)\sqrt{\tfrac{\log(m+n)}{\log n}}
    \leq C''
    \big(t_m^{\Q}(x) \wedge t_n^{\P}(x)\big),
\end{align*}
for all $n\geq m \geq M.$ 
\end{proof}

\begin{proof}[\textbf{Proof of Lemma 17}] 
Using Theorem 14, we have that 
\begin{align*}
        \int_0^1 \big(\wh f_{n,m}(x)-f_0(x)\big)^2 \, dx\lesssim \int_0^1 \big(t_n^{\P}(x) \wedge t_m^{\Q}(x)\big)^2 \, dx
\end{align*}
with probability tending to one as $n$ and $m$ tend to infinity. Let $a_n=(\log n/n2^{\alpha + 1})^{1/(\alpha + 3)}.$ Rewriting the condition $n^{3/(3+\alpha)}\log^{\alpha/(3+\alpha)} n \ll m$ shows that $(m/n)^{1/\alpha}>a_n$ for all sufficiently large $n.$ Using Lemma 10, Lemma 11 and $m\leq n,$ we find that
\begin{align*}
    \int_0^1 \big(t_n^{\P}(x) \wedge t_m^{\Q}(x)\big)^2 \, dx
    &\lesssim 
    \int_0^{a_n} \!\!\Big(\frac{\log n}{n}\Big)^{\frac{2}{\alpha+3}}\!\!\wedge  \Big(\frac{\log m}{m}\Big)^{\frac 23} \, dx
    +
    \!\!\int_{a_n}^1 \!\Big(\frac{\log n}{n x^\alpha}\Big)^{\frac{2}{3}}\!\wedge  \Big(\frac{\log m}{m}\Big)^{\frac 23} \, dx \\
    &\leq 
    \Big(\frac mn\Big)^{1/\alpha}
    \Big(\frac{\log m}{m}\Big)^{\frac 23}
    +\int_{(m/n)^{1/\alpha}}^{1} 
    \Big(\frac{\log n}{n x^\alpha}\Big)^{\frac{2}{3}} \, dx\\
    &\leq 
    \Big(\frac mn\Big)^{1/\alpha}
    \Big(\frac{\log m}{m}\Big)^{\frac 23}
    +\Big(\frac{\log n}{n}\Big)^{2/3}\frac{1}{2\alpha/3-1}\Big(\frac{m}{n}\Big)^{1/\alpha-2/3}\\
    &\leq \Big(1+\frac{1}{2\alpha/3-1}\Big)
    \Big(\frac mn\Big)^{1/\alpha}
    \Big(\frac{\log n}{m}\Big)^{\frac 23}.
\end{align*}
\end{proof}

\section{Proofs for Section 6}
\label{app:adapted_wright}
\begin{proof}[\textbf{Proof of Theorem 18}] Set $T_n(x_0) := (\tfrac{\alpha + 1}{A})^{1/(2\alpha + 1)} (np(x_0))^{\alpha/(2\alpha + 1)}(\wh f_n(x_0) - f(x_0))$. Let $\varepsilon_n$ be such that $n^{-1/2} \ll \varepsilon_n \ll n^{-1/(2\alpha + 1)}$. If $F$ denotes the c.d.f.\ of the design distribution and $F_n$ denotes the empirical c.d.f.\ $F_n(x)=\tfrac 1n\sum_{i=1}^n \mathds{1}(X_i\leq t).$ We have that
\begin{align*}
    \P\Big(T_n(x_0) \leq x\Big) = &\P\Big(T_n(x_0)\leq x \cap \sup_{y \in \RR} |F_n(y) - F(y)| \geq \varepsilon_n\Big)\\
    &+ \P\Big(T_n(x_0) \leq x \cap \sup_{y \in \RR} |F_n(y) - F(y)| < \varepsilon_n\Big).
\end{align*}
Using the Dvoretzky-Kiefer-Wolfowitz (DKW) inequality \cite{MR1062069} yields
\begin{align}
\label{eq.DKW}
\begin{split}
    \P\Big(T_n(x_0) \leq x \cap \sup_{y \in \RR} |F_n(y) - F(y)| \geq \varepsilon_n\Big) &\leq \P\Big(\sup_{y \in \RR} |F_n(y) - F(y)| \geq \varepsilon_n\Big)\\
    &\leq 2e^{-2n\varepsilon_n^2} \underset{n \to \infty}{\longrightarrow} 0.
\end{split}    
\end{align}
We now check that all the assumptions of Theorem 1 in \cite{wright_isotonic_1981} are verified on $\{\sup_{y \in \RR} |F_n(y) - F(y)| < \varepsilon_n\}$. Since $\varepsilon_n \ll n^{-1/(2\alpha + 1)},$ we have that $\sup_{y \in \RR}|F_n(y) - F(y)| \ll n^{-1/(2\alpha + 1)}$. By definition of the c.d.f., we also have $F'(x_0) = p(x_0) > 0.$ Applying the setting of \cite{wright_isotonic_1981} to our case gives $\phi = 1$ while $w \equiv 1$ and $Y_i - f_0(X_i) = \varepsilon_i \sim N(0, 1)$ for all $1 \leq i \leq n$. Therefore, for all $0 \leq j_n < k_n \leq n, \ ((k_n - j_n))^{-1/2}\sum_{i = j_n}^{k_n} (Y_i - f_0(X_i)) \sim \mN(0, 1).$ Finally, let $Z$ be a random variable with Chernoff distribution. By applying Theorem 1 of \cite{wright_isotonic_1981}, we obtain
\begin{align*}
    \P\Big(T_n(x_0) \leq x \big| \sup_{y \in \RR} |F_n(y) - F(y)| < \varepsilon_n\Big) \underset{n \to \infty}{\longrightarrow} \P(Z \leq x).
\end{align*}
Hence, applying DKW again yields $\P(\sup_{y \in \RR} |F_n(y) - F(y)| < \varepsilon_n) \geq 1 - 2\exp(-2n\varepsilon_n^2) \to 1$ as $n \to \infty,$ and
\begin{align}
    \label{eq.DKW_wright}
    \P\Big(T_n(x_0) \leq x \cap \sup_{y \in \RR} |F_n(y) - F(y)| < \varepsilon_n\Big) \underset{n \to \infty}{\longrightarrow} \P(Z \leq x).
\end{align}
Combining inequalities \eqref{eq.DKW} and \eqref{eq.DKW_wright} gives $T_n(x_0) \overset{d}{\longrightarrow} Z$.
\end{proof}

Recall that $\rho_\eta(\P_\sX, \Q_\sX) := \int_0^1 \P_\sX([x \pm \eta])^{-1}\, d\Q_\sX(x),$ and that for any $\gamma\geq 1,$ and any $C > 0$, the class $\mS$ is defined as the set of all pairs $(\P_\sX, \Q_\sX),$ such that $\sup_{\eta \in (0, 1]} \eta^\gamma \rho_\eta(\P_\sX, \Q_\sX) \leq C$.

\begin{lemma}
\label{lem.increasing_distrib_in_specific_space}
    If $\alpha \geq 1, \ \P_\sX$ is the distribution with Lebesgue density $p \colon x \mapsto (\alpha + 1)x^{\alpha}\mathds{1}(x \in [0, 1]),$ and $\Q_\sX$ is the uniform distribution on $[0, 1],$ then there exists a constant $0 < C <\infty,$ such that for any $\varepsilon \in (0, \alpha),$ $$(\P_\sX, \Q_\sX) \in \mS(\alpha, C)\setminus \mS(\alpha - \varepsilon, C).$$
\end{lemma}

\begin{proof}
    Consider the function $f\colon \eta \mapsto \int_0^1 \P_\sX([x \pm \eta])^{-1}\, dx$. It can be checked that $f$ is decreasing and continuous. Consequently, to prove that $\sup_{\eta \in (0, 1]} \eta^\alpha f(\eta)$ is bounded, it is sufficient to prove that $\lim_{\eta \to 0} \eta^{\alpha}f(\eta) < \infty.$ For $0 < \eta < 1/10,$
    \begin{align*}
        f(\eta) = \int_0^\eta \P_\sX([x \pm \eta])^{-1} \, dx + \int_\eta^{1-\eta} \P_\sX([x \pm \eta])^{-1}\, dx + \int_{1-\eta}^1 \P_\sX([x \pm \eta])^{-1} \, dx.
    \end{align*}
    We now bound the three terms above. We have
    \begin{align*}
        \int_0^\eta \P_\sX([x \pm \eta])^{-1} \ dx &= \int_0^\eta \frac 1{(x + \eta)^{\alpha + 1}}\, dx \leq \frac 1{\alpha \eta^{\alpha}},\\
        \int_{1-\eta}^1 \P_\sX([x \pm \eta])^{-1}\, dx &= \int_{1- \eta}^1 \frac 1{1 - (x - \eta)^{\alpha + 1}}\, dx \leq \int_{1-\eta}^1 \frac 1{1 - x + \eta} \, dx = \log(2).
    \end{align*}
    For $x \in (\eta, 1- \eta),$ we have, by Bernoulli's inequality and the fact that $\alpha + 1 > 1,\ \P_\sX([x \pm \eta]) = x^{\alpha + 1}[(1 + \eta/x)^{\alpha + 1} - (1 - \eta/x)^{\alpha + 1}] \geq x^{\alpha}(\alpha + 2)\eta.$ Therefore,
    \begin{align*}
        \int_\eta^{1- \eta} \P_\sX([x \pm \eta])^{-1}\, dx \leq 
        \begin{cases}
        \frac1{(\alpha+2)(\alpha - 1)\eta^\alpha} &\text{if $\alpha > 1$},\\
        \frac1{(\alpha+2)\eta}\log\big(\frac 1{\eta}\big) &\text{if $\alpha = 1,$ and}\\
        \frac1{(\alpha+2)\eta} &\text{if $\alpha < 1$.}
        \end{cases}
    \end{align*}
    In any case, $\eta^{\alpha}f(\eta) \leq C$ for some $C > 0$. This proves 
    that $\sup_{\eta \in (0, 1]} \eta^\alpha f(\eta)$ is bounded and $(\P_\sX, \Q_\sX) \in \mS(\alpha, C).$ 
    
    For $0 < \varepsilon < \alpha$ and $0 < \eta \leq 1/2,$
    \begin{align*}
        \eta^\varepsilon f(\eta)=\int_0^1 \frac{\eta^{\varepsilon}}{\P_\sX([x \pm \eta])}\, dx \geq \int_0^\eta \frac{\eta^{\varepsilon}}{\P_\sX([0,x +\eta])} \, dx = \frac 1{2^{\alpha}\alpha} \eta^{\varepsilon - \alpha} \underset{\eta \to 0}{\longrightarrow} \infty,
    \end{align*}
    proving that if $\varepsilon \in (0, \alpha),$ then $(\P_\sX, \Q_\sX) \in \mS(\alpha, C) \setminus \mS(\alpha - \varepsilon, C).$ This concludes the proof.
\end{proof}

\begin{proof}[Proof of Lemma 19]
In a first step of the proof, we show that there exists a LSE over the $(1-\delta)$-Lipschitz functions that are piecewise linear with at most $n-1$ linear pieces. Suppose this is wrong and let $\wh f_n$ be a LSE over all $\Lip(1-\delta)$ functions. Write $\wt f_n$ for the function interpolating the values $(X_i,\wh f_n(X_i)).$ This is a piecewise linear function with at most $n-1$ pieces, and the Lipschitz constant is not larger than the Lipschitz constant of $\wh f_n$ implying $\wt f_n \in \Lip(1-\delta).$ By construction, $\wt f_n$ has the same loss as the LSE $\wh f_n$. Consequently, $\wt f_n$ is also a LSE. This contradiction proves that a piecewise linear LSE $f_n^L$ with $n-1$ pieces exists.

A piecewise linear function with $m$ pieces can be exactly represented by a shallow network $f(x)=\sum_{i=1}^N a_i(w_i x-v_i)_+,$ whenever $N\geq m.$ This proves that the piecewise linear LSE $\wh f_n^L$ is in $\ReLU_N(1-\delta)$ for $N\geq n-1.$ Let us introduce the shorthand notation $L(f):=\sum_{i=1}^n \big(Y_i-f(X_i)\big)^2.$ Since, by construction, $\ReLU_N(1-\delta)\subset \Lip(1-\delta),$ we have, for any minimizer $\wh f_n\in \argmin_{f\in\ReLU_N(1-\delta)} L(f),$ that
\begin{align*}
    \min_{f\in\Lip(1-\delta)} L(f)
    \leq L(\wh f_n)\leq L(\wh f_n^L)=\min_{f\in\Lip(1-\delta)} L(f).
\end{align*}
Thus, all $\leq$ can be replaced by $=$ in the last display. This proves that $\wh f_n$ is a LSE for the function class $\Lip(1-\delta).$
\end{proof}
\end{appendix}

\end{document}